\documentclass{article}
\usepackage[utf8]{inputenc}
\usepackage{amsthm}
\usepackage{amsmath}
\usepackage{amsfonts}
\usepackage{latexsym}
\usepackage{geometry}
\usepackage{enumerate}
\usepackage{csquotes}
\usepackage{graphicx}
\usepackage{comment}
\usepackage{appendix}
\usepackage{hyperref}

\emergencystretch=\maxdimen
\def \W {\mathcal{W}}
\def \N {\mathcal{N}}
\def \PP {\mathcal{P}}
\def \XX {\mathcal{X}}
\def \bW {\overline{\mathcal{W}}}
\def \dist {{\rm dist}}
\DeclareMathOperator{\Rm}{Rm}
\DeclareMathOperator{\Ric}{Ric}
\DeclareMathOperator{\Vol}{Vol}
\DeclareMathOperator{\spt}{spt}
\makeatletter
\newcommand*{\rom}[1]{\rm {\expandafter\@slowromancap\romannumeral #1@}}
\makeatother

\title{Ancient Ricci flows with asymptotic solitons}
\author{Pak-Yeung Chan, Zilu Ma, and Yongjia Zhang}
\numberwithin{equation}{section}
\newtheorem{Theorem}{Theorem}[section]
\newtheorem{Proposition}[Theorem]{Proposition}
\newtheorem{Lemma}[Theorem]{Lemma}
\newtheorem{Corollary}[Theorem]{Corollary}
\theoremstyle{definition}
\newtheorem{Definition}[Theorem]{Definition}

\begin{document}

\maketitle

\begin{abstract}
We study ancient Ricci flows which admit asymptotic solitons in the sense of Perelman \cite[Proposition 11.2]{Per02}. We prove that the asymptotic solitons must coincide with Bamler's tangent flows at infinity \cite{Bam20c}. Furthermore, we show that Perelman's $\nu$-functional is uniformly bounded on such ancient solutions; this fact leads to logarithmic Sobolev inequalities and Sobolev inequalities. Lastly, as an important tool for the proofs of the above results, we also show that, for a complete Ricci flow with bounded curvature, the bound of the Nash entropy depends only on the local geometry around an $H_n$-center of its base point.
\end{abstract}

\setcounter{tocdepth}{1}
\tableofcontents

\section{Introduction}

Through Hamilton's program \cite{Ham93b}, the study of ancient solutions emerged into the field. Perelman's seminal work \cite{Per02} manifests not only that the understanding of ancient solutions is crucial to the analysis of the singularity formation, but also that ancient solutions are interesting in their own right. The ancient solution can be approached from different directions: one is the classification (c.f. \cite{DHS12}, \cite{Bre20}, \cite{Li20}, etc.), which, in the higher-dimensional non-K\"ahler case, is impossible without any strong curvature assumption; another is the construction of examples (\cite{BKN12}, \cite{Lau13}, \cite{LW16}, etc.), whose immense quantity with their diverse behaviors discloses the difficulty of the whole field; a third approach is the least specific---to study the geometric and analytic properties of ancient solutions, among which Perelman's monotonicity formulas are very attracting ones. Works in this category include \cite{Y09}, \cite{N10}, \cite{CZ11}, \cite{Zhang21}, to list but a few. 

Bamler's recent ground-breaking works \cite{Bam20a}---\cite{Bam20c} have shed more light on the field of Ricci flow. The estimates therein within the smooth category can be regarded as some good improvements of Perelman's entropy techniques, whereas the weak version of the Ricci flow (called the \emph{metric flow}) provides an intriguing object for the future study. While the latter is fascinating and promising, it is yet beyond the scope of the present article, and we shall focus on the former, just as the last two authors have done in \cite{MZ21}.

As is well-known now, for ancient solutions under certain curvature conditions, Perelman's monotonicity formulas \cite{Per02} imply the existence of an asymptotic shrinker (c.f. \cite[Proposition 11.2]{Per02} and \cite{N10}). On the other hand, Bamler \cite{Bam20c} has proved that an ancient solution has a tangent flow at infinity, and that, if we further assume that the Nash entropy is uniformly bounded and that the underlying manifold is closed, then the tangent flow is a metric soliton with a singular set of co-dimension 4 (see \cite[Section 1.8]{Bam20c}; c.f. \cite[Theorem 7.8]{Bam20c}). A natural question arises that how Perelman's asymptotic shrinker and Bamler's tangent flow at infinity are related. Bamler himself remarks that, should the ancient solution be a $\kappa$-solution, they two would then coincide. But what about other cases? What is the weakest possible condition for this to be true? These are the questions that motivated the current study. In fact, one of our results is that the asymptotic shrinker and the tangent flow at infinity coincide if the former exists. We do not assume any other curvature condition for the ancient solution except for its time-wise boundedness.

Let $(M^n,g(t))_{t\in(-\infty,0]}$ be an ancient Ricci flow. Throughout this article, we shall make a technical assumption that $g(t)$ has bounded curvature within each compact time interval, namely,
\begin{eqnarray}\label{curvaturebound}
\sup_{M\times[t_1,t_2]}\big|{\Rm}_{g_t}\big|<\infty \quad \text{ for all }\quad -\infty<t_1\leq t_2\leq 0.
\end{eqnarray}
Note that the curvature bound may depend on the interval $[t_1,t_2]$, and is not uniform on $M\times(-\infty,0]$. This assumption is only for the validity of the classical formulas (especially the integration by parts at infinity) concerning conjugate heat kernels and entropies. In particular, we do not assume the ancient solution in question to be a $\kappa$-solution or to be Type I. 

Let $p_0\in M$ be a fixed point. Then, for a sequence $\{\tau_i\}_{i=1}^\infty$ with $\tau_i\nearrow\infty$, we may find $\{p_i\}_{i=1}^\infty$ such that $(p_i,-\tau_i)$ are $\ell$-centers of $(p_0,0)$, namely, $\ell_{p_0,0}(p_i,\tau_i)\leq\frac{n}{2}$; see section 2.3 for more details of the definitions. 
For most of the time in this article, we make the following assumption for the ancient solutions which we consider.
\\

\noindent\textbf{Assumption B:} For the fixed point $(p_0,0)$ and the sequences $\{\tau_i\}_{i=1}^\infty$ and $\{p_i\}_{i=1}^\infty$ as described above, there exists a smooth and complete Ricci flow $\big(M_\infty,g_\infty(t),p_\infty\big)_{t\in[-2,-1]}$, such that
\begin{eqnarray}\label{smoothconvergence}
\big(M,g_i(t),p_i\big)_{t\in[-2,-1]}\xrightarrow{\makebox[1cm]{}} \big(M_\infty,g_\infty(t),p_\infty\big)_{t\in[-2,-1]}
\end{eqnarray}
in the smooth Cheeger-Gromov-Hamilton sense, where the Ricci flow $g_i(t)$ is obtained by the Type I scaling
\begin{eqnarray}\label{TypeIscale}
g_i(t):=\tau_i^{-1}g(\tau_it).
\end{eqnarray}
\\

The statement of Assumption B is involved with a base point $(p_0,0)$, a sequence of positive scales $\{\tau_i\}_{i=1}^\infty$, and the choices of $\ell$-centers $(p_i,-\tau_i)$. Hence, if necessary, we shall refer to an ancient solution as ``satisfying Assumption B with respect to $\big(p_0,0,\{\tau_i\}_{i=1}^\infty,\{p_i\}_{i=1}^\infty\big)$''. For the limit Ricci flow $(M_\infty,g_\infty(t))$ in (\ref{smoothconvergence}), we do not assume any shrinker structure, neither do we make any geometric assumption except for its completeness. However, it will soon be clear that, because of \cite{CZ20}, this limit is Perelman's asymptotic shrinker; see the statement of Theorem \ref{Theorem_main}(1) below. 


We are now ready to state the first of our main theorems.

\begin{Theorem}\label{Theorem_main}
Under Assumption B, the following hold (see section 2 for all the definitions involved).
\begin{enumerate}[(1)]
\item $\big(M_\infty,g_\infty(t),p_\infty\big)_{t\in[-2,-1]}$ admits a shrinker structure, which makes it an asymptotic shrinker in the sense of Perelman (\cite[Proposition 11.2]{Per02}).
\item We have
\begin{eqnarray*}
\lim_{\tau\rightarrow\infty}\mathcal{N}_{p_0,0}(\tau)=\mu_\infty,
\end{eqnarray*}
where $\mathcal{N}_{p_0,0}$ is the Nash entropy based at $(p_0,0)$ and $\mu_\infty$ is the entropy of the asymptotic shrinker. In particular, $\mu_\infty$ is the infimum of $\mathcal{N}_{p_0,0}(\tau)$, $\tau>0$.
\item 
Any $\mathbb{F}$-limit of the sequence $\left\{\left((M,g_i(t))_{t\in[-2,-1]},(\nu_t^i)_{t\in [-2,-1]}\right)\right\}_{i=1}^\infty$, where $\nu^i_t:=\nu_{p_0,0\,|\,\tau_it}$, given by Bamler's compactness theorem \cite[Theorem 7.6]{Bam20b} is of the form 
\[
    \left(
    (M_\infty,g_\infty(t))_{t\in[-2,-1]},
    (\nu^\infty_t)_{t\in [-2,-1)}
    \right),
\]
up to isometry, where $(\nu_t^\infty)_{t\in [-2,-1)}$ is a conjugate heat flow made of a shrinker potential function.
\end{enumerate}
\end{Theorem}

\bigskip

The proof of the above theorem consists of several elements. First of all, the smooth convergence (\ref{smoothconvergence}) in Assumption B implies locally uniform geometry bounds for $g_i(t)$ around points $p_i$, which, after scaling back, implies that $(M,g(t))$ is locally uniformly Type I (see Definition \ref{def}) along $\{(p_i,-\tau_i)\}_{i=1}^\infty$. Hence, by \cite{CZ20}, the limit $(M_\infty,g_\infty(t))$ admits a shrinker structure, and it is an asymptotic shrinker in the sense of Perelman. Then, applying Theorem \ref{Coro_nash_2} to the the $\ell$-centers $(p_i,-\tau_i)$, we obtain a uniform lower bound for the Nash entropy. Thirdly, the Nash entropy bound provides a Gaussian upper bound for the conjugate heat kernel \cite[Theorem 7.2]{Bam20a}, which is applied to show Theorem \ref{Theorem_main}(2). Finally, the Gaussian upper bound of the conjugate heat kernel also implies that an $H_n$-center is always not far from an $\ell$-center; this is the main idea in the proof of Theorem \ref{Theorem_main}(3).

The last two authors \cite{MZ21} have observed  that, on an ancient solution with bounded curvature within each compact time interval, the limit of the Nash entropy $\lim_{\tau\rightarrow\infty}\mathcal{N}_{x,t}(\tau)$ and the limit of Perelman's entropy $\lim_{\tau\rightarrow\infty}\mathcal{W}_{x,t}(\tau)$ are independent of the base point $(x,t)$. What is new in the corollary below is that the limit of the reduced volume $\lim_{\tau\rightarrow\infty}\mathcal{V}_{x,t}(\tau)$ is also independent of the base point $(x,t)$. This fact follows not from \cite{MZ21}, but from an observation that in the statement of Assumption B, the base point $(p_0,0)$ is not so important as the sequence of positive scales $\{\tau_i\}_{i=1}^\infty$.

\begin{Corollary}\label{entropynoloss}
Let $(M,g(t))_{t\in(-\infty,0]}$ be an ancient solution satisfying (\ref{curvaturebound}) and Assumption B with respect to $\big(p_0,0,\{\tau_i\}_{i=1}^\infty,\{p_i\}_{i=1}^\infty\big)$. Then the following holds.
\begin{eqnarray*}
\lim_{\tau\rightarrow\infty}\mathcal{N}_{p'_0,t'_0}(\tau)=\lim_{\tau\rightarrow\infty}\mathcal{W}_{p'_0,t'_0}(\tau)=\lim_{\tau\rightarrow\infty}\log\mathcal{V}_{p'_0,t'_0}(\tau)=\mu_\infty,
\end{eqnarray*}
where $\mathcal{W}$ and $\mathcal{V}$ stand for Perelman's entropy and reduced volume, respectively, $(p'_0,t'_0)$ is an arbitrary point in $M$, (in particular, $(p'_0,t'_0)$ is not necessarily $(p_0,0)$), and $\mu_\infty$ is the entropy of the asymptotic shrinker $(M_\infty,g_\infty(t))$ (see section 2.2 and section 2.3 for the definitions). 
\end{Corollary}

As a complement of Theorem \ref{Theorem_main}, we also consider its reciprocal problem: is a tangent flow of an ancient solution also Perelman's asymptotic shrinker? This, of course, is not true if the tangent flow is not smooth. However, if the tangent flow is smooth, then this problem can be easily answered with Bamler's results in \cite{Bam20c}, which are currently only proved to be true for a sequence of Ricci flows on closed manifolds and with bounded entropies. Hereby we would like to clarify that the following theorem is true \textbf{only if} the results in \cite{Bam20c} (especially Theorem 1.6 therein) are valid for noncompact Ricci flows with bounded geometry within each compact time interval (as it is conventional, by \textbf{bounded geometry} we always mean that the sectional curvature is bounded from above and from below, and that the volumes of unit balls are uniformly bounded away from $0$). This generalization is not among the goals of the present article. However, its validity, with relatively strong assurance, is not only our speculation, but expectation.

\begin{Theorem}[The reciprocal of Theorem \ref{Theorem_main}]\label{Thm_main_reciprocal}
Let $(M^n,g(t))_{t\in(-\infty,0]}$ be an ancient solution with bounded geometry within each compact time interval and with bounded Nash entropy, that is, there is a point $p_0\in M$ such that $\mathcal{N}_{p_0,0}(\tau)\geq -Y$ for all $\tau>0$, where $Y$ is a constant. If a tangent flow of $(M,g(t))_{t\in(-\infty,0]}$ at infinity (c.f. Definition \ref{tangentflowatinfinity}) is smooth, then it is also an asymptotic shrinker in the sense of Perelman.
\end{Theorem}

\textbf{Remark:} Of course, if we assume the underlying manifold $M^n$ (not necessarily the underlying manifold of the tangent flow) to be closed, then the proof of the above theorem can easily go through, since the techniques of \cite{Bam20c} do apply in this case.
\\

Under Assumption B, we also consider Perelman's $\nu$-functional (see section 2.2 for the definition) on the ancient solution. We prove that on such an ancient solution, the $\nu$-functional is uniformly bounded everywhere, and a lower bound (indeed, the infimum) is the entropy of the asymptotic shrinker.

\begin{Theorem}\label{nu-functional}
Let $(M,g(t))_{t\in(-\infty,0]}$ be an ancient Ricci flow satisfying (\ref{curvaturebound}) and Assumption B. Then we have
\begin{eqnarray}\label{nu_nonsense_00}
\inf_{t\leq 0}\nu(g(t))=\mu_\infty>-\infty,
\end{eqnarray}
where $\mu_\infty$ is the entropy of the asymptotic shrinker in Theorem \ref{Theorem_main}(1). In particular, this result is true for all noncollapsed Type I ancient solutions and all $\kappa$-solutions.
\end{Theorem}

Once the $\nu$-functional is known to be bounded, the following logarithmic Sobolev inequalities and Sobolev inequalities are simply consequences of straightforward computations (c.f. \cite{Zhq07, LW20}).

\begin{Corollary}[The logarithmic Sobolev and Sobolev inequalities]
Let $(M,g(t))_{t\in(-\infty,0]}$ be an ancient Ricci flow satisfying (\ref{curvaturebound}) and Assumption B. Then, for any $t\in(-\infty,0]$, the following are true.
\begin{enumerate}[(1)]
    \item Logarithmic Sobolev inequality: for any compactly supported locally Lipschitz function $u$ and positive scale $\tau>0$, we have
    \begin{align*}
    \int_M u^2\log u^2dg_t-\left(\int_Mu^2dg_t\right)\int_M u^2dg_t+\left(\mu_\infty+n+\frac{n}{2}\log(4\pi\tau)\right)\int_Mu^2dg_t&
    \\
    \leq \tau\int_M(4|\nabla u|^2+Ru)dg_t&.
    \end{align*}
    \item Sobolev inequality: for any compactly supported locally Lipschitz function $u$, we have
    \begin{eqnarray*}
    \left(\int_M |u|^{\frac{2n}{n-2}}dg_t\right)\leq C(n)e^{-\frac{2\mu_\infty}{n}}\int_M(4|\nabla u|^2+Ru)dg_t.
    \end{eqnarray*}
\end{enumerate}
Here $\mu_\infty$ is the entropy of the asymptotic shrinker in Theorem \ref{Theorem_main}(1). In particular, this result is true for all noncollapsed Type I ancient solutions and all $\kappa$-solutions.
\end{Corollary}

Obviously, we may also use Theorem \ref{nu-functional} to obtain a finer version of the main results in \cite{Zhang20} and \cite{MZ21}. We include the following corollary, and the details of the proof are merely combinations of Theorem \ref{nu-functional} with \cite{Zhang20} and \cite{MZ21}.

\begin{Corollary}
Let $(M,g(t))_{t\in(-\infty,0]}$ be an ancient solution satisfying (\ref{curvaturebound}). Furthermore, assume \emph{either one} of the following conditions is true.
\begin{enumerate}[(1)]
    \item $g(t)$ satisfies a Type I curvature bound, that is, there is a constant $C$ such that $$|\Rm_{g_t}|\leq\frac{C}{|t|}\quad\text{ for all }\quad t\in(-\infty,0).$$
    \item $g(t)$ satisfies Hamilton's trace Harnack $$\frac{\partial R}{\partial t}+2\langle X,\nabla R\rangle+2\Ric(X,X)\geq 0\quad\text{ for all vector field } X,$$
    and there is a constant $C$ such that
   $$|{\Rm}|\leq CR\quad\text{ everywhere on }\quad M\times(-\infty,0].$$
\end{enumerate}
Then, $(M,g(t))_{t\in(-\infty,0]}$ is strongly $\kappa$-noncollapsed on all scales if and only if $$\inf_{t\leq 0}\nu(g(t))\geq -\beta,$$ where $\kappa$ and $\beta$ are mutually dependent on.
\end{Corollary}

The next two results imply that the Nash entropy depends only on the local geometry around the $\ell$-centers or the $H_n$-centers, and does not depend on the geometry near the base point. These theorems are particularly useful in the long-time analysis of the Nash entropy. 

\begin{Theorem}\label{Coro_nash_2}
Let $(M^n,g(t))_{t\in I}$ be a complete Ricci flow with bounded curvature within each time interval compact in $I$. Let $s$, $t\in I$, $s<t$, and assume that $s-(t-s)\in I$. Let $x\in M$ and let $(z,s)$ be an $\ell$-center of $(x,t)$. Furthermore, assume that
\begin{gather*}
|{\Ric}|\leq\frac{C_0}{t-s}\quad \text{ on }\quad B_s(z,\sqrt{t-s})\times[s-(t-s),s],
\\
\operatorname{Vol}_{g_s}\big(B_s(z,\sqrt{t-s})\big)\geq\alpha (t-s)^{\frac{n}{2}}.
\end{gather*}
Then we have
\begin{eqnarray*}
\mathcal{N}_{x,t}(t-s)\geq-\beta,
\end{eqnarray*}
where $\beta$ is a positive constant depending only on $n$, $\alpha$, and $C_0$. Consequently, under the same assumption of the theorem, we have
\begin{eqnarray*}
K(x,t\,|\,y,s)\leq\frac{C}{(t-s)^{\frac{n}{2}}}\exp\left(-\frac{\dist_s^2(z,y)}{C(t-s)}\right)\quad\text{ for all }\quad y\in M,
\end{eqnarray*}
where $C$ depends only on $n$, $\alpha$, and $C_0$, and $K$ is the conjugate heat kernel (see section 2.1 for the definition.)
\end{Theorem}

\begin{Theorem}\label{Thm_nash}
Let $(M^n,g(t))_{t\in I}$ be a complete Ricci flow with bounded curvature within each time interval compact in $I$. Let $s$, $t\in I$, $s<t$, and assume that $s-(t-s)\in I$. Let $x\in M$ and let $(z,s)$ be an $H_n$-center of $(x,t)$. Furthermore, assume that
\begin{gather*}
\operatorname{Vol}_{g_s}\big(B_s(z,\sqrt{t-s})\big)\geq\alpha (t-s)^{\frac{n}{2}}.
\end{gather*}
Then all the conclusions in Theorem \ref{Coro_nash_2} are true, with constants $\beta$ and $C$ depending on $n$ and $\alpha$.
\end{Theorem}

\textbf{Remarks:} \begin{enumerate}
    \item In Theorem \ref{Coro_nash_2}, the assumptions can be replaced by\begin{gather*}
|{\Ric}|\leq\frac{C_0}{\varepsilon^2(t-s)}\quad \text{ on }\quad B_s(z,\varepsilon\sqrt{t-s})\times[s-\varepsilon^2(t-s),s],
\\
\operatorname{Vol}_{g_s}\big(B_s(z,\varepsilon\sqrt{t-s})\big)\geq\alpha \varepsilon^n(t-s)^{\frac{n}{2}},
\\
s-\varepsilon^2(t-s)\in I;
\end{gather*}
in Theorem \ref{Thm_nash}, the assumptions can be replaced by\begin{gather*}
\operatorname{Vol}_{g_s}\big(B_s(z,\varepsilon\sqrt{t-s})\big)\geq\alpha \varepsilon^n(t-s)^{\frac{n}{2}},
\\
s-\varepsilon^2(t-s)\in I,
\end{gather*}
where $\varepsilon$ is any positive constant. If so, then the constants in the conclusions also depend on $\varepsilon$.
\item Theorem \ref{Coro_nash_2} and Theorem \ref{Thm_nash} do not imply each other. The reason is because unless the Nash entropy is known to be bounded, one does not know whether an $\ell$-center is close to an $H_n$-center. In fact, the statement of Theorem \ref{Thm_nash} is much stronger than Theorem \ref{Coro_nash_2}. This is because of Bamler's good gradient estimates for heat kernels.
\item Comparing Theorem \ref{Thm_nash} with \cite[Theorem 6.2]{Bam20a}, we have that, if $(z,t-r^2)$ is an $H_n$-center of $(x,t)$, then $r^{-n}\Vol_{g_{t-r^2}}\left(B_{t-r^2}(z,\sqrt{2H_n}r)\right)$ and $\N_{x,t}(r^2)$ mutually bound each other.
\end{enumerate}

Theorem \ref{Coro_nash_2} and Theorem \ref{Thm_nash} show an interesting behavior of the Nash entropy: it is the geometry around the $H_n$-centers or the $\ell$-centers that determines (a lower bound of) the Nash entropy. This is somewhat in contrast to our first impression, since the Nash entropy is defined as a global integral. But the $H_n$-concentration property provides a likely explanation: as we know, the $H_n$-centers are points around which the conjugate heat flow (regarded as a probability measure) accumulates its measure (\cite[Proposition 3.13]{Bam20a}), and hence the region far away from an $H_n$-center should make little contribution to the Nash entropy. Furthermore, this ``locality'' behavior is also in time. In other words, $\mathcal{N}_{x,t}(t-s)$ is bounded so long as the geometry near $(z,s)$, an $H_n$-center of $(x,t)$, is bounded, regardless of what is happening to the Ricci flow on $M\times[s,t]$. We hope these results to be useful in future studies.

This paper is organized as follows. In section 2 we review some basic definitions and results. In section 3 we prove Bamler's conjugate heat kernel estimates for noncompact Ricci flows. In section 4 we prove Theorem \ref{Coro_nash_2} and Theorem \ref{Thm_nash}. In section 5 we prove some basic properties for ancient solutions satisfying Assumption B, including the boundedness of the Nash entropy; Theorem \ref{Theorem_main}(1) is proved in this section. In section 6 we explore the relation between smooth convergence and $\mathbb{F}$-convergence. In section 7 we prove Theorem \ref{Theorem_main}(2)(3). In section 8 we prove Corollary \ref{entropynoloss}. In section 9 we prove Theorem \ref{nu-functional}. In section 10 we consider two classical cases in which the asymptotic shrinker is known to exist, and show that they are special cases of Theorem \ref{Theorem_main}. In section 11 we prove Theorem \ref{Thm_main_reciprocal}. 
\\

\emph{Acknowledgement:} The authors are much indebted to Professor Bennett Chow for many insightful discussions. The first author would also like to thank Professor Richard Bamler for many good suggestions.

\section{Preliminaries}

\subsection{Conjugate heat kernel}
Let $(M^n,g(t))_{t\in I}$ be a complete Ricci flow.
We denote by 
    $K(x,t\,|\, y,s)$
the unique minimal fundamental solution to the heat equation coupled with  $g(t)$, that is, 
\begin{align*}
\Box_{x,t} K(x,t\,|\,y,s)=0,
& \quad \lim_{t\rightarrow s+} K(x,t\,|\,y,s)
= \delta_y(x),\\
\Box^*_{y,s} K(x,t\,|\,y,s)=0,
& \quad \lim_{s\rightarrow t-} K(x,t\,|\,y,s)
= \delta_x(y),
\end{align*}
where 
\[
\Box_{x,t}:= \partial_t - \Delta_{g_t,x},
\quad
\Box^*_{y,s} :=-\partial_s - \Delta_{g_s,y} + R(y,s).
\]
Throughout this paper, the laplacians and the covariant derivatives are all time-dependent, computed using the evolving metric of the Ricci flow. We shall suppress the subindices in the notations such as $\Delta_{g_t,x}$ when the metric and the variables are understood. It is well known that, whenever the integration by parts at infinity is valid, the conjugate heat equation preserves the integral. Hence, at least for Ricci flows with bounded curvature within each compact time interval, the measure $\nu_{x,t\,|\,s}$ defined as
\begin{eqnarray}\label{CHK_measure}
\nu_{x,t\,|\,s}(A):=\int_AK(x,t\,|\, \cdot, s) dg_s,\quad A\subset M
\end{eqnarray}
is always a probability measure, where $x\in M$, $s,t\in I$, and $s\le t.$ In this article, both the fundamental solution $K(x,t\,|\,\cdot,\cdot)$ and the evolving probability measure $(\nu_{x,t\,|\,s})_{s\in I\cap(-\infty,t]}$ will be referred to as the \emph{conjugate heat kernel}, whenever there is no ambiguity. One can also use other positive solutions to the conjugate heat equation instead of a fundamental solution to construct time-dependent probability measures as (\ref{CHK_measure}), and such evolving probability measures are usually referred to as \emph{conjugate heat flows}. For the conjugate heat kernel, there are the following logarithmic Sobolev and Poincar\'e inequalities proved in \cite{HN14}.

\begin{Proposition}[Hein-Naber's logrithmic Sobolev and Poincar\'e inequalities \cite{HN14}]\label{Hein-Naber-log-Sobolev}
Suppose that $(M^n,g(t))_{t\in [-T,0]}$ is a complete Ricci flow with bounded curvature. Let $x_0\in M$ be a fixed point and $\nu_s:=\nu_{x_0,0\,|\,s}$ for $s\in [-T,0).$ Then, for any function $u\ge 0, \sqrt{u}\in C_0^{0,1}(M)$, we have
\begin{equation}
   \label{ineq: log-Sobolev}
    \int_M u\log u \, d\nu_{s}
    -\left( \int_M u\, d\nu_s\right)
    \log \left( \int_M u\, d\nu_s\right)
    \le |s| \int_M \frac{|\nabla u|^2}{u} d\nu_s.
\end{equation}
For any $u\in C_0^{0,1}(M),$ we have
\begin{equation}
    \int_M u^2d\nu_s-\left(\int_M ud\nu_s\right)^2\leq 2|s|\int_M |\nabla u|^2d\nu_s.\label{ineq: poincare} 
\end{equation}
\end{Proposition}

In the original statement of the above results, Hein-Naber \cite{HN14} assumed bounded geometry for the Ricci flow. However, given all the estimates for the conjugate heat kernel (c.f. \cite[Corollary 26.26, Theorem 26.31]{RFV3}, \cite[Theorem 10]{EKNT08}, and \cite{Zhq06,BCP10}), it is not difficult to weaken the assumption to bounded curvature alone. We shall include the proof in Appendix A.

Hein-Naber also proved the following Gaussian concentration theorem using Davies' technique and their logarithmic Sobolev inequalities.

\begin{Proposition}[Hein-Naber's Gaussian concentration \cite{HN14}]\label{prop: gaussian concentration}
Under the same assumption as the above proposition, we have
\begin{equation}
\label{ineq: gaussian concentration}
   \nu_s(A)\nu_s(B)
    \le \exp\left\{
        - \frac{\dist^2_s(A,B)}{8|s|}
    \right\}, 
\end{equation}
for any measurable subsets $A$, $B\subset M$. Here $\dist_s$ means the classical distance between two sets computed using the metric $g(s)$, not the Hausdorff distance.
\end{Proposition} 

\subsection{Perelman's entropy functionals}

Let us recall Perelman's definition of the $\mathcal{W}$-functional: 
\begin{eqnarray}\label{Perelmansentropy}
\W(g,f,\tau):=\int_M\Big(\tau\big(|\nabla f|^2+R\big)+f-n\Big)(4\pi\tau)^{-\frac{n}{2}}e^{-f}dg,
\end{eqnarray}
where $(M^n,g)$ is a Riemannian manifold, $f$ is a smooth function, and $\tau>0$ is a positive scale. If we let $u:=(4\pi\tau)^{-\frac{n}{2}}e^{-f}$, then we may rewrite
\begin{eqnarray}\label{anotherPerelmansentropy}
\bW(g,u,\tau)&:=&\int_M\left(\tau\left(\frac{|\nabla u|^2}{u^2}+R\right)-\log u-\frac{n}{2}\log(4\pi\tau)-n\right)u\,dg.
\end{eqnarray}
In fact, this form of the $\mathcal{W}$ functional is what we will apply in the discussion of section 9. Note that for any $c>0$, we have
\begin{eqnarray*}
\W(cg,f,c\tau)=\W(g,f,\tau),\quad \bW(cg,c^{-\frac{n}{2}}u,c\tau)=\bW(g,u,\tau).
\end{eqnarray*}

Perelman's $\mu$-functional and $\nu$-functional are defined as follows.
\begin{eqnarray*}
\mu(g,\tau)&:=&\inf\left\{\bW(g,u,\tau)\ \bigg|\ \sqrt{u}\in C^\infty_0(M),\ u\geq 0,\ \text{ and } \int_Mudg=1\right\},
\\
\nu(g)&:=&\inf_{\tau>0}\mu(g,\tau).
\end{eqnarray*}
Here for the $\mu$ functional we are adopting the definition in \cite{RFV1}; see the formula above Lemma 6.28 therein. It is well understood that $\mu(g,\tau)$ is the logarithmic Sobolev constant at scale $\tau$, and that $\nu$ is the Sobolev constant.

Let us fix a point $(p_0,t_0)$ in the space-time of a Ricci flow $(M^n,g(t))_{t\in I}$, and denote the conjugate heat kernel based at $(p_0,t_0)$ as
\begin{eqnarray*}
K(p_0,t_0\,|\,x,t):=(4\pi(t_0-t))^{-\frac{n}{2}}e^{-f(x,t)}, \quad t\in I\cap(-\infty,t_0).
\end{eqnarray*}
Then, \emph{Perelman's entropy} and the \emph{Nash entropy} are respectively defined as
\begin{eqnarray}
\mathcal{W}_{p_0,t_0}(\tau)&:=& \mathcal{W}\big(g(t_0-\tau),f(\cdot,t_0-\tau),\tau\big)
\\\nonumber
&=&\int_M\Big(\tau\left(|\nabla f|^2+R\right)+f-n\Big)(\cdot,t_0-\tau)\,d\nu_{p_0,t_0\,|\, t_0-\tau},
\\
\mathcal{N}_{p_0,t_0}(\tau)&:=&\int_M f(\cdot,t_0-\tau)\,d\nu_{p_0,t_0\,|\, t_0-\tau}-\frac{n}{2},
\end{eqnarray}
for all $\tau>0$ and $t_0-\tau\in I$. The point $(p_0,t_0)$ is called the \emph{base point}. It is well known from Perelman \cite{Per02} that both Perelman's entropy and the Nash entropy are increasing in time (and hence decreasing in $\tau$). Although the monotonocity of $\mathcal{W}_{p_0,t_0}$ or $\mathcal{N}_{p_0,t_0}$ requires the integration by parts at infinity, yet this is valid at least for Ricci flows with bounded curvature.

\subsection{Perelman's reduced distance}
We briefly review Perelman's reduced distance and reduced volume. Let $(M,g(t))_{t\in[-T,0]}$ be a Ricci flow with bounded curvature. Let $(p_0,t_0)\in M\times(-T,0]$ be a fixed point in space-time. Then, Perelman's \emph{reduced distance} is defined as
\begin{eqnarray}\label{definitionofl}
\ell_{p_0,t_0}(x,\tau):=\frac{1}{2\sqrt{\tau}}\inf_{\gamma}\int_0^\tau\sqrt{s}\left(|\dot{\gamma}(s)|_{g(t_0-s)}^2+R(\gamma(s),t_0-s)\right)ds,
\end{eqnarray}
where $x\in M$, $\tau\in(0,T-|t_0|]$, and the infimum is taken over all piecewise smooth curves $\gamma:[0,\tau]\rightarrow M$ satisfying $\gamma(0)=p_0$ and $\gamma(\tau)=x$. The minimizer of (\ref{definitionofl}) is usually called a minimal $\mathcal{L}$-geodesic from $(p_0,t_0)$ to $(x,t_0-\tau)$. $(p_0,t_0)$ is called the base point of $\ell$, and whenever the base point is understood, we shall suppress the subindex in the notaion $\ell_{p_0,t_0}(\cdot,\cdot)$. The \emph{reduced volume based at $(p_0,t_0)$} is defined as
\begin{eqnarray}
\mathcal{V}_{p_0,t_0}(\tau):=\int_M(4\pi\tau)^{-\frac{n}{2}}e^{-\ell_{p_0,t_0}(\cdot,\tau)}dg_{t_0-\tau}.
\end{eqnarray}

Perelman's reduced distance and reduced volume satisfy many nice equations and inequalities. Among them the most important one is the monotonicity of the reduced volume.

\begin{Proposition}
Perelman's reduced volume $\mathcal{V}(\tau)$ is an increasing function in time (and hence a decreasing function in $\tau$).
\end{Proposition}

The underlying reason for the monotonicity of the reduced volume is the fact that its integrand is a ``sub''-conjugate heat kernel.

\begin{Proposition}\label{basic_l}
Let $\ell_{p,0}(x,\tau)$ be the reduced distance based at $(p,0)$. Then, $\displaystyle u(x,t):=(4\pi|t|)^{-\frac{n}{2}}e^{-\ell_{p,0}(x,|t|)}$ is a subsolution to the conjugate heat equation $-\partial_t u-\Delta u+Ru=0$ which also converges to the Dirac delta measure based at $p$ as $\tau\rightarrow 0+$. Precisely, this means 
\begin{gather*}
\frac{\partial\ell}{\partial\tau}-\Delta_{g_{-\tau}}\ell+\left|\,\nabla_{g_{-\tau}}\ell\,\right|_{g_{-\tau}}^2-R_{g_{-\tau}}+\frac{n}{2\tau}\geq 0,\\
\lim_{\tau\rightarrow 0+}(4\pi\tau)^{-\frac{n}{2}}e^{-\ell_{p,0}(\cdot,\tau)}=\delta_{p}.
\end{gather*}
Both of the formulas above are understood in the sense of distribution. In consequence, we have
\begin{eqnarray}\label{subsolution}
(4\pi|t|)^{-\frac{n}{2}}e^{-\ell_{p,0}(x,|t|)}\leq K(p,0\,|\,x,t)\quad \text{ for all } (x,t)\in M\times [-T,0),
\end{eqnarray}
where $K(p,0\,|\,\cdot,\cdot)$ is the conjugate heat kernel based at $(p,0)$.
\end{Proposition}

By an elementary application of the maximum principle, Perelman \cite{Per02} proved that $\ell(\cdot,\tau)$ always attains its minimum. This minimum point should be viewed as the ``center'' of the reduced distance.

\begin{Proposition}
Let $\ell_{p,0}$ be the reduced distance based at $(p,0)$. Then we have
\begin{eqnarray}\label{lcenter}
\min_{M}\ell(\cdot,\tau)\leq\frac{n}{2} \quad \text{ for all } \quad \tau\in(0,T].
\end{eqnarray}
\end{Proposition}

The point(s) where the minimum in formula (\ref{lcenter}) is attained plays an important role in our arguments. In most of the cases, it turns out that such a minimum point is not far from Bamler's $H_n$-centers. Hence, we would like to assign a special term to these points.

\begin{Definition}
Let $(M^n,g(t))_{t\in I}$ be a Ricci flow, and let $(x,t)\in M\times I$ be a point in space time. Let $s\in I\cap(-\infty,t)$. Then, $(z,s)$ is called an \emph{$\ell$-center} of $(x,t)$ if
\begin{eqnarray*}
\ell_{x,t}(z,t-s)\leq\frac{n}{2}.
\end{eqnarray*}
\end{Definition}

\textbf{Remark:} Similar to the case of the $H_n$-center, the $\ell$-center is not necessarily unique at a fixed time $s$ for a fixed base point $(x,t)$. Furthermore, in practice (especially when considering the base points for the blow-down sequence from which we obtain an asymptotic shrinker), a sequence of space-time points along which $\ell$ is uniformly bounded serves equally well as a sequence of $\ell$-centers; see Definition \ref{def}.

\subsection{Shrinking gradient Ricci soliton and its entropy} 

A shrinking gradient Ricci soliton (or Ricci shrinker for short) is a tuple $(M^n,g,f)$, where $(M^n,g)$ is a smooth Riemannian manifold, and $f$ is a smooth function on $M$ called a \emph{shrinker potential}, satisfying
\[
    \Ric + \nabla^2 f = \tfrac{1}{2}g.
\]
Apparently, the shrinker potential is not unique, for one may always add a constant to it. If we normalize $f$ such that
\[
	(4\pi)^{-n/2}\int_M e^{-f}dg = 1,
\]
then
\[
	 \mu:= f-(|\nabla f|^2 + R)
\]
is a constant called the \emph{shrinker entropy}. 

Let $\Phi_t$ be the 1-parameter group of self-diffeomorphisms generated by $\nabla f$ with $\Phi_0:M\rightarrow M$ being the identity map. 
Define $\phi_t:= \Phi_{-\log|t|}$,
$g(t) := |t| \phi^*_t g$, and $f_t=f\circ\phi_t$. Then, $(M,g(t))_{t\in(-\infty,0)}$ is an ancient solution to the Ricci flow, and $(4\pi|t|)^{-\frac{n}{2}}e^{-f_t}$ is a solution to the conjugate heat equation. The evolving tuple $(M,g(t),f_t)_{t\in(-\infty,0)}$ is called the \emph{canonical form} of the Ricci shrinker.  The canonical form satisfies
\begin{align}\label{canonicalformnormalization}
\Ric_{g_t}+\nabla^2 f_t=\frac{1}{2|t|}g(t),&
\\\nonumber
-|t|\big(|\nabla f_t|^2_{g_t}+R_{g_t}\big)+f_t\equiv \mu,&
\end{align}
where $\mu$ is the shrinker entropy. It turns out that the entropy of a fixed shrinker is unique (even if the normalized potential is not necessarily unique). In fact, Li-Wang \cite{LW20} proved the following strong statement.

\begin{Proposition}[\cite{LW20}]\label{shrinkernufunctional}
Let $(M^n,g,f)$ be a Ricci shrinker, where $f$ is normalized so that $\int_M (4\pi)^{-\frac{n}{2}}dg=1$. Let $\mu$ be the shrinker entropy. Then we have
\begin{eqnarray*}
\mu=\mathcal{W}(g,f,1)=\mu(g,1)=\nu(g).
\end{eqnarray*}
Furthermore, if we let $(M,g(t),f_t)_{t\in(-\infty,0)}$ be the canonical form, then we also have
\begin{eqnarray*}
\mu=\mathcal{W}(g(t),f_t,|t|)=\mu(g(t),|t|)=\nu(g(t))\quad\text{ for all }\quad t\in(-\infty,0).
\end{eqnarray*}
\end{Proposition}

\subsection{Locally uniformly Type I ancient solution}

Cheng-Zhang \cite{CZ20} studied a geometric condition for ancient Ricci flows, called the locally uniformly Type I condition. Precisely, it is defined as follows. 

\begin{Definition}[Locally uniformly Type I ancient solutions]\label{def}
Let $(M,g(t))_{t\in(-\infty,0]}$ be an ancient Ricci flow. Fix $p_0\in M$ and let $\{(p_i,-\tau_i)\}_{i=1}^\infty\subset M\times(-\infty,0)$ be a sequence of space-time points with $\tau_i\nearrow\infty$. Then, $(M,g(t))_{t\in(-\infty,0]}$ is called \emph{locally uniformly Type I} along the sequence $\{(p_i,-\tau_i)\}_{i=1}^\infty$, if the following hold.
\begin{enumerate}[(1)]
 \item $\ell_{p_0,0}$ is bounded along $\{(p_i,-\tau_i)\}_{i=1}^\infty$, that is, \begin{eqnarray*}\label{def1}
 \limsup_{i\rightarrow\infty}\ell(p_i,\tau_i)<\infty.
 \end{eqnarray*}

\item Around the space-time points $(p_i,-\tau_i)$, the curvature has a locally uniformly Type I bound. More precisely, there exists a positive function $C:(0,\infty)\rightarrow(0,\infty)$ with the following property: for all $A>0$, there exists $i_0\in\mathbb{N}$, depending on $A$, such that
    \begin{eqnarray*}\label{def2}
    \sup_{B_{-\tau_i}(p_i,r\sqrt{\tau_i})\times[-2\tau_i,-\tau_i]}|{\Rm}|\leq \frac{C(r)}{\tau_i} \quad\text{for all}\quad i\geq i_0\quad \text{and for all}\quad r\leq A.
    \end{eqnarray*}

\item There is a time-wise Ricci curvature lower bound for $g(t)$. In other words, there exists a continuous positive function $K:(-\infty,0]\rightarrow(0,\infty)$, such that
    \begin{eqnarray*}\label{def3}
    \Ric_{g_t}\geq-K(t)g(t)\quad\text{ for all }\quad t\in(-\infty,0].
    \end{eqnarray*}

\item $g(t)$ is noncollapsed along $(p_i,-\tau_i)$. In other words,
    \begin{eqnarray*}\label{def4}
    \liminf_{i\rightarrow\infty} \Big((\tau_i)^{-\frac{1}{2}}\text{inj}\big(g(-\tau_i),p_i\big)\Big)>0,
    \end{eqnarray*}
    where $\text{inj}(g, x)$ stands for the injectivity radius of the Riemannian metric $g$ at $x$.
\end{enumerate}
\end{Definition}

With the locally uniformly Type I condition, Cheng-Zhang \cite{CZ20} obtained locally uniformly $C^0$ and $C^1$ estimates for the reduced distance $\ell$, and consequently proved the existence of an asymptotic shrinker. We shall include these results below. 

\begin{Proposition}[Proposition 5.1 in \cite{CZ20}]\label{C0-l-estimate-LUTypeI}
Let $(M,g(t))_{t\in(-\infty,0]}$ be a locally uniformly Type I ancient solution as described in Definition \ref{def}. Let $\ell_i(\cdot,\tau):=\ell_{p_0,0}(\cdot,\tau_i\tau)$ and $g_i(t):=\tau_i^{-1}g(\tau_it)$. Then, for any $\varepsilon\in(0,\frac{1}{4})$, there is a positive function $C(\cdot,\varepsilon):(0,\infty)\rightarrow(0,\infty)$ with the following property: for any $r>0$, it holds that $0\leq\ell_i(x,|t|)\leq C(r,\varepsilon)$ for all $(x,t)\in B_{g_{i,-1}}(p_i,r)\times[-2,-1-\varepsilon]$ whenever $i$ is large enough.
\end{Proposition}

\begin{Proposition}[Asymptotic Shrinker \cite{CZ20}]\label{TIshrinker}
Let $g_i(t)$ and $\ell_i$ be as described in the above proposition, then, after possibly passing to a subsequence, the sequence of tuples $\displaystyle\Big\{\big(M,g_i(t),\ell_i(\cdot,|t|)\big)_{t\in[-2,-1]}\Big\}_{i=1}^\infty$ converges to (the canonical form of) a shrinking gradient Ricci soliton $\big(M_\infty,g_\infty(t),\ell_\infty(\cdot,|t|)\big)_{t\in(-2,-1)}$, satisfying
\begin{eqnarray*}
\Ric_{g_\infty}(t)+\nabla^2\ell_\infty(\cdot,|t|)=\frac{1}{2|t|}g_\infty(t).
\end{eqnarray*}
The convergence of the Ricci flows is in the pointed smooth Cheeger-Gromov-Hamilton sense, and $\ell_i\rightarrow\ell_\infty$ in the $C^{0,\alpha}_{\operatorname{loc}}$ sense or in the weak $*W^{1,2}_{\operatorname{loc}}$ sense on $M_\infty\times(1,2)$, where $\ell_i$ should be understood to be pulled back by the defining diffeomorphisms of the Cheeger-Gromov-Hamilton convergence, and $\alpha$ is any number in $(0,1)$.
\end{Proposition}

\begin{Corollary}[Noncollapsing \cite{CZ20}]\label{noncollapsing}
A locally uniformly Type I ancient solution is (weakly) $\kappa$-noncollapsed on all scales, where $\kappa>0$ depends only on the entropy of the asymptotic shrinker. 
\end{Corollary}

\subsection{Metric flow and
$H_n$-center}

In this subsection we briefly review the notion of the metric flow introduced by Bamler \cite{Bam20b}. Since here (and in the next subsection) we are not attempting to be comprehensive about all the nuances, the reader should always resort to \cite{Bam20b} for more details. We shall use the notation $\PP(X)$ to represent the space of probability measures on a metric space $X.$  The metric flow is introduced as a natural generalization of the Ricci flow space-time. A metric flow over $I\subset \mathbb{R}$ is a tuple 
\[\big(\XX,\mathfrak{t}, (\dist_t)_{t\in I}, (\nu_{x\,|\,s})_{x\in \XX, s\in I, s\le \mathfrak{t}(x)}\big),\]
where $\mathfrak{t}:\XX\to I$ is the time function and $\XX_t:= \mathfrak{t}^{-1}(t)$ is called the time slice at $t$, $\dist_t$ is a metric on $\XX_t$, $\nu_{x\,|\,s}\in \PP(\XX_s)$ is a family of probability measures called the \emph{conjugate heat kernel} based at $x\in\XX$, and it satisfies the usual reproduction formula: for any $t_1\le t_2\le t_3$ in $I$ and for any $x\in \XX_{t_3}$, we have
\[
    \nu_{x\,|\,t_1}
    = \int_{\XX_{t_2}} \nu_{\cdot\,|\,t_1}\, d\nu_{x\,|\,t_2}.
\]
The sharp gradient estimate of Bamler \cite[Theorem 4.1]{Bam20b} is also axiomized into the definition of the metric flow. More details could be found in \cite[Definition 3.2]{Bam20b}. The last two authors generalized \cite[Theorem 4.1]{Bam20b} to noncompact Ricci flows with bounded curvature within each compact time interval (c.f. \cite{MZ21}). Hence, the following observation is straightforward.

\begin{Theorem}[Bamler]
\label{thm: smooth flows are metric flows}
Let $(M^n,g(t))_{t\in I}$ be a complete Ricci flow with bounded curvature within each time interval compact in $I$.
Then it induces a canonical metric flow in the sense of \cite[Definition 3.2]{Bam20b}.
\end{Theorem}

\begin{proof}
By \cite[Lemma 26.16]{RFV3}, the reproduction formula holds for the minimal heat kernel $K(x,t\,|\,y,s)$ coupled with the flow $(M,g(t))$. 
So \cite[Definition 3.2 (7)]{Bam20b} is satisfied if we choose $d\nu_{x,t\,|\,s} := K(x,t\,|\,\cdot,s)\, dg_s.$
As mentioned in the above paragraph, \cite[Definition 3.2 (1)-(6)]{Bam20b} are satisfied. 
\end{proof}

We shall then introduce some definitions and results for the metric flow. In particular, they can be applied to smooth Ricci flows satisfying the condition of the above theorem. Before we proceed to define the notion of $H$-concentration, let us recall the \emph{variance} of two probability measures. Let $\mu,\nu\in \PP(X)$, where $X$ is a metric space, then their variance is defined as
\[
    {\rm Var}(\mu,\nu)
    := \int_{X\times X} \dist^2(y_1,y_2)
    \, d\mu(y_1)d\nu(y_2),
\]
where $\dist$ is the metric on $X$. If $\mu$ is the same as $\nu$, then we also denote ${\rm Var}(\mu):={\rm Var}(\mu,\mu)$. $H$-concentration is defined as follows.

\begin{Definition}[Definition 3.30 in \cite{Bam20b}] A metric flow $\mathcal{X}$ over $I\subset\mathbb{R}$ is said to be $H$-concentrated, where $H$ is a positive number, if for any  $s, t\in I$, $s\leq t$, and $x_1, x_2$ $\in \mathcal{X}_t$, we have,
\begin{equation}\label{thedefinitionofHconcentrarion}
    \text{Var}(\nu_{x_1\,|\,
    s}, \nu_{x_2\,|\, s})\leq \text{dist}_{t}^2(x_1, x_2)+H(t-s).
\end{equation}
\end{Definition}
Combining (\ref{thedefinitionofHconcentrarion}) with the reproduction formula, we have that, on an $H$-concentrated metric flow $\XX$, for any $x_1,x_2\in \XX_t,$
\[
    {\rm Var}(\nu_{x_1\,|\,s},\nu_{x_2\,|\,s}) + H s
\]
is non-decreasing in $s\in I\cap(-\infty, t]$, and is bounded from above by $\dist_t^2(x_1,x_2)+Ht$ (c.f. \cite[Proposition 3.34]{Bam20b}). This fact guarantees the existence of $H$-centers (c.f. \cite[Proposition 3.36]{Bam20b}), which are defined as follows.

\begin{Definition}[Definition 3.35 in \cite{Bam20b}]\label{def_H_n_center}
Let $s,t\in I$ and $s\leq t$. A point $z$ $\in \mathcal{X}_s$ is called an $H$-center of $x$ $\in \mathcal{X}_t$ if 
\begin{equation*}
    \text{Var}(\delta_z, \nu_{x\,|\,s})\leq H(t-s).
\end{equation*}
\end{Definition}

The $H$-center adopted its name partially because the conjugate heat kernel accumulates its measure around it. Precisely, we have (\cite[Proposition 3.13]{Bam20a} and \cite[Lemma 3.37]{Bam20b}):

\begin{Proposition}[Bamler]\label{measureaccumulationofHcenter}
Let $\mathcal{X}$ be an $H$-concentrated metric flow over $I\subset\mathbb{R}$. Let $x\in \mathcal{X}_t$ be a fixed point, and let $z\in\mathcal{X}_s$ be an $H$-center of $x$, where $s<t$ and $s,t\in I$. Then, for any $A>1$, we have
\begin{eqnarray*}
\nu_{x\,|\, s}\big(B_s(z,\sqrt{AH(t-s)})\big)\geq 1-\frac{1}{A}.
\end{eqnarray*}
\end{Proposition}

\textbf{Remark:} Bamler \cite[Proposition 3.2]{Bam20a} proved that if $\mathcal{X}=M^n\times I$ is a Ricci flow space-time on a closed manifold $M^n$ over an interval $I$, then $\mathcal{X}$ must be $H_n$-concentrated, where $$H_n:=\frac{(n-1)\pi^2}{2}+4.$$ The same argument also works when the Ricci flow is complete and has bounded curvature within compact time intervals. \emph{This is a fact which shall be used throughout this article.} In general, given $x$ $\in \mathcal{X}_t$, and $s\leq t$, $H$-centers of $x$ may not be unique in $\XX_s$.

\subsection{Metric flow pair and $\mathbb{F}$-convergence}

Suppose $\XX$ is a metric flow over $I\subset\mathbb{R}$.
A family of probability measures $\mu_s\in \PP(\XX_s)$, where $s\in I'\subset I$, is called a \emph{conjugate heat flow} if it satisfies the reproduction formula: for any $s,t\in I',s\le t,$ we have
\[
    \mu_s
    = \int_{\XX_t} \nu_{x\,|\,s}\, d\mu_t(x).
\]
A metric flow pair is then defined to be a metric flow coupled with a conjugate heat flow.

\begin{Definition}[Definition 5.1 in \cite{Bam20b}] A pair $\left(\mathcal{X}, (\mu_t)_{t\in I'}\right)$ is called a \emph{metric flow pair} over $I\subset \mathbb{R}$ if the following conditions are satisfied:
\begin{itemize}
    \item $I'\subset I$, and $|I\setminus I'|=0$, where $|\,\cdot\,|$ is the Lebesgue measure;
    \item $\mathcal{X}$ is a metric flow over $I'$;
    \item $(\mu_t)_{t\in I'}$ is a conjugate heat flow on $\mathcal{X}$ with $\spt\mu_t=\mathcal{X}_t$ for all $t\in I'$.
\end{itemize}

\end{Definition}

The definition of $\mathbb{F}$-convergence requires the notions of coupling and $1$-Wasserstein distance between probability measures. Let $X, Y$ be metric spaces.
For any $\mu\in \PP(X)$ and $\nu\in \PP(Y),$ we denote by $\Pi(\mu,\nu)$ the space of \emph{couplings} between $\mu$ and $\nu$, namely, the set of all the probability measures $q\in \PP(X\times Y)$ satisfying
\[
    q(A\times Y)=\mu(A),\quad
    q(X\times B)=\nu(B),
\]
for any measurable subsets $A\subset X$ and $ B\subset Y$. The $1$-Wasserstein distance between $\mu,\nu\in \PP(X)$ is defined to be
\[
    \dist_{W_1}(\mu,\nu)
    := \inf_{q\in \Pi(\mu,\nu)} \int_{X\times X} \dist(x,y) \, dq(x,y).
\]
By the Kantorovich-Rubinstein Theorem, this definition is equivalent to
\[
    \dist_{W_1}(\mu,\nu)
    = \sup_{f}\left( \int f\, d\mu 
    - \int f\, d\nu\right),
\]
where the supremum is taken over all bounded $1$-Lipschitz functions $f$ on $X.$

It is to be noted that for any metric flow $\XX$, the $1$-Wassernstein distance between two conjugate heat flows satisfies a monotonicity property (\cite[Proposition 3.24(2)]{Bam20b}), namely, for any conjugate heat flows $(\mu^1_s)_{s\in I'}$ and $(\mu^2_s)_{s\in I''}$, we have
\begin{eqnarray}\label{monotoneofdW1}
    \dist_{W_1}^{\XX_s}
    (\mu^1_2,\mu^2_s)\quad\text{ is non-decreasing in }\quad s\in I'\cap I''.
\end{eqnarray}
Consequently, for any $x_1$ and $x_2\in\mathcal{X}_t$, we have
\begin{eqnarray}
\dist_{W_1}^{\XX_s}(\nu_{x_1\,|\,s},\nu_{x_2\,|\,s})\leq \dist_t(x_1,x_2)\quad\text{ for all }\quad s< t.
\end{eqnarray}
In fact, this monotonicity property is 
a consequence of the prescribed gradient estimate in the definition of the metric flow (\cite[Definition 3.2(6)]{Bam20b}).

We now introduce the definition of  $\mathbb{F}$-convergence within a correspondence, because this is the only version we shall use, and because it is essentially equivalent to the definition of $\mathbb{F}$-convergence itself. See more details in \cite[Section 6]{Bam20b}. Let $\displaystyle\big\{(\XX^i,(\mu_t^i)_{t\in I'^{, i}})\big\}_{i\in \mathbb{N}\cup\{\infty\}}$ be a sequence of metric flow pairs over a finite interval $I\subset\mathbb{R}$. A correspondence $\mathfrak{C}$ is a collection of complete and separable metric spaces $(Z_t,\dist^{Z_t})_{t\in I}$ together with isometric embeddings $\phi_t^i:(\XX_t^i,\dist_t^i)\to (Z_t,\dist^{Z_t})$ for $t\in I'^{,i}.$
Then, $(\XX^i,(\mu_t^i)_{t\in I'^{, i}})$ $\mathbb{F}$-converges to $(\XX^\infty,(\mu_t^\infty)_{t\in I'^{, \infty}})$ within the correspondence $\mathfrak{C}$ uniformly on $J\subset I$, denoted as
\[
    \left(\XX^i,(\mu_t^i)_{t\in I'^{, i}}\right)
    \xrightarrow{\makebox[1.5cm]{$\mathbb{F},\mathfrak{C},J$}}
    \left(\XX^\infty,(\mu_t^\infty)_{t\in I'^{,\infty}}\right),
\]
if for any $\varepsilon>0,$ there is an $\bar i\in \mathbb{N},$ such that if $i\ge \bar i,$
there is a measurable subset $E_i\subset I$ with $$J\subset I\setminus E_i\subset I'^{,i}\cap I'^{,\infty},$$  and there are couplings $q_t^i\in \Pi(\mu_t^i,\mu_t^\infty)$ for $t\in I\setminus E_i$ with the following properties:
\begin{itemize}
    \item $|E_i| \le \varepsilon^2$.
    \item For any $s,t\in I\setminus E_i,s\le t,$ it holds that
\[
\int_{\XX^i_t\times \XX^\infty_t}
\dist^{Z_s}_{W_1}
\left(
\phi^i_{s*}\nu^i_{x_1\,|\,s},
\phi^\infty_{s*}\nu^\infty_{x_2\,|\,s}
\right)\, dq_t^i(x_1,x_2)
\le \varepsilon.
\]
\end{itemize}
If $J$ above can be taken as any compact sub-interval of $I$, we say that the convergence is uniform over any compact sub-intervals.

The tangent flow at infinity of an ancient solution is then defined to be the $\mathbb{F}$-limit of a sequence of Ricci flows obtained by a Type I scaling process for the ancient solution.

\begin{Definition}[Tangent flow at infinity \cite{Bam20b}]\label{tangentflowatinfinity}
Let $(M,g(t))_{t\in(-\infty,0]}$ be an ancient solution. A metric flow pair $\big(\mathcal{X},(\mu_t)_{t\in (-\infty,0)}\big)$, where $\XX$ is a metric flow over $(-\infty,0]$, is called a \emph{tangent flow at infinity} of $(M,g(t))_{t\in(-\infty,0]}$, if there exist a fixed point $(p_0,t_0)\in M\times(-\infty,0]$ and a sequence $\tau_i\nearrow\infty$, such that
\begin{eqnarray*}
\left((M,g_i(t))_{t\in(-\infty,0]},(\nu^i_t)_{t\in(-\infty,0]}\right)\xrightarrow{\makebox[1cm]{$\mathbb{F}$}}\left(\mathcal{X},(\mu_t)_{t\in (-\infty,0)}\right),
\end{eqnarray*}
where 
\begin{gather*}
    g_i(t):=\tau_i^{-1}g(t_0+\tau_it),
    \\
    \nu^i_t:=\nu_{p_0,t_0\,|\,t_0+\tau_it},
\end{gather*}
for all $t\in(-\infty,0]$. Since the interval $(-\infty,0]$ is infinite, the $\mathbb{F}$-convergence here should be understood to be the $\mathbb{F}$-convergence on each finite subinterval of $(-\infty,0]$.
\end{Definition}

If we assume that the ancient solution in question has bounded curvature within each compact time intervals, then it must be $H_n$-concentrated. By \cite[Theorem 7.6]{Bam20b}, a tangent flow always exists for any fixed base point $(p_0,t_0)$ and for any sequence of scaling factors $\{\tau_i\}_{i=1}^\infty$. However, without a uniform lower bound on the Nash entropy, the tangent flow could be collapsed.

\section{Conjugate heat kernel estimates on noncompact manifolds}

In this section, we shall verify that \cite[Theorem 7.2]{Bam20a} and \cite[Proposition 5.13]{Bam20a} are still valid for Ricci flows with bounded curvature within each compact time interval. Bamler's original proof was for Ricci flows on closed manifolds. Given the assumptions which we make, these generalizations are but straightforward arguments, yet they are very important to the proofs of our main theorems.

\subsection{The Gaussian upper bound}

\begin{Theorem}[Theorem 7.1 in \cite{Bam20a}]\label{aCHKcoarseupperboundofbamler}
Let $(M^n,g(t))_{t\in I}$, where $I$ is a compact interval, be a complete Ricci flow with bounded curvature. Suppose $[s,t]\subset I$ and $R\geq R_{\operatorname{min}}$ on $M\times[s,t]$. Then for any $x,y\in M$, we have
\begin{eqnarray*}
K(x,t\,|\,y,s)\leq\frac{C}{(t-s)^{\frac{n}{2}}}\exp(-\mathcal{N}_{x,t}(t-s)),
\end{eqnarray*}
where $C:=C(R_{\operatorname{min}}(t-s))$ is a constant depending on $R_{\operatorname{min}}(t-s)$.
\end{Theorem}

\begin{Theorem}[Theorem 7.2 in \cite{Bam20a}]\label{gaussianupperbound}
Let $(M^n,g(t))_{t\in I}$, where $I$ is a compact interval, be a complete Ricci flow with bounded curvature. Suppose $[s,t]\subset I$ and $R\geq R_{\operatorname{min}}$ on $M\times[s,t]$. Let $(z,s)\in M\times I$ be an $H_n$-center of a point $(x,t)\in M\times I$. Then for any $\varepsilon>0$ and $y\in M$, we have
\begin{eqnarray*}
K(x,t\,|\,y,s)\leq\frac{C\exp(-\mathcal{N}_{x,t}(t-s))}{(t-s)^{\frac{n}{2}}}\exp\left(-\frac{\dist_s^2(z,y)}{(8+\varepsilon)(t-s)}\right),
\end{eqnarray*}
where $C:=C(R_{\operatorname{min}}(t-s),\varepsilon)$ is a constant depending on $R_{\operatorname{min}}(t-s)$ and $\varepsilon$.
\end{Theorem}

\begin{proof}[Proof of Theorem \ref{aCHKcoarseupperboundofbamler}]
The proof is not essentially different from \cite[Theorem 7.1]{Bam20a}, since the latter does not rely heavily on the fact that the background manifold is closed. The only potential issue is the application of the maximum principle. Hence, we shall verify \cite[(7.11)]{Bam20a} below. By parabolic rescaling, we may assume that $s=0$ and $t=1$. It suffices to show
\begin{equation} \label{ineq: v is subsolution}
    v(x,t)
    \le \int_{M} K(x,t\,|\,y,\tfrac{1}{2})
    v(y,\tfrac{1}{2})\, dg_{\frac{1}{2}}(y),
\end{equation}
for all $(x,t)\in M\times [\frac{1}{2},1]$, where
\begin{equation} \label{ultranonsense}
   v := (t-\tfrac{1}{2})|\nabla u|^2 + u^2,\quad
    u(x,t):= K(x,t\,|\,y,0), 
\end{equation}
and $y\in M$ is a fixed point.


In fact, by elementary computations, we have that $v$ is a subsolution to the heat equation, that is, $$\Box v=|\nabla u|^2-2(t-\tfrac{1}{2})|\nabla^2u|^2-2|\nabla u|^2\leq 0.$$Hence, to prove (\ref{ineq: v is subsolution}), we need only to verify that the maximum principle can be applied to $v$ on $M\times[\frac{1}{2},1]$.

By \cite[Lemma 26.17]{RFV3}, we can find a positive constant $J>0$, depending on the curvature bound on $M\times[0,1]$ and $\Vol_{g_0}\big(B_{g_0}(y,1)\big)>0$, such that $0<u(x,t)\leq J$ for all $(x,t)\in M\times[\frac{1}{4},1]$. We may then apply \cite[Theorem 3.2]{Zhq06} (for the proof in our noncompact setting, see \cite[Lemma 2.4]{Zhang21}) to obtain
\begin{eqnarray*}
\frac{|\nabla u|^2}{u^2}\leq\frac{1}{t-\tfrac{1}{4}}\log\frac{J}{u}\quad\text{ on }\quad M\times(\tfrac{1}{4},1].
\end{eqnarray*}
Hence, we have the following estimate on $M\times[\frac{1}{2},1]$.
\begin{eqnarray*}
|\nabla u|^2\leq 4\Big(u^2\log J-u^2\log u\Big)\leq C(J),
\end{eqnarray*}
where $C(J)$ is a constant depending only on $J$. Here we have applied the fact that $0<u\leq J$ on $M\times[\frac{1}{2},1]$. It then follows that the function $v$ as defined in (\ref{ultranonsense}) is bounded on $M\times[\frac{1}{2},1]$, and the standard maximum principle (c.f. \cite[Theorem 12.10]{RFV2}) implies that
\begin{equation*}
(t-\tfrac{1}{2})|\nabla u|^2(x,t)\leq v(x,t)
\leq \int_{M}K(x,t\,|\, y, \tfrac{1}{2})u^2(y,\tfrac{1}{2})d g_{\tfrac{1}{2}}.
\end{equation*}
This is exactly \cite[(7.11)]{Bam20a}.

For the rest of the proof, the reader may follow \cite{Bam20a} line by line, and we shall not include all the details therein. Besides the argument above, Bamler's original proof also applied \cite[Theorem 5.9]{Bam20a}, \cite[Corollary 5.11]{Bam20a}, and \cite[Theorem 6.2]{Bam20a}; the first two of them are already proved in \cite[Theorem 4.4, Corollary 4.5]{MZ21} for our case, and the third result can be easily generalized to our case using (\ref{Nashintegral_2}) below.
\end{proof}

The proof of Theorem \ref{gaussianupperbound} also follows closely after Bamler's argument, for it does not depend on anything holding exclusively on closed manifolds. In the course of the proof, all the auxiliary results have already been established in our case, as described in the last paragraph above. We shall not exhibit all these details here, since this could be a distraction from our main purpose, and since this type of generalizations are usually understood to be straightforward.

\subsection{An integral inequality involved in the Nash entropy}

Next, we prove \cite[Proposition 5.13]{Bam20a} on a complete Ricci flow with bounded curvature on compact time intervals. The reader will see later that formula (\ref{Nashintegral_2}) is particularly useful for the local analysis of the Nash entropy, that is, the estimation of the Nash entropy using local geometry (see section 4 below). Interestingly, we shall apply this formula in a somewhat reverse way as Bamler did in the proof of \cite[Theorem 6.2]{Bam20a}.

\begin{Proposition}[Proposition 5.13 in \cite{Bam20a}]\label{Nashintegral}
Let $(M^n,g(t))_{t\in I}$, where $I$ is a compact interval, be a Ricci flow with bounded curvature. Let $t_0-\tau$, $t_0\in I$, where $\tau>0$, and let $d\nu:=(4\pi\tau)^{-\frac{n}{2}}e^{-f}dg$ be the conjugate heat kernel based at $(x_0,t_0)$. Furthermore, assume $R(\cdot,t_0-\tau)\geq R_{\operatorname{min}}$, then we have
\begin{eqnarray}
\int_M\tau(|\nabla f|^2+R)d\nu_{t_0-\tau}&\leq&\frac{n}{2},\label{Nashintegral_1}
\\
\int_M\left(f-\mathcal{N}_{x_0,t_0}(\tau)-\frac{n}{2}\right)^2d\nu_{t_0-\tau}&\leq& n-2R_{\operatorname{min}}\tau.\label{Nashintegral_2}
\end{eqnarray}
\end{Proposition}

\begin{proof}
(\ref{Nashintegral_1}) follows from the fact that the Nash entropy is always no smaller than Perelman's entropy, which is always true when the Ricci flow has bounded curvature. (\ref{Nashintegral_2}) follows from (\ref{Nashintegral_1}) together with Hein-Naber's Poincar\'e inequality (\ref{ineq: poincare}). To apply the Poincar\'e inequality to the function $f-\mathcal{N}_{x_0,t_0}(\tau)-\frac{n}{2}$, let us first of all recall the coarse heat kernel Gaussian upper and lower bounds (c.f. Theorem 26.25 and Theorem 26.31 in \cite{RFV3}), that is,
\begin{eqnarray*}
C^{-1}\dist^2_{t_0-\tau}(x,x_0)-C\leq f(x,t_0-\tau)\leq C\dist^2_{t_0-\tau}(x_0,x)+C \quad \text{ for all }\quad x\in M,
\end{eqnarray*}
where $C$ is a constant depending only on the curvature bounds on $M\times[t_0-\tau,t_0]$, the value of $\tau$, and $\Vol_{g_{t_0}}\big(B_{g_{t_0}}(x_0,1)\big)$. It then follows from some straightforward computations that
\begin{align}\label{H-N-bound}
&\big|\,\mathcal{N}_{x_0,t_0}(\tau)\,\big|=\left|\,\int_Mfd\nu_{t_0-\tau}-\frac{n}{2}\,\right|<\infty,
\\\nonumber
& \int_M\left(f-\mathcal{N}_{x_0,t_0}(\tau)-\frac{n}{2}\right)^2d\nu_{t_0-\tau}<\infty.
\end{align}
Let us then fix an arbitrary number $A>0$, and let $\phi^A:M\rightarrow[0,1]$ be a smooth cut-off function such that $\phi^A\equiv 1$ on $B_{t_0-\tau}(x_0,A)$, $\phi^A\equiv 0$ on $M\setminus B_{t_0-\tau}(x_0,2A)$, and $|\nabla\phi^A|_{g_{t_0-\tau}}\leq 2A^{-1}$ everywhere on $M$. We may then apply the PoincaR\'e inequality (\ref{ineq: poincare}) to $u=\phi^A\left(f-\mathcal{N}_{x_0,t_0}(\tau)-\frac{n}{2}\right)$ and obtain
\begin{align}\label{H-N-f}
&\int_M\left(\phi^A\left(f-\mathcal{N}_{x_0,t_0}(\tau)-\frac{n}{2}\right)\right)^2d\nu_{t_0-\tau}-\left(\int_M\phi^A\left(f-\mathcal{N}_{x_0,t_0}(\tau)-\frac{n}{2}\right)d\nu_{t_0-\tau}\right)^2
\\\nonumber
&\leq2\tau\int_M\left|\nabla \left(\phi^A\left(f-\mathcal{N}_{x_0,t_0}(\tau)-\frac{n}{2}\right)\right)\right|^2d\nu_{t_0-\tau}.
\end{align}
Taking $A\rightarrow\infty$ and using (\ref{H-N-bound}), we have the following computations for all terms in formula (\ref{H-N-f}).
\begin{align*}
    &\int_M\left(\phi^A\left(f-\mathcal{N}_{x_0,t_0}(\tau)-\frac{n}{2}\right)\right)^2d\nu_{t_0-\tau}\rightarrow \int_M\left(f-\mathcal{N}_{x_0,t_0}(\tau)-\frac{n}{2}\right)^2d\nu_{t_0-\tau},
    \\
    &\int_M\phi^A\left(f-\mathcal{N}_{x_0,t_0}(\tau)-\frac{n}{2}\right)d\nu_{t_0-\tau}\rightarrow \int_M\left(f-\mathcal{N}_{x_0,t_0}(\tau)-\frac{n}{2}\right)d\nu_{t_0-\tau}=0,
    \\
    &\int_M\left|\nabla \left(\phi^A\left(f-\mathcal{N}_{x_0,t_0}(\tau)-\frac{n}{2}\right)\right)\right|^2d\nu_{t_0-\tau}=\int_M|\nabla\phi^A|^2\left(f-\mathcal{N}_{x_0,t_0}(\tau)-\frac{n}{2}\right)^2d\nu_{t_0-\tau}
    \\
    &\quad +\int_M(\phi^A)^2|\nabla f|^2d\nu_{t_0-\tau}+2\int_M\langle\nabla\phi^A,\nabla f\rangle\phi^A\left(f-\mathcal{N}_{x_0,t_0}(\tau)-\frac{n}{2}\right)d\nu_{t_0-\tau},
\end{align*}
and for the last three terms, we have
\begin{align*}
    &\int_M(\phi^A)^2|\nabla f|^2d\nu_{t_0-\tau}\rightarrow \int_M|\nabla f|^2d\nu_{t_0-\tau},
    \\
    &\int_M|\nabla\phi^A|^2\left(f-\mathcal{N}_{x_0,t_0}(\tau)-\frac{n}{2}\right)^2d\nu_{t_0-\tau}\leq\frac{4}{A^2}\int_M\left(f-\mathcal{N}_{x_0,t_0}(\tau)-\frac{n}{2}\right)^2d\nu_{t_0-\tau}\rightarrow 0
    \\
    &\left|\int_M\langle\nabla\phi^A,\nabla f\rangle\phi^A\left(f-\mathcal{N}_{x_0,t_0}(\tau)-\frac{n}{2}\right)d\nu_{t_0-\tau}\right|
    \\
    &\quad\quad \leq\frac{2}{A}\left(\int_M|\nabla f|^2d\nu_{t_0-\tau}\right)^{\frac{1}{2}}\left(\int_M\left(f-\mathcal{N}_{x_0,t_0}(\tau)-\frac{n}{2}\right)^2d\nu_{t_0-\tau}\right)^{\frac{1}{2}}\rightarrow 0.
\end{align*}  
Hence, (\ref{Nashintegral_2}) follows from (\ref{H-N-f}).
 \end{proof}

\section{A pseudolocality for the Nash entropy} 

In this section we prove Theorem  \ref{Coro_nash_2} and Theorem \ref{Thm_nash}. These results shall be important tools in the proof of Theorem \ref{Theorem_main}. We emphasize again that though the statements of Theorem \ref{Coro_nash_2} and Theorem \ref{Thm_nash} are closely related,  yet without a Nash entropy bound, there is no Gaussian upper bound for the conjugate heat kernel, and it is not clear whether the $H_n$-centers are close to the $\ell$-centers (compare with Proposition \ref{H_n_l_n}). Therefore, Theorem \ref{Coro_nash_2} and Theorem \ref{Thm_nash} do not imply each other. 

\begin{proof}[Proof of Theorem \ref{Coro_nash_2}] In the proof, we shall use the lower case letter $c$ to represent a positive estimate constant which is intuitively small, and the capital letter $C$ to represent a positive estimate constant which is intuitively large. These constants may vary from line to line, we nevertheless use the same letters in order not to cause notational complexity. Throughout the proof, we define 
\begin{eqnarray*}
r:=\sqrt{t-s},\quad s':=s-\tfrac{r^2}{9}.
\end{eqnarray*}

First of all, we state a consequence of the Bishop-Gromov comparison theorem, which will be used in the estimation of the local upper and lower bounds for the conjugate heat kernel.
\\

\noindent\textbf{Claim 1.} There are positive constants $c$ and $C$, depending only on $\alpha$ and $C_0$, such that the following hold for all $\displaystyle y\in B_s(z,\tfrac{r}{3})$.
\begin{gather}\label{claim1_1}
|{\Ric}|\leq C_0r^{-2} \quad \text{ on } \quad B_s(y,\tfrac{r}{3})\times [s',s],
\\
 \Vol_{g_{s}}\left(B_s(y,\tfrac{r}{3})\right)\geq c\left(\tfrac{r}{3}\right)^n. \label{claim1_2}
\end{gather}

\begin{proof}[Proof of Claim 1]
(\ref{claim1_1}) is almost a restatement of an assumption of the theorem. We will prove (\ref{claim1_2}) by the Bishop-Gromov comparison theorem. Fixing any point $y\in B_s(z,\frac{r}{3})$, since $\Ric\geq -C_0r^{-2}$ on $B_s(z,r)$, we may argue as follows. Note that all the volumes are computed using $g(s)$.
\begin{eqnarray*}
\frac{\Vol\big(B_s(y,\tfrac{r}{3})\big)}{(\tfrac{r}{3})^n}&\geq& c\frac{\Vol\big(B_s(y,\tfrac{2r}{3} )\big)}{(\tfrac{2r}{3})^n}\geq c\frac{\Vol\big(B_s(z,\tfrac{r}{3} )\big)}{(\tfrac{2r}{3} )^n}
\\
&=&2^{-n}c\frac{\Vol\big(B_s(z,\tfrac{r}{3} )\big)}{(\tfrac{r}{3} )^n}\geq c\frac{\Vol\big(B_s(z,r)\big)}{ r^n}\geq c,
\end{eqnarray*}
where the constant $c$ depends only on $\alpha$ and $C_0$. 
\end{proof}

Next, we shall obtain some estimates for the conjugate heat kernel $K(x,t\,|\,\cdot,\cdot)$. The local upper bound follows straightforwardly from the standard mean value inequality, and the local lower bound is a consequence of (\ref{subsolution}).
\\

\noindent\textbf{Claim 2.} There is a constant $C$, depending only on $\alpha$ and $C_0$, such that
\begin{eqnarray}
K\left(x,t\,\left|\,y,s'\right.\right)\leq Cr^{-n}\leq\frac{C}{(t-s')^{\tfrac{n}{2}}}\quad \text{ for all }\quad  y\in B_s(z,\tfrac{r}{3}).
\end{eqnarray}
\begin{proof}[Proof of Claim 2]
Let us fix an arbitrary $y\in B_s(z,\frac{r}{3})$. Since $K(x,t\,|\,\cdot,\cdot)$ is a conjugate heat kernel, we have
\begin{eqnarray*}
\int_{s-\frac{1}{9}r^2}^s\int_{B_s(y,\frac{r}{3})}K(x,t\,|\,\cdot,\eta)dg_\eta d\eta\leq \int_{s-\frac{1}{9}r^2}^s\int_{M}K(x,t\,|\,\cdot,\eta)dg_\eta d\eta=\int_{s-\frac{1}{9}r^2}^s1 d\eta=\frac{1}{9}r^2.
\end{eqnarray*}

On the other hand, the bounds provided by (\ref{claim1_1}) and (\ref{claim1_2}) are sufficient for the implementation of the standard mean value inequality (c.f. \cite[Theorem 25.2]{RFV3}) on the parabolic disk $B_s(y,\frac{r}{3})\times[s-\frac{1}{9}r^2,s]$. We then obtain that
\begin{eqnarray*}
K\left(x,t\,\left|\,y,s'\right. \right)&\leq&\frac{C}{\frac{1}{9}r^2\Vol_{g_{s}}\big(B_s(y,\frac{r}{3})\big)}\int_{s-\frac{1}{9}r^2}^s\int_{B_s(y,\frac{r}{3})}K(x,t|\cdot,\eta)dg_\eta d\eta
\\
&\leq&\frac{C}{\frac{1}{9}r^2\cdot c(\frac{r}{3})^n}\cdot\tfrac{1}{9}r^2\leq C r^{-n}.
\end{eqnarray*}
\end{proof}

\noindent\textbf{Claim 3:} There is a constant $c$, depending only on $\alpha$ and $C_0$, such that 
\begin{eqnarray*}
K(x,t\,|\,y,s')\geq cr^{-n}\geq\frac{c}{(t-s')^{\frac{n}{2}}}\quad\text{ for all }\quad y\in B_s(z,\tfrac{r}{3}).
\end{eqnarray*}

\begin{proof}[Proof of Claim 3]
Let us fix an arbitrary $y\in B_s(z,\frac{1}{3}r)$, and let $\gamma_2:[0,\frac{1}{9}r^2]\rightarrow B_s(z,\frac{1}{3}r)$ be the minimal constant-speed $g(s)$-geodesic, such that $\gamma_2(0)=z$ and $\gamma_2(\frac{1}{9}r^2)=y$. Then, obviously we have
\begin{eqnarray*}
|\gamma_2'(\tau)|^2_{g_{s}}\leq\left(\frac{r/3}{r^2/9}\right)^2= 9r^{-2}\quad \text{ for all }\quad \tau\in[0,\tfrac{1}{9}r^2].
\end{eqnarray*}
Therefore, by the Ricci curvature bound in (\ref{claim1_1}), we have
\begin{eqnarray}\label{maximusnonsense001}
|\gamma'_2(\tau)|^2_{g_{s-\tau}}\leq|\gamma'_2(\tau)|^2_{g_{s}}\exp\left(\int_0^{\tfrac{1}{9}r^2}\left(\sup_{B_s\left(z,\tfrac{1}{3}r\right)\times[s',s]}|\Ric|\right)d\eta\right)\leq Cr^{-2},
\end{eqnarray}
for all $\tau\in[0,\tfrac{1}{9}r^2]$. Let $\gamma_1:[0,r^2]\rightarrow M$ be a minimal $\mathcal{L}$-geodesic from $(x,t)$ to $(z,s)$. Then, we may use
\begin{eqnarray*}
\gamma(\tau):=\left\{\begin{array}{lcr}
     \gamma_1(\tau) &\text{if} &\tau\in[0,r^2], \\
     \gamma_2(\tau-r^2)&\text{if}&\tau\in[r^2,\tfrac{10}{9}r^2]. 
\end{array}\right.
\end{eqnarray*}
as a test curve in (\ref{definitionofl}) and estimate
\begin{eqnarray*}
\ell_{x,t}(y',t-s')&=&\ell_{x,t}\left(y',\tfrac{10}{9}r^2\right)
\\
&\leq&\frac{1}{2\sqrt{\tfrac{10}{9}r^2}}\bigg\{\int_0^{r^2}\sqrt{\tau}\left(|\gamma_1'(\tau)|^2_{g_{t-\tau}}+R_{g_{t-\tau}}(\gamma_1(\tau))\right)d\tau
\\
&&\quad\quad\quad +\int_{0}^{\tfrac{1}{9}r^2}\sqrt{\tau+r^2}\left(|\gamma_2'(\tau)|^2_{g_{s-\tau}}+R_{g_{s-\tau}}(\gamma_2(\tau))\right)d\tau\bigg\}
\\
&\leq&\frac{3}{2\sqrt{10}r}\left(2\sqrt{r^2}\ell_{x,t}(z,r^2)+\int_0^{\frac{1}{9}r^2}\sqrt{\tau+r^2}Cr^{-2}d\tau\right)
\\
&\leq&\frac{3}{2\sqrt{10}r}\left(nr+Cr\right)
\\
&\leq& C.
\end{eqnarray*}
Here we have used the facts that $\ell_{x,t}(z,t-s)\leq\frac{n}{2}$ and that $\displaystyle\sup_{\tau\in[0,\frac{1}{9}r^2]}\left|R_{g_{s-\tau}}(\gamma_2(\tau))\right|\leq C_0r^{-2}$, where the latter is because $\gamma_2\subset B_s(z,\frac{r}{3})$. Claim 3 then follows from (\ref{subsolution}).
\end{proof}

With all the preparations above, we are ready to estimate the Nash entropy. By (\ref{claim1_1}), we have
\begin{eqnarray*}
\Vol_{g_{s'}}\big(B_s(z,\tfrac{r}{3})\big)&\geq&\Vol_{g_{s}}\big(B_s(z,\tfrac{r}{3})\big)\exp\left(-\int_0^{\tfrac{1}{9}r^2}\left(\sup_{B_s(z,\tfrac{r}{3})\times[s',s]}|R|\right)d\eta\right)
\\\nonumber
&\geq&cr^n\geq c(t-s')^{\tfrac{n}{2}}.
\end{eqnarray*}
Combining this with Claim 3, we have
\begin{eqnarray}\label{P1}
\nu_{x,t\,|\,s'}\big(B_s(z,\tfrac{r}{3})\big)=\int_{B_s(z,\tfrac{r}{3})}K(x,t\,|\,\cdot,s')dg_{s'}\geq c.
\end{eqnarray}
Since $[s-r^2,s]\subset I$, we have, by the maximum principle, 
\begin{eqnarray}\label{P2}
R(\cdot,s')\geq-\frac{n}{2(s'-(s-r^2))}\geq-Cr^{-2}\geq-\frac{C}{t-s'}.
\end{eqnarray}
Let us define function $f$ by
\begin{eqnarray*}
K(x,t\,|\,\cdot,s'):=\frac{1}{(4\pi(t-s'))^{\frac{n}{2}}}e^{-f(\cdot,s')}.
\end{eqnarray*}
Then, Claim 2 and Claim 3 imply
\begin{eqnarray}\label{P3}
-C\leq f(\cdot,s')\leq C\quad\text{ on }\quad B_s(z,\tfrac{r}{3}).
\end{eqnarray}

With the estimates (\ref{P1})---(\ref{P3}), we may compute by using (\ref{Nashintegral_2})
\begin{align*}
    & \quad\quad n+2(t-s')\cdot\frac{C}{t-s'}
    \\
    &\geq\int_M\left(f(\cdot,s')-\mathcal{N}_{x,t}(t-s')-\tfrac{n}{2}\right)^2d\nu_{x,t\,|\,s'}
    \\\nonumber
    &\geq\int_{B_s(z,\tfrac{r}{3})}\left(f(\cdot,s')-\mathcal{N}_{x,t}(t-s')-\tfrac{n}{2}\right)^2d\nu_{x,t\,|\,s'}
    \\\nonumber
    &=\int_{B_s(z,\tfrac{r}{3})}\bigg(f^2(\cdot,s')-2f\left(\mathcal{N}_{x,t}(t-s')+\tfrac{n}{2}\right)+\left(\mathcal{N}_{x,t}(t-s')+\tfrac{n}{2}\right)^2\bigg)d\nu_{x,t\,|\,s'}
    \\\nonumber
    &\geq\int_{B_s(z,\tfrac{r}{3})}\left(-f^2(\cdot,s')+\tfrac{1}{2}\left(\mathcal{N}_{x,t}(t-s')+\tfrac{n}{2}\right)^2\right)d\nu_{x,t\,|\,s'}
    \\\nonumber
    &\geq\tfrac{1}{2}\left(\mathcal{N}_{x,t}(t-s')+\tfrac{n}{2}\right)^2\nu_{x,t\,|\,s'}\big(B_s(z,\tfrac{r}{3})\big)-\nu_{x,t\,|\,s'}\big(B_s(z,\tfrac{r}{3})\big)\sup_{B_s(z,\tfrac{r}{3})}f^2
    \\\nonumber
    &\geq c\left(\mathcal{N}_{x,t}(t-s')+\tfrac{n}{2}\right)^2-C.
\end{align*}
It then follows that
\begin{eqnarray*}
\mathcal{N}_{x,t}(t-s)\geq\mathcal{N}_{x,t}(t-s')\geq -\beta.
\end{eqnarray*}
This finishes the proof of the theorem.
\end{proof}

Before we proceed to prove Theorem \ref{Thm_nash}, a few preparatory results are needed. These results are all proved in \cite{Bam20a} for Ricci flows on closed manifolds, but they can be easily generalized to our case without much effort. 

\begin{Proposition}[Corollary 5.11 in \cite{Bam20a}; c.f. Corollary 4.5 in \cite{MZ21}]\label{harnackofnashentropy}
Let $(M^n,g(t))_{t\in I}$ be a complete Ricci flow with bounded curvature within each time interval compact in $I$. Assume $s$, $t^*$, $t_1$, and $t_2\in I$ with $s<t^*\leq t_1, t_2$. Moreover, assume $R(\cdot,t^*)\geq R_{\operatorname{min}}$. Then, for any $x_1$ and $x_2\in M$, we have
\begin{eqnarray}
\N^*_s(x_1,t_1)-\N_s^*(x_2,t_2)\leq\left(\frac{n}{2(t^*-s)}-R_{\operatorname{min}}\right)^{\frac{1}{2}}\dist_{W_1}^{g_{t^*}}(\nu_{x_1,t_1\,|\,t^*},\nu_{x_2,t_2\,|\,t^*})+\frac{n}{2}\log\left(\frac{t_2-s}{t^*-s}\right),
\end{eqnarray}
where we have defined $\N_s^*(x,t):=\N_{x,t}(t-s)$.
\end{Proposition}

\begin{Theorem}[Theorem 8.1 in \cite{Bam20a}]\label{localnoncollapsing}
Let $(M^n,g(t))_{t\in I}$ be a complete Ricci flow with bounded curvature within each time interval compact in $I$. Assume $[t-r^2,t]\subset I$ and $R\geq R_{\operatorname{min}}$ on $M\times[t-r^2,t]$. Then for all $A\in[1,\infty)$, we have
$$\Vol_{g_t}\left(B_{g_t}(x,Ar)\right)\leq C(R_{\operatorname{min}}r^2)\exp(\N_{x,t}(r^2))\exp(C_0A^2)r^n,$$ where $C_0$ is a constant depending only on the dimension. 
\end{Theorem}
\begin{proof}
In the original proof of \cite[Theorem 8.1]{Bam20a}, the results applied are \cite[Theorem 5.9]{Bam20a} and \cite[Theorem 7.5]{Bam20a}, and the former is already established in \cite[Theorem 4.4]{MZ21} under the assumption of bounded curvature within each compact time interval. For \cite[Theorem 7.5]{Bam20a}, the results applied in the original proof are \cite[Theorem 5.9]{Bam20a} (c.f. \cite[Theorem 4.4]{MZ21}), \cite[Corollary 5.11]{Bam20a} (c.f. \cite[Corollary 4.5]{MZ21}), the Gaussian concentration theorem (c.f. Proposition \ref{prop: gaussian concentration} above), \cite[Theorem 7.1]{Bam20a} (c.f. Theorem \ref{aCHKcoarseupperboundofbamler} above), and Proposition 4.2 (c.f. \cite[Proposition 3.4]{MZ21}). In conclusion, the proof in our case follows step by step after \cite[Theorem 8.1]{Bam20a}; all these results hold under our current assumption. Since the proofs of both \cite[Theorem 7.5]{Bam20a} and \cite[Theorem 8.1]{Bam20a} are short, we shall leave the verification to the readers.
\end{proof}

\begin{proof}[Proof of Theorem \ref{Thm_nash}]
Under the assumption of the theorem, we denote
$$r:=\sqrt{t-s}.$$
Since $s-r^2\in I$, by the maximum principle, we have $R(\cdot,s)\geq -\frac{n}{2r^2}$. Then, applying Theorem \ref{localnoncollapsing} to the point $(z,s)$ and $A=1$, we have
$$\alpha r^n\leq \operatorname{Vol}_{g_s}\big(B_s(z,r)\big)\leq C(-\tfrac{n}{2})\exp(\N_{z,s}(r^2))\exp(C_0)r^n,$$
that is,
\begin{eqnarray}\label{thecomparisonofthenashentropy}
\N^*_{s-r^2}(z,s)=\N_{z,s}(r^2)\geq -C_1,
\end{eqnarray}
where $C_1$ depends only on $\alpha$ and $n$.

Next, applying Proposition \ref{harnackofnashentropy} with $(z,s)\to (x_1,t_1)$, $(x,t)\to (x_2,t_2)$, $ s\to t^*$, $s-r^2\to s$, $-\frac{n}{2r^2}\to R_{\operatorname{min}}$, we have
\begin{eqnarray*}
\N_{s-r^2}^*(z,s)-\N_{s-r^2}^*(x,t)&\leq& \left(\frac{n}{r^2}\right)^{\frac{1}{2}}\dist_{W_1}^{g_s}(\nu_{x,t\,|\, s},\delta_z)+\frac{n}{2}\log\left(\frac{t-s+r^2}{s-s+r^2}\right)
\\
&\leq& \left(\frac{n}{r^2}\right)^{\frac{1}{2}}\sqrt{H_n(t-s)}+\frac{n}{2}\log\frac{2r^2}{r^2}
\\
&\leq& C(n),
\end{eqnarray*}
where we have applied the definition of the $H_n$-center and the fact $\dist_{W_1}^{g_s}(\nu_{x,t\,|\, s},\,\delta_z)\leq \sqrt{\operatorname{Var}(\nu_{x,t\,|\, s},\,\delta_z)}\leq\sqrt{H_n(t-s)}$ (c.f. \cite[Lemma 3.2]{Bam20a}). In combination with (\ref{thecomparisonofthenashentropy}), we have
$$\N_{x,t}(2r^2)=\N_{s-r^2}^*(x,t)\geq \N_{s-r^2}^*(z,s)-C(n)\geq -C_1-C(n).$$
This finishes the proof of the theorem.
\end{proof}

\section{Properties of ancient solutions satisfying Assumption B}

In this section, we prove several properties for an ancient solution satisfying (\ref{curvaturebound}) and Assumption B. We first of all verify that such an ancient solution must be locally uniformly Type I, which implies a shrinker structure on the blow-down limit in (\ref{smoothconvergence}). This fact together with Theorem \ref{Coro_nash_2} implies that the ancient solution has bounded Nash entropy. Finally, by using Theorem 3.1, we show that an $H_n$-center is always not far from an $\ell$-center.

Throughout this section, we let $(M,g(t))_{t\in(-\infty,0]}$ be an ancient solution satisfying (\ref{curvaturebound}) and Assumption B with respect to $\big(p_0,0,\{\tau_i\}_{i=1}^\infty,\{p_i\}_{i=1}^\infty\big)$. Let $(M_\infty,g_\infty(t))_{t\in[-2,-1]}$ be the limit Ricci flow in (\ref{smoothconvergence}). From the smooth convergence (\ref{smoothconvergence}), we may easily check that $(M,g(t))_{t\in(-\infty,0]}$ satisfies the locally uniformly Type I condition.

\begin{Proposition}\label{LUTypeI}
$(M,g(t))_{t\in(-\infty,0]}$ is locally uniformly Type I along the sequence $\{(p_i,-\tau_i)\}_{i=1}^\infty$.
\end{Proposition}
\begin{proof}
By the following Cheeger-Gromov-Hamilton convergence in the statement of Assumption B
\begin{eqnarray*}
\big(M,g_i(t),p_i\big)_{t\in[-2,-1]}\xrightarrow{\makebox[1cm]{}} \big(M_\infty,g_\infty(t),p_\infty\big)_{t\in[-2,-1]},
\end{eqnarray*}
where $g_i(t)=\tau_i^{-1}g(\tau_it)$, we may find $\varepsilon_i\searrow 0$, open sets $U_i\subset M_\infty$, and diffeomorphisms 
\begin{eqnarray}\label{CGH-COMVERGENCE-11111}
\Psi_i:M_\infty\supset U_i\rightarrow V_i\subset M,
\end{eqnarray}
such that the following hold
\begin{gather}
     U_i\supset B_{g_{\infty, -1}}(p_\infty, \varepsilon_i^{-1}) \quad \text{ and }\quad V_i=\Psi_i(U_i)\supset B_{g_{i, -1}}(p_i, \varepsilon_i^{-1}),
    \\
    \Psi_i(p_\infty)=p_i,\label{base_B}
    \\
    \big\|\Psi_i^*g_i-g_\infty\big\|_{C^{[\varepsilon_i^{-1}]}(U_i\times[-2,-1])}\leq \varepsilon_i.\label{norm}
\end{gather}
As a consequence of (\ref{base_B}) and (\ref{norm}), we can find a positive function $C:(0,\infty)\rightarrow(0,\infty)$, such that
\begin{gather}\label{supranonsense008}
    \lim_{i\rightarrow\infty}\operatorname{inj}_{g_{i, -1}}(p_i)=\operatorname{inj}_{g_{\infty, -1}}(p_\infty)>0,\\
    \lim_{i\rightarrow\infty} \sup_{B_{g_{i, -1}}(p_i,r)\times[-2,-1]}|\Rm_{g_i}|=\sup_{B_{g_{\infty, -1}}(p_\infty,r)\times[-2,-1]}|\Rm_{g_{\infty, -1}}|\leq C(r)<\infty,\label{supranonsense009}
\end{gather}
where the latter holds for all $r>0$. Scaling (\ref{supranonsense008}) and (\ref{supranonsense009}) with factor $\tau_i$, we have verified Definition \ref{def}(2)(4) for $(M,g(t))$. Furthermore, by (\ref{curvaturebound}), we have a time-wise Ricci curvature lower bound for $g(t)$. Lastly, since $(p_i,-\tau_i)$'s are $\ell$-centers of $(p_0,0)$, we also have $\ell_{p_0,0}(p_i,\tau_i)\leq\frac{n}{2}$ for all $i$. Therefore, $(M,g(t))$ is locally uniformly Type I along the sequence $\{(p_i,-\tau_i)\}_{i=1}^\infty$.
\end{proof}

The following proposition and corollary are but restatements of Proposition \ref{TIshrinker} and Corollary \ref{noncollapsing}, respectively. Theorem \ref{Theorem_main}(1) follows from Proposition \ref{shrinkerstructure}.

\begin{Proposition}\label{shrinkerstructure}
The limit $(M_\infty,g_\infty(t))_{t\in[-2,-1]}$ has a shrinker structure. More precisely, $\ell_i\rightarrow\ell_\infty$ in the $C^\alpha_{\operatorname{loc}}$ sense or in the weak $*W^{1,2}_{\operatorname{loc}}$ sense on $M_\infty\times(1,2)$, where $\ell_i(\cdot,\tau)=\ell(\cdot,\tau_i\tau):M\rightarrow\mathbb{R}$ for $\tau\in[1,2]$,  $\ell$ is the reduced distance based at $(p_0,0)$ with respect to the Ricci flow $(M,g(t))$, and $\ell_\infty$ is a smooth shrinker potential satisfying
\begin{eqnarray*}
\Ric_{g_\infty}(-\tau)+\nabla^2\ell_\infty(\cdot,\tau)=\frac{1}{2\tau}g_\infty(-\tau)\quad\text{ for all }\quad \tau\in(1,2).
\end{eqnarray*}
\end{Proposition}

\begin{Corollary}\label{LUTypeInoncollapsed}
$(M,g(t))_{t\in(-\infty,0)}$ is weakly $\kappa$-noncollapsed on all scales. As a consequence, $(M,g(t))_{t\in(-\infty,0]}$ has bounded geometry in each compact time-interval.
\end{Corollary}

\textbf{Remark:} From this point on, any ancient solution satisfying (\ref{curvaturebound}) and Assumption B is understood to have bounded geometry within each compact time interval.

\begin{Theorem}\label{entropybound}
The Nash entropy of $(M,g(t))_{t\in(-\infty, 0]}$ based at $(p_0,0)$ is bounded independent of time, that is,
\begin{eqnarray*}
\mathcal{N}_{p_0,0}(\tau)\geq -Y\quad\text{ for all }\quad\tau>0,
\end{eqnarray*}
where $Y\in(0,\infty)$ depends on the local geometry bounds around the $\ell$-centers $(p_i,-\tau_i)$.
\end{Theorem}

\begin{proof}
By the assumption (\ref{curvaturebound}), we have that each $(M,g_i(t))_{t\in[-2,0]}$ is a complete Ricci flow with bounded curvature, where $g_i(t):=\tau_i^{-1}g(\tau_it)$. Since $(p_i,-1)$ is an $\ell$-center of $(p_0,0)$ with respect to the Ricci flow $(M,g_i(t))$, and since (\ref{supranonsense008}) and (\ref{supranonsense009}) imply that
\begin{gather*}
\big|\Rm_{g_i}\big|\leq C\quad\text{ on }\quad B_{g_{i,-1}}(p_i,1)\times[-2,-1],
\\
\Vol_{g_{i,-1}}(p_i,1)\geq c,
\end{gather*}
where $c$ and $C$ are positive constants independent of $i$, we have, by Theorem \ref{Coro_nash_2},
\begin{eqnarray*}
\mathcal{N}^i_{p_0,0}(1)\geq -Y \quad\text{ for all }\quad i\in\mathbb{N},
\end{eqnarray*}
where $\mathcal{N}^i$ is the Nash entropy of the Ricci flow $(M,g_i(t))$ and $Y$ depends only on $c$ and $C$. Therefore, by the scaling property of the Nash entropy, we have $$\mathcal{N}_{p_0,0}(\tau_i)\geq -Y\quad\text{ for all }\quad i,$$and the theorem follows from the monotonicity of $\mathcal{N}$.
\end{proof}

Next, we show that the $H_n$-centers and the $\ell$-centers are not far from each other, and consequently the base points in (\ref{smoothconvergence}) can be replaced by  $H_n$-centers. First of all, the following result is merely a consequence of Theorem \ref{gaussianupperbound} and Theorem \ref{entropybound}.

\begin{Proposition}\label{gauss0}
There exists a positive number $C$ depending only on the number $Y$ in the conclusion of Theorem \ref{entropybound}, such that
\begin{eqnarray}\label{gauss00}
K(p_0,0\,|\,x,t)\leq\frac{C}{|t|^{\frac{n}{2}}}\exp\left(-\frac{\dist_{t}^2(x,z)}{C|t|}\right)\quad\text{ for all }\quad x\in M\quad\text{ and }\quad t\in (-\infty,0),
\end{eqnarray}
where $(z,t)$ is an $H_n$-center of $(p_0,0)$.
\end{Proposition}

Let us find another sequence of points $\{z_i\}_{i=0}^\infty\subset M$, such that each $(z_i,-\tau_i)$ is an $H_n$-centers of $(p_0,0)$. For $H_n$-concentrated Ricci flows (this category obviously contains all Ricci flows with bounded curvature 
in compact time intervals; see the remark below Proposition \ref{measureaccumulationofHcenter}), 
one may always find an $H_n$-center at each time for any base point. We then have the following estimate.

\begin{Proposition} \label{H_n_l_n}
There is a constant $C$ depending on the number Y in the conclusion of Theorem \ref{entropybound}, such that the following holds.
\begin{eqnarray}\label{boundeddist}
\dist_{-\tau_i}(p_i,z_i)\leq C\sqrt{\tau_i} \quad\text{ for all }\quad i.
\end{eqnarray}
\end{Proposition}

\textbf{Remark:} Indeed, the same argument of the proof below shows that, for any $t<0$, we have $$\dist_t(p,z)\leq C\sqrt{|t|},$$where $(p,t)$ and $(z,t)$ are an $\ell$-center and an $H_n$-center of $(p_0,0)$, respectively.

\begin{proof}
Combining (\ref{subsolution}) and (\ref{gauss00}), we have
\begin{align*}
   &\frac{C}{\tau_i^{\frac{n}{2}}}\exp\left(-\frac{\dist_{-\tau_i}^2(p_i,z_i)}{C\tau_i}\right) \geq  K(p_0,0\,|\,p_i,-\tau_i)\geq\frac{1}{(4\pi\tau_i)^{\frac{n}{2}}}e^{-\ell_{p_0,0}(p_i,\tau_i)}\geq \frac{1}{(4\pi\tau_i)^{\frac{n}{2}}}e^{-\frac{n}{2}},
\end{align*}
where we have used the fact that $\ell_{p_0,0}(p_i,\tau_i)\leq\frac{n}{2}$. It then follows that
\begin{eqnarray*}
\dist_{-\tau_i}^2(p_i,z_i)\leq C\tau_i\quad\text{ for all }\quad i.
\end{eqnarray*}
\end{proof}

\begin{Corollary}\label{change-l-center-to-Hn-center}
In the convergence in (\ref{smoothconvergence}), possibly after passing to a subsequence, one may replace $p_i$ by $z_i$. If so, then the limit Ricci flow is still the same, whose base point is correspondingly replace by some $z_\infty$ satisfying $\dist_{g_{\infty, -1}}(p_\infty,z_\infty)\leq C$, where $C$ is the constant in (\ref{boundeddist}).
\end{Corollary}

\begin{proof}
After the parabolic scaling $g_i(t)=\tau_i^{-1}g(\tau_it)$, (\ref{boundeddist}) becomes
\begin{eqnarray*}
\dist_{g_{i,-1}}(p_i,z_i)\leq C\quad\text{ for all }\quad i.
\end{eqnarray*}
The corollary then follows from the definition of Cheeger-Gromov-Hamilton convergence.
\end{proof}

\section{Smooth convergence and $\mathbb{F}$-convergence}

For the purpose of proving Theorem \ref{Theorem_main}(3), we shall prove a more general statement in this section, namely, that from a sequence of complete smooth Ricci flows, if, using a sequence of $H_n$-centers as base points, one obtains a smooth Cheeger-Gromov-Hamilton limit, then, this smooth limit is the same as the $\mathbb{F}$-limit.

Let $\{(M_i,g_i(t))_{t\in(-T_i,0]}\}_{i=1}^\infty$ be a sequence of complete Ricci flows,  each one with bounded curvature within each compact time interval, where $\infty\geq T_i>c>0$ for some constant $c$. For each $i$, let $p_i\in M_i$ be a fixed point and $(\nu^i_{t})_{t\in(-T_i,0]}$ the conjugate heat kernel on $(M_i,g_i)$ based at $(p_i,0)$. According to \cite[Theorem 7.4, Corollary 7.5]{Bam20b}, 
$\displaystyle\left\{\big(M_i,g_i(t))_{t\in(-T_i,0]},(\nu^i_t)_{t\in(-T_i,0]}\big)\right\}_{i=1}^\infty$ has an $\mathbb{F}$-convergent  subsequence. 

\begin{Theorem}\label{convergence}
Assume that there exist a
compact
interval $I=[a,b]\subset (-T_\infty,0)$, where $T_\infty:=\limsup_{i\rightarrow\infty }T_i$, and a smooth Ricci flow $(M_\infty,g_\infty(t),z_\infty)_{t\in I}$, such that 
\begin{eqnarray}\label{smoothconvergence1}
\left(M_i,g_i(t),z_i\right)_{t\in I}\xrightarrow{\makebox[1cm]{}}\left(M_\infty,g_\infty(t),z_\infty\right)_{t\in I}
\end{eqnarray}
in the smooth Cheeger-Gromov-Hamilton sense, where $(z_i,b)$ is an $H_n$-center of $(p_i,0)$ for each $i\in\mathbb{N}$. Then $(M_\infty,g_\infty(t))_{t\in [a,b]}$ induces an $H_n$-concentrated continuous metric flow $\XX^\infty$ and
there is a conjugate heat flow 
$(\nu_t^\infty)_{t\in [a,b)}$ on $\XX^\infty$, such that, by passing to a subsequence, we have
\begin{eqnarray}\label{theequivalentFconvergence}
\left((M_i,g_i(t))_{t\in [a,b]},(\nu^i_t)_{t\in [a,b]}\right)\xrightarrow{\makebox[1cm]{$\mathbb{F}$}}\left(\XX^\infty,(\nu_t^\infty)_{t\in [a,b)}\right),
\end{eqnarray}
where the convergence is uniform over any compact sub-interval of $[a,b).$

\end{Theorem}

 \textbf{Remarks:}\begin{enumerate}
     \item This theorem implies that, for any subsequence of $\left\{\left((M_i,g_i(t))_{t\in [a,b]},(\nu^i_t)_{t\in [a,b]}\right)\right\}_{i=1}^\infty$, there is a further subsequence that $\mathbb{F}$-converges to $\left(\XX^\infty,(\nu_t^\infty)_{t\in [a,b)}\right)$, where $\nu_t^\infty$ is a conjugate heat flow on $\XX^\infty$.
So any continuous metric flow pair $(\mathcal{Y}^\infty, (\mu_t)_{t\in [a,b)})$ that arises as an $\mathbb{F}$-limit given by the compactness theorem \cite[Theorem 7.6]{Bam20b} should be of the form $\left(\XX^\infty, (\nu_t^\infty)_{t\in [a,b)}\right)$. However, this does not imply that the $\mathbb{F}$-convergence in $(\ref{theequivalentFconvergence})$ holds without passing to a subsequence, since we are not able to show that $\nu_t^\infty$ is independent of the subsequence.
\item We may replace the $H_n$-centers $(z_i,b)$ in the assumptions with $\ell$-centers and the conclusions still hold. This is because $H_n$-centers are not far away from $\ell$-centers by the arguments in Proposition \ref{H_n_l_n}. This point will become clear from the proof.

 \end{enumerate}

Throughout this section, we assume that the conditions of Theorem \ref{convergence} hold. Just as we have done in (\ref{CGH-COMVERGENCE-11111})---(\ref{norm}), by the definition of smooth convergence, we may find an increasing sequence of pre-compact open sets $U_i\subset M_\infty$ with $\cup_{i=1}^\infty U_i=M_{\infty}$ and a sequence of diffeomorphisms $\Psi_i:U_i\to V_i\subset M_i$,
such that $\Psi_i(z_\infty)=z_i$ and
\begin{eqnarray}\label{ultranonsense1}
    \|\Psi_i^*g_i - g_\infty\|_{C^{[\varepsilon_i^{-1}]}(U_i\times [a,b])}
    < \varepsilon_i,
\end{eqnarray}
for some $\varepsilon_i\searrow 0.$

The proof of Theorem \ref{convergence} is divided into several components. We shall first of all show that the smooth limit flow is indeed a metric flow. This is not as obvious as it appears to be, since we do not make any geometric assumption for $(M_\infty,g_\infty(t))_{t\in [a,b]}$ except for its smoothness and completeness. The idea is to construct a metric flow structure using the converging sequence.

\begin{Lemma}
\label{lem: limit is metric flow}
 $(M_\infty,g_\infty(t))_{t\in [a,b]}$ induces an $H_n$-concentrated continuous metric flow $\XX^\infty$.
\end{Lemma}
\begin{proof}

As mentioned in Theorem \ref{thm: smooth flows are metric flows}, 
each $(M_i,g_{i}(t))$ induces an $H_n$-concentrated metric flow (see also the remark below Proposition \ref{measureaccumulationofHcenter}). Let us then verify \cite[Definition 3.2]{Bam20b} and the $H_n$-concentrarion condition for $(M_\infty,g_{\infty}(t))_{t\in I}$.

The conjugate heat kernel $K_\infty(\cdot,\cdot\,|\,\cdot,\cdot)$ on $(M_\infty,g_{\infty}(t))_{t\in I}$ is given by Theorem \ref{thm: HK conv under CGH}. Furthermore, by Theorem \ref{thepropertiesoftheCHK}, if we define $d\nu^\infty_{x,t\,|\,s}=K(x,t\,|\,\cdot,s)dg_{\infty,s}$, then we have that $\nu^\infty_{x,t\,|\,s}\in\PP(M_\infty)$ and \cite[Definition 3.2(1)---(5),(7)]{Bam20b} holds.

To verify that $\left(M_\infty,g_\infty(t),\nu^\infty_{x,t\,|\,s}\right)$ satisfies \cite[Definition 3.2(6)]{Bam20b}, we shall apply \cite[Lemma 3.8]{Bam20b}. Let us fix $s, t\in I$ with $s<t$, a positive number $T$, and a measurable function $u: M_\infty\rightarrow (0,1)$ such that $f:=\Phi^{-1}\circ u$ is $T^{-\frac{1}{2}}$-Lipschitz with respect to the metric $g_\infty(s)$, where $\Phi:\mathbb{R}\rightarrow(0,1)$ is defined as in \cite[(3.1)]{Bam20b}. Let $\eta:[0,\infty)\rightarrow[0,1]$ be the standard cut-off function such that $\eta\equiv 1$ on $[0,1]$, $\eta\equiv 0$ on $[2,\infty)$, and $0\geq\eta'\geq -2$ everywhere. Let $\zeta:\mathbb{R}\rightarrow[-1,1]$ be the function satisfying $\zeta(t)=t$ for $|t|<1$, $\zeta|_{[1,\infty)}\equiv1$, and $\zeta|_{(-\infty,-1]}\equiv-1.$ For any $A>0$, we define
\begin{gather*}
f^A:=\eta\left(\tfrac{1}{A^2}\dist_{g_{\infty,s}}(z_\infty,\cdot)\right)\cdot A\zeta\left(\tfrac{f}{A}\right),
\\
u^A:=\Phi\circ f^A.
\end{gather*}
Then we have
\begin{gather}\label{supernonsense_001}
u^A\in[\Phi(-A),\Phi(A)]\subset(0,1)\quad \text{ everywhere on }\quad M_\infty,
\\
\big|\,u-u^A\,\big|\leq 1-\Phi(A)=\Phi(-A)\quad\text{ on }\quad B_{g_{\infty,s}}(z_\infty,A^2),\label{supernonsense_002}
\\
u^A\equiv 1/2\quad\text{ on }\quad M_\infty\setminus B_{g_{\infty,s}}(z_\infty,2A^2),\label{supernonsense_003}
\\
\operatorname{Lip}f^A\leq \tfrac{2}{A^2}\sup\left|A\zeta\left(\tfrac{f}{A}\right)\right|
+\sup|\eta|\cdot A\operatorname{Lip}\zeta \cdot\operatorname{Lip} \tfrac{f}{A}
\leq T^{-\frac{1}{2}}+\tfrac{2}{A}.\label{supernonsense_004}
\end{gather}
Let us fix two points $x, y\in M_\infty$. By (\ref{supernonsense_001}),  (\ref{supernonsense_002}), and the bounded convergence theorem, we have 
\begin{eqnarray}\label{supernonsense006}
\lim_{A\rightarrow\infty}u^A_t(x),\ \lim_{A\rightarrow\infty}u^A_t(y)= u_t(x),\ u_t(y),
\end{eqnarray}
where$$u_t^A(\cdot):=\int_{M_\infty}u^A
\,d\nu^\infty_{\cdot,t\,|\, s},\quad u_t(\cdot):=\int_{M_\infty}u\,d\nu^\infty_{\cdot,t\,|\, s}.$$
In view of (\ref{supernonsense_003}), for all $i$ large enough we may define
\begin{eqnarray*}
    f_i^A&=&f^A\circ\Psi_i^{-1}\quad\text{ on }\quad V_i=\Psi_i(U_i)
    \\
    f_i^A&\equiv&
    0
    \quad\text{ on }\quad M_i\setminus V_i.
    \\
    u_i^A&:=&\Phi\circ f_i^A.
\end{eqnarray*}
Then, by the smooth convergence (\ref{ultranonsense1}), we have
\begin{eqnarray*}
\operatorname{Lip}f_i^A\leq T^{-\frac{1}{2}}+\frac{2}{A}+\varepsilon_i, 
\end{eqnarray*}
and consequently (c.f. \cite[Theorem 3.1]{MZ21}), 
\begin{align}\label{supernonsense005}
&\Big|\,\Phi^{-1}\circ u^A_{i,t}(\Psi_i(x))-\Phi^{-1}\circ u^A_{i,t}(\Psi_i(y))\,\Big|
\\\nonumber
&\quad\quad\leq \left(\left(T^{-\frac{1}{2}}+\tfrac{2}{A}+\varepsilon_i\right)^{2}+t-s\right)^{-\frac{1}{2}}\dist_{g_{i,t}}\big(\Psi_i(x),\Psi_i(y)\big),
\end{align}
for all $i$ large enough, where  $u_{i,t}^A:=\int_{M_i}u_i^Ad\nu^i_{\cdot,t\,|\,s}$. In view of (\ref{supernonsense_001}) and Proposition \ref{measureaccumulationofHcenter}, the integral of $\int_{M_i}u_i^Ad\nu^i_{\cdot,t\,|\,s}$ is uniformly negligible outside a large disk. Then, by the locally smoothly convergence of the conjugate heat kernels, we have (c.f. the proof of (\ref{someimportantconvergenceargument})) $$\lim_{i\rightarrow\infty}u^A_{i,t}(\Psi_i(x)),\ \lim_{i\rightarrow\infty}u^A_{i,t}(\Psi_i(y))=u_t^A(x),\ u_t^A(y).$$
Hence, (\ref{supernonsense005}) implies that
\begin{eqnarray*}
\big|\Phi^{-1}\circ u_t^A(x)-\Phi^{-1}\circ u_t^A(y)\big|\leq \left(\left(T^{-\frac{1}{2}}+\tfrac{2}{A}\right)^{2}+t-s\right)^{-\frac{1}{2}}\dist_{g_{\infty,t}}(x,y).
\end{eqnarray*}
    Therefore, by (\ref{supernonsense006}) and taking $A\rightarrow\infty$, we have $$\big|\Phi^{-1}\circ u_t(x)-\Phi^{-1}\circ u_t(y)\big|\leq \left(T+t-s\right)^{-\frac{1}{2}}\dist_{g_{\infty,t}}(x,y).$$
Since $x,y\in M_\infty$ are arbitrary, we have verified \cite[Definition 3.2(6)]{Bam20b}.

Finally, it is also easy to verify the $H_n$-concentration property for $\XX^\infty$, since it is but a consequence of Fatou's lemma.
\end{proof}

\begin{Lemma}
\label{lem: dnu^oo_s}
There is a positive solution to the conjugate heat equation $v:M_\infty\times [a,b)\rightarrow\mathbb{R}$ coupled with $(M_\infty,g_{\infty}(t))$, satisfying  $d\nu^\infty_s := v_s \, dg_{\infty,s}\in \PP(M_\infty)$ and
\[
    \Psi_i^* K_i(p_i,0\,|\, \cdot,\cdot)
    \to v
\]
locally smoothly on $M_\infty\times [a,b).$
\end{Lemma}

\begin{proof}
Arguing in the same way as the proof of \cite[Theorem 2.1]{Lu12}, we can find a nonnegative solution $v:M_\infty\times[a,b)\rightarrow\mathbb{R}$ to the conjugate heat equation, such that
\begin{eqnarray}\label{supernonsense009}
    \Psi_i^*K_i(p_i,0\,|\, \cdot, \cdot)
    \to v
\end{eqnarray}
locally smoothly on $M_\infty\times [a,b)$.

In this case, Theorem \ref{thm: HK conv under CGH} does not imply that $\int v_s \, dg_{\infty,s}=1$ as it did in the proof of formula (\ref{theintegralisequaltoone}). This is because the base point $(p_i,0)$ of the conjugate heat kernel $K_i(p_i,0\,|\,\cdot,\cdot)$ is not in the region of the Cheeger-Gromov-Hamilton convergence. We first of all observe from (\ref{ultranonsense1}) that, there is a positive function $C:(0,\infty)\rightarrow(0,\infty)$ with the following property: for any $r>0$, we have 
\begin{eqnarray}\label{supranonsense001}
\left|\Rm_{g_{i}}\right|\leq C(r)\quad\text{ on }\quad B_{g_{i,b}}(z_i,r)\times[a,b]
\end{eqnarray}
whenever $i$ is large enough. Since we also have \begin{gather*}
    \liminf_{i\rightarrow\infty}\Vol_{g_{i,b}}\big(B_{g_{i,b}}(z_i,1)\big)=\Vol_{g_{\infty,b}}\big(B_{g_{\infty,b}}(z_\infty,1)\big)>0,
\end{gather*}
we may apply Theorem \ref{Thm_nash} and the remark below it to the $H_n$-center $(z_i,b)$ of $(p_i,0)$, and consequently, we can find a positive number $Y$ independent of $i$, such that
\begin{eqnarray*}
\mathcal{N}^i_{p_i,0}(|b|)\geq -Y \quad\text{ for all }\quad i\in\mathbb{N},
\end{eqnarray*}
where $\mathcal{N}^i$ is the Nash entropy of the Ricci flow $(M_i,g_i(t))$. The argument in Proposition \ref{H_n_l_n} then implies that
\begin{eqnarray}\label{supranonsense002}
\dist_{g_{i,b}}(z_i,p'_i)\leq C\quad\text{ for all }\quad i\in\mathbb{N},
\end{eqnarray}
where $(p'_i,b)$ is an $\ell$-center of $(p_i,0)$, and $C$ is a constant depending only on $Y$. 

Now we will use (\ref{subsolution}) to obtain a uniform lower bound for $K_i(p_i,0\,|\,p'_i,b-\varepsilon)$, where $\varepsilon$ is an arbitrarily fixed number in $(0,b-a]$. Let us fix such an $\varepsilon$. (\ref{supranonsense001}) and (\ref{supranonsense002}) imply 
$$\sup_{t\in
[a,b]
}\left|R_{g_i}(p'_i,t)\right|\leq C\quad\text{ for all }i\in\mathbb{N},$$
where $C$ is a constant independent of $i$. 
We may concatenate a minimal $\mathcal{L}$-geodesic from $(p_i,0)$ to $(p'_i,b)$ and 
the static curve from $(p'_i,b)$
to $(p'_i,b-\varepsilon)$ as a test curve in (\ref{definitionofl}). This yields
\begin{eqnarray*}
\ell_{p_i,0}(p'_i,|b|+\varepsilon)&\leq&\frac{1}{2\sqrt{|b|+\varepsilon}}\left(2\sqrt{|b|}\ell_{p_i,0}(p'_i,|b|)
+ \int_{|b|}^{|b|+\varepsilon} 
\sqrt{\tau}  R_{g_i}(p'_i,-\tau)\,d\tau
\right)
\\
&\leq&\frac{1}{2\sqrt{|b|+\varepsilon}}\left(2\sqrt{|b|}\cdot\frac{n}{2}
+C(|b|+\varepsilon)^{3/2}
\right)
\\
&\leq& C\quad\text{ for all }\quad i\in\mathbb{N}.
\end{eqnarray*}
Hence, by (\ref{subsolution}), we have
$$K_i(p_i,0\,|\,p'_i,b-\varepsilon)\geq\frac{1}{4\pi(|b|+\varepsilon)^{\frac{n}{2}}}e^{-\ell_{p_i,0}(p'_i,|b|+\varepsilon)}\geq c(\varepsilon)>0\quad\text{ for all }\quad i\in\mathbb{N}.$$
By (\ref{supranonsense002}) again, we may find a point $p'_\infty\in M_\infty$, such that $\Psi_i^{-1}(p'_i)\rightarrow p'_\infty$ after possibly passing to a subsequence. Consequently
$$v(p'_\infty,b-\varepsilon)=\lim_{i\rightarrow\infty}K_i(p_i,0\,|\,p'_i,b-\varepsilon)\geq c(\varepsilon)>0.$$
Since $\varepsilon\in(0,b-a]$ is arbitrary, it then follows from the strong maximum principle that $v>0$ everywhere on $M_\infty\times[a,b)$.

 Once we know that $v$ is positive, by exactly the same argument as in the proof of (\ref{theintegralisequaltoone}), we have
\[
    \int_{M_\infty} v_{s}\, dg_{\infty,s}
    =1,
\]
for any $s\in [a,b).$
Indeed, here one may use the fact that $$\lim_{i\rightarrow\infty} \int_{B_{g_{i,s}}(z_i,1)}K_i(p_i,0\,|\,\cdot,s)dg_{i,s}=\int_{B_{g_{\infty,s}}(z_\infty,1)}v(\cdot,s)dg_{\infty,s}>0$$ to show that $(z_i,s)$ is not far from an $H_n$-center of $(p_i,0)$, for any $s\in[a,b)$, and the rest of the argument is the same as the proof of (\ref{theintegralisequaltoone}).
\end{proof}

\begin{Lemma}
\label{lem: compact pyramid}
Let $(M^n,g(t))_{t\in I}$ be a complete Ricci flow over a compact interval $I$ and $o\in M$. 
Then for any $A>0,\Omega:=\cup_{t\in I} \bar B_t(o,A)$ is compact.
\end{Lemma}
\begin{proof}
Let $x_j\in \bar B_{s_j}(o,A)\subset \Omega$ be an arbitrary sequence. Assume that $s_j\to \bar s\in I.$ We shall prove that $\{x_j\}$ has a convergent subsequence.
Suppose that
\[
\sup_{B_{\bar s}(o,10A)\times I} |{\Ric}| \le \Lambda.
\]
We claim that for any $\epsilon\in (0,1),$ there is $\bar j,$ such that if $j\ge \bar j,$ $x_j\in B_{\bar s}(o,(1+\epsilon)A)$. 
Suppose not. Then by passing to a subsequence, we may assume that there are $g(s_j)$-minimal geodesics $\gamma_j:[0,\sigma_j]\to M$ with 
$\gamma_j(0)=o,\gamma_j(\sigma_j)=x_j,\sigma_j<A$ but there is a first time $\lambda_j<\sigma_j$ such that $\gamma_j(\lambda_j)\in \partial B_{\bar s}(o,(1+\epsilon/2)A).$
Pick $\delta>0$ such that $e^{-\Lambda \delta} > 1- \epsilon/4.$ 
When $|s_j-\bar s|< \delta$, we have
\begin{align*}
    A\ge L_{s_j}(\gamma_j|_{[0,\lambda_j]})
    &\ge e^{-\Lambda |s_j-\bar s|}
    L_{\bar s}(\gamma_j|_{[0,\lambda_j]})
    \ge (1 - \epsilon/4) \dist_{\bar s}(o,\gamma_j(\lambda_j))\\
    &=  (1 - \epsilon/4)(1+\epsilon/2)A
    > (1+ \epsilon/8)A,
\end{align*}
which is a contradiction. Hence, after passing to a subsequence, we have $x_j\to \bar x$ for some $\bar x\in \bar B_{\bar s}(o,A).$
\end{proof}

We are now ready to prove Theorem \ref{convergence}.

\begin{proof}[Proof of Theorem \ref{convergence}] 
We divide the  proof into several steps.\\

\noindent
\textbf{Step 1:} Construction of the correspondence. For $t\in I$, set $Z^i_t:=M_i\sqcup M_\infty$ and we shall extend the metrics on $(M_i,g_i(t))$ and $(M_\infty, g_\infty(t))$ to $Z^i_t.$ For any $y_i\in M_i,y\in M_\infty,$ define
\[
    \dist^{Z^i_t}(y,y_i)
    = \dist^{Z^i_t}(y_i,y)
    := \inf_{w\in U_i} \dist_{g_{\infty,t}}(y,w)
    + \dist_{g_{i,t}}(\Psi_i(w),y_i)
    + \varepsilon_i.
\]
It is routine to verify that this is indeed a metric, and $M_i, M_\infty\rightarrow M_i\sqcup M_\infty$ are isometric embeddings.
By \cite[Lemma 2.13]{Bam20b}, we may assume that $Z_t^i$ are isometrically embedded into a common metric space $Z_t$ that is complete and separable.
Let $\phi^i_t:(M_i,g_i(t))\to Z_t$ be the isometric embedding for $i\in \mathbb{N}\cup\{\infty\}.$
Note that for any $x\in U_i,$
\[
    \dist^{Z_t}(\phi_t^\infty(x),
    \phi_t^i(\Psi_i(x))) = \varepsilon_i.
\]
\bigskip

\noindent
\textbf{Step 2:} Construction of the couplings. Henceforth until the end of the proof of the theorem, we shall fix an arbitrarily small $\varepsilon >0$ and denote $E=(b-\varepsilon^2,b].$ By Lemma \ref{lem: dnu^oo_s}, there is a conjugate heat flow
\[
    d\nu^\infty_s := v_s\, dg_{\infty,s},
\]
where $\nu^\infty_s\in \mathcal{P}(M_\infty)$ for $s\in [a,b)$, and $\Psi_i^*\nu_s^i\to \nu_s^\infty$ on $M_\infty\times[a,b)$ in the $C_c^\infty$-topology as smooth $n$-forms.
\\

\noindent
\textbf{Claim 1:}
For $i=\infty$ or for all $i$ sufficiently large, if $s\in I\setminus E$ and $r\ge 10\sqrt{|a|}$, then
\begin{equation}
\label{ineq: gaussian decay}
    \nu^i_s(M_i\setminus B_{g_{i,s}}(z_i,r))
    \le Ce^{-cr^2},
\end{equation}
where $c$ and $C$ are constants depending only on the geometry of $(M_\infty,g_{\infty}(t))_{t\in I\setminus E})$.
    
\begin{proof}[Proof of Claim 1.] By the smooth convergence of $\Psi_i^*K_i(p_i,0\,|\,\cdot,\cdot)$ and the fact that $v>0$, there is $c_0>0,$ such that for any $s\in I\setminus E$, we have
\begin{align}
\label{ineq: noncollapsing for nu_s^i}
    \nu_{s}^i\left(B_{g_{i,s}}(z_{i},\sqrt{|s|})\right)
    \ge \tfrac{1}{2} \nu_s^\infty\left(B_{g_{\infty,s}}(z_\infty, \sqrt{|s|})\right) \ge c_0,
\end{align}
if $i\ge \bar i$ is sufficiently large. 
For $i\ge\bar i$ and $r\ge 10\sqrt{|a|},$ by the Gaussian concentration \eqref{ineq: gaussian concentration} and
\eqref{ineq: noncollapsing for nu_s^i}, we have
\begin{align*}
   c_0\nu_{s}^i\left(M_i\setminus B_{g_{i,s}}(z_i,r)\right)\leq \nu_s^i\left(B_{g_{i,s}}(z_i,\sqrt{|s|})\right) \nu_{s}^i\left(M_i\setminus B_{g_{i,s}}(z_i,r)\right)
 \le 
\exp\left\{- \frac{(r-\sqrt{|s|})^2}{8|s|}
\right\}\leq e^{-cr^2},
\end{align*}
for some $c$ depending on $a$ and $b$.
Note that the Gaussian concentration is also true for $\nu^\infty_s$ by Fatou's lemma. Thus, \eqref{ineq: gaussian decay} also holds for $i=\infty.$
\end{proof}

For all $s\in I\setminus E$, set
$\Omega_s=\bar B_{g_{\infty,s}}(z_\infty, A)$, where $A$ is some large number to be determined. Let us also denote $\Omega:=\cup_{s\in I\setminus E} \Omega_s$, which is compact by Lemma \ref{lem: compact pyramid}. For all $s\in I\setminus E$ and for $i$ large enough or $i=\infty$, define
\[
    \mu_s^\infty:= \nu_s^\infty|_{\Omega_s}
    + \eta_s \delta_{z_\infty},\quad
    \mu_s^i := \Psi_{i*}(\mu_s^\infty),
\]
where the push-forward by $\Psi_i$ makes sense because $\spt\mu_s^\infty\subset U_i$ when $i$ is large enough. 
\\

\noindent
\textbf{Claim 2:} We can fix $A$ large enough, such that for $i=\infty$ or for $i\in \mathbb{N}$ sufficiently large, we have
\[
    \sup_{s\in I\setminus E} \dist_{W_1}^{(M_i,g_{i}(s))}(\nu_s^i, \mu_s^i)
    <\varepsilon.
\]

\begin{proof}[Proof of Claim 2.]
For any $s\in I\setminus E$ and any $1$-Lipschitz function $\phi$ on $(M_\infty,g_{\infty}(s)),$
 by \eqref{ineq: gaussian decay}, we have
\begin{align*}
   \int \phi \, d(\nu_s^\infty-\mu_s^\infty)&=\int (\phi-\phi(z_\infty)) \, d(\nu_s^\infty-\mu_s^\infty)
   \le \int_{M_\infty\setminus \Omega_s} \dist_{g_{\infty,s}}(z_\infty,x) \, d\nu_s^\infty(x)
    \\
    &\le AC e^{-cA^2}
    + C\int_A^\infty r e^{-cr^2} dr < \varepsilon,
\end{align*}
if $A$ is large enough. Here we have used a standard real analysis result (c.f.  \cite[Lemma 3.3]{MZ21}).
Let $s\in I\setminus E$ and $\phi$ be any $1$-Lipschitz function  on $(M_i,g_{i}(s)).$ 
By \eqref{ineq: gaussian decay}, if $A$ is fixed large enough, then whenever $i$ is sufficiently large (depending on $A$), we have
\begin{align*}
    &\int \phi \, d(\nu_s^i-\mu_s^i)
    = \int [\phi-\phi(z_i)] \, d(\nu_s^i-\mu_s^i)\\
    \le& A \int_{\Omega_s}d(\Psi_i^*\nu_s^i-\nu_s^\infty)
    + \int_{M_i\setminus \Psi_i(\Omega_s)}
    \dist_{g_{i,s}}(z_i, x) \, d\nu^i_s(x)
    \\
    \le & 2A|\Omega|_{g_{\infty,s}}
    \Big\|K_i(p_i,0\,|\,\Psi_i(\cdot),s)-v_s^\infty\Big\|_{C^0(\Omega)}
    + AC e^{-cA^2} + C\int_A^\infty r e^{-cr^2} dr < \varepsilon.
\end{align*}
Here we have used the smooth convergence (\ref{supernonsense009}) and the fact $\Psi_i(z_\infty)=z_i$. This finishes the proof of the claim.
    
\end{proof}

Next, we define a sequence of coupling by
\[
    \tilde q_s^i := ({\rm id}, \Psi_i)_*(\mu_s^\infty)\in \Pi(\mu_s^\infty,\mu_s^i).
\]
Then, for any $s,t\in I$ with $s<t,$ we have
\begin{align*}
    & \int_{M_\infty\times M_i}
    d^{Z_s}_{W_1}\left(
        \phi^\infty_{s*} \nu^\infty_{x,t\,|\,s},
        \phi^i_{s*} \nu^i_{y,t\,|\,s}
    \right)\, d\tilde q_t^i(x,y)\\
    =& \int_{\Omega_t}
    d^{Z_s}_{W_1}\left(
        \phi^\infty_{s*} \nu^\infty_{x,t\,|\,s},
        \phi^i_{s*} \nu^i_{\Psi_i(x),t\,|\,s}
    \right)\, d\nu^\infty_t(x) 
    +  \eta_t \cdot d^{Z_s}_{W_1}\left(
        \phi^\infty_{s*} \nu^\infty_{z_\infty,t\,|\,s},
        \phi^i_{s*} \nu^i_{z_i,t\,|\,s}\right).
\end{align*}

\noindent\textbf{Claim 3:}
There is a large $\bar i\in\mathbb{N}$, such that if $i\ge \bar i$, then, for any $s,t\in I\setminus E = [a, b-\varepsilon^2]$ with $s<t$ and for any $x\in \Omega_t$, we have
\[
    \dist^{Z_s}_{W_1}\left(
        \phi^\infty_{s*} \nu^\infty_{x,t\,|\,s},
        \phi^i_{s*} \nu^i_{\Psi_i(x),t\,|\,s}\right)
        <\varepsilon.
\]
\begin{proof}[Proof of Claim 3]
Suppose not. By passing to a subsequence, we may assume that there are $s_i,t_i\in I\setminus E, s_i<t_i, x_i\in \Omega_{t_i},$ such that
\begin{equation}
\label{ineq: 1-W dist no conv}
    \dist^{Z_{s_i}}_{W_1}\left(
        \phi^\infty_{s_i*} \nu^\infty_{x_i,t_i\,|\,s_i},
        \phi^i_{s_i*} \nu^i_{\Psi_i(x_i),t_i\,|\,s_i}\right)
        \ge \varepsilon.
\end{equation}
By passing to a further subsequence, we may assume that $t_i\to \bar t, s_i\to \bar s\le \bar t, x_i\to \bar x\in \Omega_{\bar t}.$ 

\noindent \textbf{Case A:} $\bar s=\bar t.$ 
Write $\bar x_i = \Psi_i(\bar x).$ 
\begin{align*}
    &\dist^{Z_{s_i}}_{W_1}\left(
        \phi^\infty_{s_i*} \nu^\infty_{x_i,t_i\,|\,s_i},
        \phi^i_{s_i*} \nu^i_{\Psi_i(x_i),t_i\,|\,s_i}\right)\\
    \le&\ \dist^{Z_{s_i}}_{W_1}\left(
        \phi^\infty_{s_i*} \nu^\infty_{x_i,t_i\,|\,s_i},
        \phi^\infty_{s_i*} \delta_{\bar x}\right)
        + \dist^{Z_{s_i}}_{W_1}\left(
        \phi^\infty_{s_i*} \delta_{\bar x},
        \phi^i_{s_i*} \delta_{\Psi_i(\bar x)}
        \right)
        +\dist^{Z_{s_i}}_{W_1}\left(
        \phi^i_{s_i*} \delta_{\Psi_i(\bar x)}, \phi^i_{s_i*}\nu^i_{\Psi_i(x_i),t_i\,|\,s_i}\right)\\
    \le&\ \dist^{g_{\infty,s_i}}_{W_1}\left(
        \nu^\infty_{x_i,t_i\,|\,s_i},
        \delta_{\bar x}\right)
        + \varepsilon_i
        +\dist^{g_{i,s_i}}_{W_1}\left(
        \delta_{\bar x_i},
        \nu^i_{\Psi_i(x_i),t_i\,|\,s_i}\right) \\
        \le&\ \dist^{g_{\infty,s_i}}_{W_1}\left(
        \nu^\infty_{x_i,t_i\,|\,s_i},\nu^\infty_{\bar x,t_i\,|\,s_i}
        \right)+\dist^{g_{\infty,s_i}}_{W_1}\left(
        \nu^\infty_{\bar x,t_i\,|\,s_i},
        \delta_{\bar x}\right)+\varepsilon_i
        \\
        &\quad\quad +\dist^{g_{i,s_i}}_{W_1}\left(
       \nu^i_{\bar x_i,t_i\,|\,s_i},
        \nu^i_{\Psi_i(x_i),t_i\,|\,s_i}\right)+\dist^{g_{i,s_i}}_{W_1}\left(
        \nu^\infty_{\bar x_i,t_i\,|\,s_i},
        \delta_{\bar x_i}\right) 
        \\
    \le&\ 2\dist_{g_{\infty,t_i}}(x_i,\bar x)
    +2\varepsilon_i
    + \dist^{g_{\infty,s_i}}_{W_1}\left(
        \nu^\infty_{\bar x,t_i\,|\,s_i},
        \delta_{\bar x}\right)
    + \dist^{g_{i,s_i}}_{W_1}\left(
        \nu^\infty_{\bar x_i,t_i\,|\,s_i},
        \delta_{\bar x_i}\right) \to 0,
\end{align*}
as $i\to \infty$, which is a contradiction to \eqref{ineq: 1-W dist no conv}. Here we have also applied (\ref{monotoneofdW1}).
The last convergence above is due to \cite[Proposition 9.5]{Bam20a} which is but a consequence of Proposition \ref{measureaccumulationofHcenter}.

\noindent \textbf{Case B:} $\bar s<\bar t.$
By Theorem \ref{thm: HK conv under CGH}, after possibly passing to a subsequence, we have
\[
    \Psi_i^*K_i(\Psi_i(x_i),t_i\,|\,\cdot,\cdot)
\to K_\infty(\bar x,\bar t\,|\, \cdot,\cdot)
\]
in the $C_c^\infty$-topology and the convergence is uniform on compact subsets of $M_\infty\times [a,\bar t).$ In particular, 
\[
\left\|\Psi_i^*\nu_{\Psi_i(x_i),t_i\,|\,s_i}- \nu^\infty_{\bar x,\bar t\,|\,s_i}\right\|_{C^0(\mathcal{K})} \to 0,
\]
as $n$-forms for any compact subset $\mathcal{K}\subset M_\infty.$
Let $(z,\bar s)$ be an $H_n$-center of $(\bar x,\bar t)$ and let $B:=B_{g_{\infty,\bar s}}(z,10D) $ for some large constant $D$ to be determined. First choose $D$ large enough so that $\nu^\infty_{\bar x,\bar t\,|\,\bar s}(M_\infty\setminus B) < \frac{\varepsilon}{10D}.$ This is possible because of Proposition \ref{measureaccumulationofHcenter}. Then we have
\begin{align*}
    &\dist^{Z_{s_i}}_{W_1}\left(
        \phi^\infty_{s_i*} \nu^\infty_{x_i,t_i|s_i},
        \phi^i_{s*} \nu^i_{\Psi_i(x_i),t_i|s_i}\right)\\
\le &\ \dist^{g_{\infty,s_i}}_{W_1}\left(
        \nu^\infty_{x_i,t_i|s_i}, 
        \nu^\infty_{\bar x,\bar t|s_i} \right)
        +\dist^{Z_{s_i}}_{W_1}\left(
        \phi^\infty_{s_i*} \nu^\infty_{\bar x,\bar t|s_i},
        \phi^i_{s*} \nu^i_{\Psi_i(x_i),t_i|s_i}\right).
\end{align*}
The first term above clearly converges to 0.
For the second term, we argue in the same way as Claim 2 above. 
By the local distance distortion estimates, we may assume that $B_i=B_{g_{\infty,s_i}}(z,D)\subset B.$
Consider any bounded $1$-Lipschitz function $\phi$ defined on $Z_{s_i}.$  We may assume that $\phi(\phi_{s_i}^\infty(z))=0,$ for otherwise we may replace it with $\phi-\phi(\phi_{s_i}^\infty(z))$. Then, we compute
\begin{align*}
&   \int_{M_\infty} \phi\circ \phi^\infty_{s_i}\, d\nu^\infty_{\bar x,\bar t\,|\,s_i}
    - \int_{M_i} \phi\circ \phi^i_{s_i}\, d \nu^i_{\Psi_i(x_i),t_i\,|\,s_i}\\
\le & \ \int_{B_i} \phi\circ \phi^\infty_{s_i} \,d\nu^\infty_{\bar x,\bar t\,|\,s_i}
- \int_{\Psi_i(B_i)} \phi\circ \phi^i_{s_i}\, d \nu^i_{\Psi_i(x_i),t_i\,|\,s_i}\\
&\ + CD e^{-cD^2}  + C \int_D^\infty s e^{-cs^2}\, ds,
\end{align*}
where we used the Gaussian concentration as in Claim 1. We can fix $D$ so that the last line above is less than $\varepsilon/4$ for all $i$ large. Then, we have
\begin{align*}
&    \int_{B_i} \phi\circ \phi^\infty_{s_i} \,d\nu^\infty_{\bar x,\bar t\,|\,s_i}
- \int_{\Psi_i(B_i)} \phi\circ \phi^i_{s_i}\, d \nu^i_{\Psi_i(x_i),t_i\,|\,s_i}\\
\le &\ \int_{B_i} \left\{ \phi\circ \phi^\infty_{s_i}
- \phi\circ \phi^i_{s_i}\circ \Psi_i\right\}
\,d\nu^\infty_{\bar x,\bar t\,|\,s_i}
+ \int_{B_i} \phi\circ \phi^i_{s_i}\circ \Psi_i
\left\{ d\nu^\infty_{\bar x,\bar t\,|\,s_i}
- \Psi_i^*d \nu^i_{\Psi_i(x_i),t_i\,|\,s_i} \right\}\\
\le &\ \varepsilon_i
+ 2D \left\|\Psi_i^*\nu_{\Psi_i(x_i),t_i\,|\,s_i}- \nu^\infty_{\bar x,\bar t\,|\,s_i}\right\|_{C^0(B)}.
\end{align*}
Note that the last line does not depend on $\phi$ and converges to $0$ since $B$ is compact. 
Hence, we have
\[
    \dist^{Z_{s_i}}_{W_1}\left(
        \phi^\infty_{s_i*} \nu^\infty_{x_i,t_i\,|\,s_i},
        \phi^i_{s_i*} \nu^i_{\Psi_i(x_i),t_i\,|\,s_i}\right)
    < \varepsilon/2,
\]
when $i$ is large enough; this is a contradiction agains \eqref{ineq: 1-W dist no conv}.
\end{proof}

By the definition of the $1$-Wassernstein distance, there are couplings $\theta^i_s\in \Pi(\mu_s^i,\nu_s^i)$ such that
\[
    \int_{M_i\times M_i} \dist_{g_{i,s}}(x,y) \, d\theta^i_s(x,y)
    < \dist_{W_1}^{(M_i,g_{i,s})}(\nu_s^i, \mu_s^i) + \varepsilon < 2 \varepsilon,
\]
if $i\ge \bar i.$
Applying \cite[Lemma 2.2]{Bam20b} for three times, there is $Q_s^i\in \mathcal{P}(M_i\times M_i \times M_\infty\times M_\infty)$ such that the marginal into the first and second factors equals $\theta^i_s,$ the marginal into the third and first factors equals $\tilde q^i_s,$ and the marginal into the third and fourth factors equals $\theta^\infty_s.$
Define $q_s^i$ to be the marginal of $Q_s^i$ into the second and fourth factors. Then $q_s^i\in \Pi(\nu_s^i,\nu_s^\infty).$

\bigskip
\noindent
\textbf{Step 4:} Final verification.
For any $s,t\in I\setminus E=[a,b-\varepsilon^2], s\le t,$ we have
\begin{align*}
    &\int_{M_\infty\times M_i}
    d^{Z_s}_{W_1}\left(
        \phi^\infty_{s*} \nu^\infty_{x,t\,|\,s},
        \phi^i_{s*} \nu^i_{y,t\,|\,s}
    \right)\, d q_t^i(x,y)\\
    = &\ \int_{M_i\times M_i\times M_\infty\times M_\infty}
    d^{Z_s}_{W_1}\left(
        \phi^\infty_{s*} \nu^\infty_{x,t\,|\,s},
        \phi^i_{s*} \nu^i_{y,t\,|\,s}
    \right)\, dQ_t^i(y,y_1,x,x_1)\\
    \le & \ \int
    \left\{\dist_{g_{\infty,t}}(x,x_1) +
    \dist_{g_{i,t}}(y,y_1) +
    d^{Z_s}_{W_1}\left(
        \phi^\infty_{s*} \nu^\infty_{x_1,t\,|\,s},
        \phi^i_{s*} \nu^i_{y_1,t\,|\,s}
    \right)\right\}\, dQ_t^i(y,y_1,x,x_1)\\
    =&\  \int_{M_\infty\times M_\infty} \dist_{g_{\infty,t}}\, d\theta^\infty_t
    + \int_{M_i\times M_i} \dist_{g_{i,t}}\, d\theta^i_t
    + \int_{M_\infty\times M_i}
    d^{Z_s}_{W_1}\left(
        \phi^\infty_{s*} \nu^\infty_{x,t\,|\,s},
        \phi^i_{s*} \nu^i_{y,t\,|\,s}
    \right)\, d\tilde q_t^i(x,y)\\
    < &\  10 \varepsilon,
\end{align*}
if $i\ge \bar i,$ where we have also used the monotonicity formula (\ref{monotoneofdW1}).

\end{proof}

As a special case of this Theorem, under the assumptions of Theorem \ref{Theorem_main}, we have, after passing to a subsequence,
\begin{eqnarray*}
\Big((M,g_i(t))_{t\in[-2,-1]},(\nu^i_s)_{s\in[-2,-1]}\Big)\xrightarrow{\makebox[1cm]{$\mathbb{F}$}}\Big((M_\infty,g_\infty(t))_{t\in[-2,-1]},(\nu_s^\infty)_{s\in[-2,-1)}\Big).
\end{eqnarray*}
To finish the proof of Theorem \ref{Theorem_main},  we only need to show that any $\nu_s^\infty$ in the proof of Theorem \ref{convergence}
is induced by the shrinker potential function. We defer the proof of this fact to Proposition \ref{f-potential} in the next section.

\section{Conergence of the Nash entropy and Smooth tangent flow at infinity}

In this section, we prove Theorem \ref{Theorem_main}(2)(3). To this end, we shall show that the conjugate heat kernel based at $(p_0,0)$, after scaling, converges to the shrinker potential of the asymptotic shrinker. Throughout this section, we still consider an ancient solution $(M,g(t))_{t\in(-\infty,0]}$ satisfying Assumption B with respect to $\big(p_0,0,\{\tau_i\}_{i=1}^\infty,\{p_i\}_{i=1}^\infty\big)$. Let $(M_\infty,g_\infty(t))_{t\in[-2,-1]}$ be the limit in (\ref{smoothconvergence}), which is known to be Perelman's asymptotic shrinker by Prposition \ref{shrinkerstructure}. Recall that at this moment we only have a Gaussian upper bound for $K(p_0,0\,|\,\cdot,\cdot)$. The lack of a Gaussian lower bound is because of the fact that we currently do not have a good estimate for the reduced distance, which, however, is not needed now. For the sake of convenience, we define the following notations for the scaled conjugate heat kernel
\begin{gather}\label{somenonsenseofscaledheatkernels}
u_i(x,t):=\tau_i^{\frac{n}{2}}K(p_0,0\,|\,x,\tau_it)=\frac{1}{(4\pi|t|)^{\frac{n}{2}}}e^{-f_i(x,t)}
\\\nonumber
d\nu^i_t:=K(p_0,0\,|\,\cdot,\tau_it)dg_{\tau_it}=u_i(\cdot,t)dg_{i,t}.
\end{gather}

\begin{Lemma}\label{unitmeasure}
There is a smooth function $u_\infty: M_\infty\times[-2,-1)\rightarrow \mathbb{R}$, satisfying
\begin{enumerate}[(1)]
    \item $u_i\rightarrow u_\infty$ locally smoothly on $M_\infty\times[-2,-1)$, where $u_i$ should be understood to be pulled back by the defining diffeomorphisms in the Cheeger-Gromov-Hamilton convergence.
    \item $u_\infty$ is a solution to the conjugate heat equation coupled with $g_\infty(t)$.
    \item $u_\infty>0$ everywhere on $M_\infty\times[-2,-1)$.
    \item $\displaystyle\int_{M_\infty}u_\infty(\cdot,t)dg_{\infty,t}\equiv1 \quad\text{ for all }\quad t\in[-2,-1).$
\end{enumerate}
\end{Lemma}
\begin{proof}
In view of Corollary \ref{change-l-center-to-Hn-center}, this lemma is but a restatement of Lemma \ref{lem: dnu^oo_s}.
\end{proof}

Since $u_\infty>0$ everywhere, we may define the function $f_\infty$ by
\begin{eqnarray*}
u_\infty(x,t):=(4\pi|t|)^{-\frac{n}{2}}e^{-f_\infty(x,t)}\quad \text{ for all }\quad (x,t)\in M_\infty\times[-2,-1).
\end{eqnarray*}
Then, we obviously have
$$f_i\rightarrow f_\infty\quad\text{ locally smoothly on } \quad M_\infty\times[-2,-1).$$

\begin{Lemma}\label{lowerquadratic}
For any $\varepsilon\in(0,1)$, there exists a constant $C$ independent of $i$, such that the following holds for each $i\in\mathbb{N}\cup\{\infty\}$.
\begin{eqnarray*}
f_i(x,t)\geq\frac{1}{C}\dist_{g_{i,t}}^2(x,z_i)-C\quad\text{ for all }\quad (x,t)\in M\times[-2,-(1+\varepsilon)].
\end{eqnarray*}
\end{Lemma}
\begin{proof}
Fixing an $\varepsilon \in(0,1)$, we have that $\cup_{t\in[-2,-1-\varepsilon]}B_{g_{\infty,t}}(z_\infty,1)\times\{t\}$ is a precompact set in $M_\infty\times [-2,-1)$, and hence $u_i\rightarrow u_\infty$ uniformly on this set. Consequently we have
\begin{eqnarray*}
\liminf_{i\rightarrow\infty}\int_{B_{g_{i,t}}(z_i,1)}u_i(\cdot,t)dg_{i,t}&=&\int_{B_{g_{\infty,t}}(z_\infty,1)}u_\infty(\cdot,t)dg_{\infty,t}
\\
&\geq& \inf_{s\in[-2,-1-\varepsilon]}\int_{B_{g_{\infty,s}}(z_\infty,1)}u_\infty(\cdot,s)dg_{\infty,s}>0,
\end{eqnarray*}
for all $t\in[-2,-1-\varepsilon]$. Hence, we can find a positive number $c_0>0$, such that
\begin{eqnarray}\label{extranonsense001}
\int_{B_{g_{i,t}}(z_i,1)}u_i(\cdot,t)dg_{i,t}\geq c_0\quad\text{ for all }\quad t\in[-2,-1-\varepsilon]\quad\text{ and for all }\quad i\in\mathbb{N}.
\end{eqnarray}

Let $t\in[-2,-1-\varepsilon]$ and let $(z'_i,t)$ be an $H_n$-center of $(p_0,0)$ with respect to the scaled Ricci flow $(M,g_i(t))$. By Proposition \ref{measureaccumulationofHcenter} and (\ref{extranonsense001}), we have that 
\begin{eqnarray}\label{extranonsense002}
\dist_{g_{i,t}}(z_i,z'_i)\leq\sqrt{\frac{2}{c_0}H_n|t|}+1\leq C\quad\text{ for all }\quad t\in[-2,-1-\varepsilon]\quad\text{ and for all }\quad i\in\mathbb{N}.
\end{eqnarray}
The case when $i\in\mathbb{N}$ follows from Proposition \ref{gauss0} and (\ref{extranonsense002}). If $i=\infty$, then this case follows from the fact that $f_\infty$ is the local smooth limit of $\{f_i\}_{i=1}^\infty$.
\end{proof}

\begin{Lemma}
There is a constant $C$ independent of $i$, such that the following hold for all $t\in[-2,-1)$ and $i\in\mathbb{N}\cup\{\infty\}$.
\begin{eqnarray}
\int_{M_i}f_i^2(\cdot,t)u_i(\cdot,t)dg_{i,t}\leq C.\label{lem5.6.2}
\end{eqnarray}
Here we have let $M_i\equiv M$ for all $i\in\mathbb{N}$.
\end{Lemma} 
\begin{proof}
By (\ref{Nashintegral_2}) and the Nash entropy bound in Theorem \ref{entropybound}, we have that (\ref{lem5.6.2}) hold for $i\in\mathbb{N}$. The case when $i=\infty$ then follows from Fatou's lemma and the locally smooth convergences of $f_i$ and $u_i$. 
\end{proof}

\begin{Proposition}\label{Nconst}
The following statements are true.
\begin{enumerate}[(1)]
    \item $\displaystyle \lim_{i\rightarrow\infty}\mathcal{N}_i(|t|)=\int_{M_\infty}f_\infty(\cdot,t)u_\infty(\cdot,t)dg_{\infty,t}-\frac{n}{2}$ for all $t\in[-2,-1)$.
    \item $N_\infty:=\displaystyle \int_{M_\infty}f_\infty(\cdot,t)u_\infty(\cdot,t)dg_{\infty,t}-\frac{n}{2}$ is a constant in $t\in[-2,-1)$.
\end{enumerate}
Here $\mathcal{N}_i(\tau):=\mathcal{N}_{p_0,0}(\tau_i\tau)$ stands for the Nash entropy of the scaled Ricci flow $(M,g_i(t))$ based at $(p_0,0)$.
\end{Proposition}

\begin{proof}
We shall only prove part (1), since part (2) obviously follows from part (1) and the monotonicity of the Nash entropy. Let us fix an arbitrary $t\in[-2,-1)$. Applying Lemma \ref{lowerquadratic} and (\ref{lem5.6.2}), we may find a positive number $C$, such that the following holds for all $A>C$.
\begin{eqnarray*}
C\geq\int_Mf_i^2u_idg_{i,t}\geq \int_{M\setminus B_{g_{i,t}}(z_i,A)}f_i^2u_idg_{i,t}\geq \frac{A^2}{2C}\int_{M\setminus B_{g_{i,t}}(z_i,A)}f_iu_idg_{i,t}.
\end{eqnarray*}
Hence
\begin{eqnarray}\label{nonsense5.5}
0<\int_{M\setminus B_{g_{i,t}}(z_i,A)}f_iu_i\,dg_{i,t}\leq\frac{C}{A^2}\quad\text{ for all }\quad A>C.
\end{eqnarray}

Since $f_i\rightarrow f_\infty$ locally smoothly, the following holds for all $A>0$.
\begin{eqnarray*}
\int_{B_{g_{\infty,t}}(z_i,A)}f_\infty u_\infty dg_{\infty,t}&=&\lim_{i\rightarrow \infty}\int_{B_{g_{i,t}}(z_i,A)}f_i u_i dg_{i,t}
\\
&=&\lim_{i\rightarrow \infty}\left(\int_{M}f_i u_i dg_{i,t}-\int_{M\setminus B_{g_{i,t}}(z_i,A)}f_iu_idg_{i,t}\right)
\\
&=&\lim_{i\rightarrow \infty}\left(\mathcal{N}_i(|t|)+\tfrac{n}{2}+O(A^{-2})\right).
\end{eqnarray*}
Here we have applied (\ref{nonsense5.5}). Note that the constant $C$ in formula (\ref{nonsense5.5}) is independent of both $i$ and $A$. Finally, taking $i\rightarrow\infty$ first and then $A\rightarrow\infty$, the conclusion follows.
\end{proof}

Therefore, if we use $u_\infty$ to construct a Nash entropy, then what we obtain is a constant---the critical point of Perelman's monotonicity formula. This should imply that $f_\infty$ is a shrinker potential. However, since we do not have any geometric condition on $(M_\infty,g_\infty(t))$ except for the fact that it is a shrinker, and since Perelman's monotonicity formula depends heavily on the integration by parts at infinity, we cannot so easily conclude that $f_\infty$ is a shrinker potential. One way to resolve this is to apply the cut-off function constructed by \cite[Lemma 3]{LW20}. We shall use an alternative method. Recall the following result of Bamler.

\begin{Proposition}[Proposition 6.1 in \cite{Bam20c}]\label{almostselfsimilar}
For any $Y<\infty$ and $\varepsilon>0$, there is a $\bar{\delta}(Y,\varepsilon)>0$, such that the following holds whenever $\delta\in (0,\bar{\delta})$. Let $(M,g(t))_{t\in I}$ be a complete Ricci flow with bounded curvature within each compact time interval. Let $(x_0,t_0)\in M\times I$ and $r>0$ be such that $[t_0-\delta^{-1}r^2,t_0-\delta r^2]\subset I$. Suppose $\mathcal{N}_{x_0,t_0}(r^2)\geq -Y$. If $$\mathcal{N}_{x_0,t_0}(\delta^{-1}r^2)\geq \mathcal{N}_{x_0,t_0}(\delta r^2)-\delta,$$ then we have
$$\int_{t_0-\varepsilon^{-1} r^2}^{t_0-\varepsilon r^2}\int_M \tau\left|\,\Ric+\nabla^2 f-\frac{1}{2\tau}g\,\right|^2d\nu_{x_0,t_0\,|\,t}dg_tdt\leq \varepsilon,$$
where $\tau(t):=t_0-t$ and $d\nu_{x_0,t_0\,|\,t}:=(4\pi\tau)^{-\frac{n}{2}}e^{-f}dg_t$ is the conjugate heat kernel based at $(x_0,t_0)$.
\end{Proposition}
\begin{proof}
In fact, in the proof of \cite[Proposition 6.1]{Bam20c}, the conclusion stated in the current proposition relies only on Perelman's monotonicity formulas for entropies, which are obviously true for Ricci flows with bounded curvature. Hence, Bamler's proof can be directly applied to our case.
\end{proof}

We are then ready to show that $f_\infty$ is a shrinker potential.

\begin{Proposition}\label{f-potential}
$f_\infty$ is a shrinker potential function, satisfying
\begin{eqnarray*}
\Ric_\infty+\nabla^2f_\infty=\frac{1}{2|t|}g_\infty\quad\text{ for all }\quad t\in [-2,-1).
\end{eqnarray*}
\end{Proposition}

\begin{proof}
By Theorem \ref{entropybound}, we have that, for any $\delta>0$, it holds that
\begin{eqnarray*}
\lim_{i\rightarrow\infty}\mathcal{N}_{p_0,0}(\delta^{-1}\tau_i)-\mathcal{N}_{p_0,0}(\delta\tau_i)=0.
\end{eqnarray*}
Hence, Proposition \ref{almostselfsimilar} implies the existence of a sequence $\varepsilon_i\searrow 0$, such that
$$\int_{-\varepsilon_i^{-1} \tau_i}^{-\varepsilon_i \tau_i}\int_M |t|\left|\,\Ric+\nabla^2 f-\frac{1}{2|t|}g\,\right|^2d\nu_{p_0,0\,|\,t}dg_tdt\leq \varepsilon_i,$$
where $d\nu_{p_0,0\,|\,t}=(4\pi|t|)^{-\frac{n}{2}}e^{-f}dg_t$ is the conjugate heat kernel based at $(p_0,0)$. Scaling by $\tau_i$, we immediately have
\begin{align}\label{extranonsense004}
   &\quad\int_{-2 }^{-1}\int_M |t|\left|\,\Ric_{g_i}+\nabla^2 f_i-\frac{1}{2|t|}g_i\,\right|^2u_idg_{i,t}dt
   \\\nonumber
   &\leq\int_{-\varepsilon_i^{-1} }^{-\varepsilon_i }\int_M |t|\left|\,\Ric_{g_i}+\nabla^2 f_i-\frac{1}{2|t|}g_i\,\right|^2u_idg_{i,t}dt
   \\\nonumber
   &\leq \varepsilon_i \rightarrow 0,
\end{align}
where $g_i(t):= \tau_i^{-1}g(\tau_it)$, $f_i$ and $u_i$ are defined in (\ref{somenonsenseofscaledheatkernels}). Because of the local smooth convergence of $g_i$, $f_i$, and $u_i$, taking a limit for (\ref{extranonsense004}) and applying Fatou's lemma, we have
$$\int_{-2 }^{-1}\int_{M_\infty} |t|\left|\,\Ric_{g_\infty}+\nabla^2 f_\infty-\frac{1}{2|t|}g_\infty\,\right|^2u_\infty dg_{\infty,t}dt=0.$$
Since $u_\infty>0$ everywhere, the Proposition follows immediately.
\end{proof}

\begin{proof}[Proof of Theorem \ref{Theorem_main}(2)]
To prove Theorem \ref{Theorem_main}(2), we need only to show  $N_\infty=\mu_\infty$, where $N_\infty$ is defined in Proposition \ref{Nconst}(2) and $\mu_\infty$ is the entropy of the asymptotic shrinker. Because of Lemma \ref{unitmeasure}(4), by (\ref{canonicalformnormalization}) we have
\begin{eqnarray*}
N_\infty&=&\int_{M_\infty} f_\infty u_\infty dg_{\infty,t}-\frac{n}{2}
\\
&=&\int_{M_\infty}|t|(|\nabla f_\infty|^2+R_\infty)u_\infty dg_{\infty,t}-\frac{n}{2}+\mu_\infty
\\
&=&\int_{M_\infty}|t|(\Delta f_\infty+R_\infty)u_\infty dg_{\infty,t}-\frac{n}{2}+\mu_\infty
\\
&=&\int_{M_\infty}|t|\frac{n}{2|t|}u_\infty dg_{\infty,t}-\frac{n}{2}+\mu_\infty
\\
&=&\mu_\infty.
\end{eqnarray*}
Here we have applied integration by parts at infinity. This is valid, since both $f_\infty$ and $|\nabla f_\infty|^2$ have quadratic growth bounds (\cite{CZ10}), and since a Ricci shrinker has at most Euclidean volume growth (\cite{CZ10, Mun09}); a standard cut-off argument easily verifies the integration by parts at infinity.
\end{proof}

\begin{proof}[Proof of Theorem \ref{Theorem_main}(3)]
 Since any Ricci flow with bounded curvature within each compact time interval is $H_n$-concentrated, we have
\begin{eqnarray*}
\big((M,g_i(t))_{t\in(-\infty,0]},(\nu^i_s)_{s\in(-\infty,0]}\big)\xrightarrow{\makebox[1cm]{$\mathbb{F}$}}\mathcal{X},
\end{eqnarray*}
where $d\nu_s^i=u_i(\cdot,s)dg_{i,s}$ and $\mathcal{X}$ is a metric flow pair over $(-\infty,0]$ (c.f. \cite[Theorem 7.8]{Bam20b}). Here the $\mathbb{F}$-convergence should be understood to be the $\mathbb{F}$-convergence over each finite subinterval of $(-\infty,0]$.

By Corollary \ref{change-l-center-to-Hn-center}, we have
\begin{eqnarray*}
\big(M,g_i(t),z_i\big)_{t\in[-2,-1]}\xrightarrow{\makebox[1cm]{}}\big(M_\infty,g_\infty(t),z_\infty\big)_{t\in[-2,-1]}.
\end{eqnarray*}
Therefore, Theorem \ref{convergence} implies that
\begin{eqnarray*}
\mathcal{X}_{[-2,-1)}=\big((M_\infty,g_\infty(t))_{t\in[-2,-1)},(\nu^\infty_s)_{s\in[-2,-1)}\big),
\end{eqnarray*}
where $\nu_s^\infty:=u_\infty(\cdot,s)dg_s$ is a conjugate heat flow on $(M_\infty,g_\infty(t))$.

Finally, the proof of Theorem \ref{convergence} indicates that $u_i\rightarrow u_\infty$ locally smoothly, where $u_i$ is defined in (\ref{somenonsenseofscaledheatkernels}). Hence, by Proposition \ref{f-potential}, $u_\infty=(4\pi|t|)^{-\frac{n}{2}}e^{-f_\infty}$ is indeed a conjugate heat flow made of a shrinker potential function; this proves Theorem \ref{Theorem_main}(3).
\end{proof}

\section{Independence of the base point}

In this section, we prove Corollary \ref{entropynoloss}. Recall that in the statement of Assumption B, the choices of the  base point $(p_0,0)$ and the sequence of $\ell$-centers $\{(p_i,-\tau_i)\}_{i=1}^\infty$ are involved. If an ancient solution satisfies Assumption B with respect to some certain base point $(p_0,0)$ and some certain sequence of scales $\{\tau_i\}_{i=1}^\infty$, it is not immediately clear whether Assumption B is still valid if we alter the base point and the sequence of scales. And even if it were, it is not yet clear whether the limit is the same as before. Although the choices of the $\ell$-centers $p_i$ are not unique either, yet, according to what we have developed by far, this fact does not affect the validity of Assumption B or the limit flow in (\ref{smoothconvergence}). What we shall prove next is that for an ancient solution satisfying Assumption B, we may freely alter the base point, while the asymptotic shrinker is still the same. (But we may not change the sequence of scales $\{\tau_i\}_{i=1}^\infty$ at will.) This leads to the proof of Corollary \ref{entropynoloss}. 

\begin{Theorem}\label{basepointindep}
Let $(M,g(t))_{t\in(-\infty,0]}$ be an ancient solution satisfying (\ref{curvaturebound}) and Assumption B with respect to $\big(p_0,0,\{\tau_i\}_{i=1}^\infty,\{p_i\}_{i=1}^\infty\big)$. Then, for any $(p'_0,t'_0)\in M\times(-\infty,0]$, the following holds.
\begin{eqnarray*}
\big(M,g'_i(t),p'_i\big)_{t\in[-2,-1]}\xrightarrow{\makebox[1cm]{}} \big(M_\infty,g_\infty(t),p'_\infty\big)_{t\in[-2,-1]},
\end{eqnarray*}
where $(p'_i,-\tau_i)$ is an $\ell$-center of $(p'_0,t'_0)$ with respect to the Ricci flow $(M,g(t))$, $g'_i(t):=\tau'^{-1}_ig(t'_0+\tau'_it)$, $\tau'_i:=\tau_i+t'_0$ for all $i$ large, and $(M_\infty,g_\infty(t))$ is the same limit as in (\ref{smoothconvergence}).  
\end{Theorem}

\textbf{Remark:} In fact, if the existing interval of $(M,g(t))$ extends beyond $t=0$, then $t_0'$ can be taken to be greater than $0$, so long as the curvature remains to be bounded up to $t'_0$. This can be easily observed from the proof of this theorem.

\begin{proof}
Since $\displaystyle g'_i(t)=\frac{\tau_i}{\tau'_i}g_i\left(\frac{\tau'_i}{\tau_i}t+\tau_i^{-1}t'_0\right)$, where $g_i(t):=\tau_i^{-1}g_i(\tau_it)$, $$\tau_i^{-1}t'_0\rightarrow 0\quad\text{ and } \quad \frac{\tau_i}{\tau'_i}\rightarrow 1$$ as $i\rightarrow\infty$. It is then clear that we need only to show the equivalence of the base points, that is
\begin{eqnarray}\label{nonsense6-2-0}
\limsup_{i\rightarrow\infty}\dist_{g'_{i,-1}}(p_i,p'_i)<\infty.
\end{eqnarray}

Obviously, our concern is only in the case when $i$ is large. Hence, throughout this proof, we assume
\begin{eqnarray*}
\tau_i\gg|t'_0|\quad\text{ and }\quad \left|\frac{\tau_i}{\tau'_i}-1\right|\ll 1.
\end{eqnarray*}
Let $(z'_i,-\tau_i)$ be an $H_n$-center of $(p'_0,t'_0)$, and $(z_i,-\tau_i)$ and $H_n$-center of $(p_0,0)$. By \cite[Proposition 4.6]{MZ21} and Theorem \ref{entropybound}, we have
\begin{eqnarray*}
\lim_{\tau\rightarrow\infty}\mathcal{N}_{p'_0,t'_0}(\tau)=\lim_{\tau\rightarrow\infty}\mathcal{N}_{p_0,0}(\tau)\geq -Y>-\infty.
\end{eqnarray*}
It then follows from Proposition \ref{H_n_l_n} that
\begin{eqnarray*}
\dist_{g_{i,-1}}(p_i,z_i),\ \dist_{g'_{i,-1}}(p'_i,z'_i)\leq C,
\end{eqnarray*}
where $C$ is independent of $i$. Hence, to prove (\ref{nonsense6-2-0}), it suffices to show that
\begin{eqnarray*}
\dist_{g_{i,-1}}(z_i,z'_i),\ \dist_{g'_{i,-1}}(z_i,z'_i)\leq C,
\end{eqnarray*}
or equivalently,
\begin{eqnarray}\label{nonsense6-2-1}
\dist_{g_{-\tau_i}}(z_i,z'_i)\leq C\sqrt{\tau_i},
\end{eqnarray}
for some constant $C$ independent of $i$.

Let us prove (\ref{nonsense6-2-1}) by contradiction. Assume that there are positive numbers $A_i\nearrow\infty$ such that
\begin{eqnarray*}
\dist_{g_{-\tau_i}}(z_i,z_i')\geq 10\sqrt{A_iH_n\tau_i}.
\end{eqnarray*}
Then, whenever $i$ is large enough, we may find a nonnegative $1$-Lipschitz function $\varphi_i$, satisfying the following properties: 
\begin{enumerate}[(1)]
    \item  $\varphi_i$ is compactly supported on $B_{g_{-\tau_i}}(z_i,2\sqrt{A_iH_n\tau_i})$, 
    \item $\varphi_i\geq \sqrt{A_iH_i\tau_i}$ on $B_{g_{-\tau_i}}(z_i,\sqrt{A_iH_n\tau_i})$, 
    \item $0\leq \varphi_i\leq 2\sqrt{A_iH_n\tau_i}$ everywhere, 
    \item the support of $\varphi_i$ is disjoint from $B_{g_{-\tau_i}}(z'_i,\sqrt{A_iH_n\tau'_i})$.
\end{enumerate} Using $\varphi_i$ as a test function, we may compute
\begin{align*}
&\quad\quad\dist_{W_1}^{g_{-\tau_i}}(\nu_{p_0,0\,|\,-\tau_i},\nu_{p'_0,t'_0\,|\,-\tau_i})
\\
&\geq \int_M \varphi_i d\nu_{p_0,0\,|\,-\tau_i}-\int_M \varphi_i d\nu_{p'_0,t'_0\,|\,-\tau_i}
\\
&\geq\int_{B_{g_{-\tau_i}}(z_i,\sqrt{A_iH_n\tau_i})} \varphi_i d\nu_{p_0,0\,|\,-\tau_i} -\int_{M\setminus B_{g_{-\tau_i}}(z'_i,\sqrt{A_iH_n\tau'_i})}\varphi_i d\nu_{p'_0,t'_0\,|\,-\tau_i}
\\
&\geq \left(1-\frac{1}{A_i}\right)\sqrt{A_iH_n\tau_i}-\frac{1}{A_i}2\sqrt{A_iH_n\tau_i}
\\
&=\left(1-\frac{3}{A_i}\right)\sqrt{A_iH_n\tau_i}\rightarrow\infty,
\end{align*}
as $i\rightarrow\infty$; here we have applied Proposition \ref{measureaccumulationofHcenter}. On the other hand, by (\ref{monotoneofdW1}), we have
\begin{eqnarray*}
d_{W_1}^{g_{-\tau_i}}(\nu_{p_0,0\,|\,-\tau_i},\nu_{p'_0,t'_0\,|\,-\tau_i})\leq d_{W_1}^{g_{t'_0}}(\nu_{p_0,0\,|\,t'_0},\delta_{p'_0})<\infty.
\end{eqnarray*}
This is a contradiction, and (\ref{nonsense6-2-1}) follows; we have finished the proof of the theorem.
\end{proof}

We are now ready to prove Corollary \ref{entropynoloss}. As before, we still consider an ancient solution $(M,g(t))_{t\in(-\infty,0]}$ satisfying (\ref{curvaturebound}) and Assumption B with respect to $\big(p_0,0,\{\tau_i\}_{i=1}^\infty,\{p_i\}_{i=1}^\infty\big)$. We first of all need to deal with the case when the base point is fixed at $(p_0,0)$. 

\begin{Lemma}\label{somethingmaybeuseful}
Let $\ell_\infty: M_\infty\times(1,2]\rightarrow\mathbb{R}$ be the shrinker potential of $(M_\infty,g_\infty(t))_{t\in[-2,-1)}$ as given by Proposition \ref{shrinkerstructure}. Then we have
\begin{eqnarray*}
V_\infty:=\int_{M_\infty}(4\pi|t|)^{-\frac{n}{2}}e^{-\ell_\infty(\cdot,|t|)}dg_{\infty,t}\equiv\lim_{\tau\rightarrow\infty}\mathcal{V}_{p_0,0}(\tau)
\end{eqnarray*}
for all $t\in [-2,-1)$.
\end{Lemma}

\begin{proof}
Let $t\in[-2,-1)$ be arbitrarily fixed. We need only to show 
\begin{eqnarray*}
\int_{M_\infty}(4\pi|t|)^{-\frac{n}{2}}e^{-\ell_\infty(\cdot,|t|)}dg_{\infty,t}=\lim_{i\rightarrow\infty}\int_{M}(4\pi|t|)^{-\frac{n}{2}}e^{-\ell_i(\cdot,|t|)}dg_{i,t},
\end{eqnarray*}
where, as before, $g_i(t)=\tau_i^{-1}g(\tau_it)$ and $\ell_i(\cdot,\tau)=\ell_{p_0,0}(\cdot,\tau_i\tau)$.

By (\ref{subsolution}) and Proposition \ref{measureaccumulationofHcenter}, we have
\begin{eqnarray*}
\int_{M\setminus B_{g_{i,t}}(z_i,\sqrt{AH_n|t|})} (4\pi|t|)^{-\frac{n}{2}}e^{-\ell_i(\cdot,|t|)}dg_{i,t}\leq \int_{M\setminus B_{g_{i,t}}(z_i,\sqrt{AH_n|t|})} u_i(\cdot,t)dg_{i,t}\leq\frac{1}{A},
\end{eqnarray*}
for all $A>1$, where $u_i$ is the conjugate heat kernel based at $(p_0,0)$ with respect to Ricci flow $g_i$. The conclusion then follows from the same argument as in the proof of Proposition \ref{Nconst}.
\end{proof}

\begin{Lemma}\label{somethingmaybeuseful1}
We have
\begin{eqnarray*}
\lim_{\tau\rightarrow\infty}\log \mathcal{V}_{p_0,0}(\tau)=\log V_\infty=\mu_\infty,
\end{eqnarray*}
where $\mu_\infty$ is the entropy of the asymptotic shrinker $(M_\infty,g_\infty,\ell_\infty)$.
\end{Lemma}

\begin{proof}
Since $\ell_\infty$ is a shrinker potential normalized in the way that
\begin{eqnarray*}
\Delta \ell_\infty -|\nabla\ell_\infty|^2+R_\infty+\frac{\ell_\infty-n}{\tau}=0,
\end{eqnarray*}
 we have $\log V_\infty=\mu_\infty$; the details of this argument can be found in the proof of \cite[Theorem 1.2]{Zhang20}.
\end{proof}

\begin{proof}[Proof of Corollary \ref{entropynoloss}]
By \cite[Proposition 4.6]{MZ21} and Theorem \ref{Theorem_main}(2), we already have
$$\lim_{\tau\rightarrow\infty}\mathcal{N}_{p'_0,t'_0}(\tau)=\lim_{\tau\rightarrow\infty}\mathcal{W}_{p'_0,t'_0}(\tau)=\mu_\infty$$
for any $(p'_0,t'_0)\in M\times(-\infty,0]$. By Theorem \ref{basepointindep}, $(M,g(t))_{t\in(-\infty,0]}$ satisfies Assumption B with respect to $\big(p'_0,t'_0, \{\tau'_i\}_{i=1}^\infty,\{p'_i\}_{i=1}^\infty\big)$, and the smooth limit is the same as the one in (\ref{smoothconvergence}). Since the entropy of a Ricci shrinker is unique, the last conclusion follows from applying Lemma \ref{somethingmaybeuseful1} to $\mathcal{V}_{p'_0,t'_0}(\tau)$.

\end{proof}

\section{Perelman's $\nu$-functional on ancient solutions satisfying Assumption B}

In this section, we shall prove Theorem \ref{nu-functional}. We still consider an ancient solution $(M^n,g(t))_{t\in(-\infty,0]}$ satisfying (\ref{curvaturebound}) and Assumption B with respect to $\big(p_0,0,\{\tau_i\}_{i=1}^\infty,\{p_i\}_{i=1}^\infty\big)$. Let $(M_\infty, g_\infty, \ell_\infty)$ be the asymptotic shrinker from Proposition \ref{shrinkerstructure}. By the remark below Corollary \ref{LUTypeInoncollapsed}, the ancient solution in question has bounded geometry withing each compact time interval, and this will be a convenient condition to be applied in this section.

Let us fix an arbitrary $t_0\in(-\infty,0]$. Without loss of generality, we may assume $t_0<0$. The reason is because $(M,g(t))$ has bounded curvature at each time and can be extended pass $t=0$, and we may replace our base point $(p_0,0)$ by some $(p_0,\varepsilon)$, where $\varepsilon>0$, if ever $t_0=0$. By Theorem \ref{basepointindep}, this the change of base point does not affect anything, except for slight modifications of the scales $\{\tau_i\}_{i=1}^\infty$. 

Let us fix a scale $\tau_0>0$ and consider an arbitrary function $u_0$ which is qualified to be a test function for $\mu(g(t_0),\tau_0)$, that is, $u_0\geq 0$,  $\sqrt{u_0}\in C_0^\infty(M)$, and $\int_M udg_{t_0}=1$. We may then solve the conjugate heat equation with $u_0$ being its initial data and obtain the solution
\begin{eqnarray}\label{thedefinitionofu}
u(x,t):=\int_M u_0 K(\cdot,t_0\,|\,x,t)dg_{t_0}.
\end{eqnarray}
Since $u$ is positive on $M\times(-\infty,t_0)$, we write
\begin{eqnarray*}
u(x,t):=(4\pi(T-t))^{-\frac{n}{2}} e^{-f(x,t)}\quad \text{ for all }\quad (x,t)\in M\times(-\infty,t_0),
\end{eqnarray*}
where $$T=t_0+\tau_0.$$
For the notational convenience, we shall use $\spt$ to denote $\spt u_0=\spt \sqrt{u_0}$, and shall fix a point $x_0\in \spt$ and a positive number $D_0>0$, such that
\begin{eqnarray*}
\spt\subset B_{t_0}(x_0,D_0).
\end{eqnarray*}

We shall consider the time-dependent function $\bW(g(t),u(\cdot,t),T-t)=\W(g(t),f(\cdot,t),T-t)$ for all $t\in(-\infty,t_0)$, where $\W$ and $\bW$ are defined in (\ref{Perelmansentropy}) and (\ref{anotherPerelmansentropy}), respectively. Since $u$ is a positive solution to the conjugate heat equation, and since $(M,g(t))$ is a Ricci flow with bounded geometry on compact time-intervals, the classical monotonicity formula of Perelman is still valid. We include the following lemma without a tedious proof. The interested reader could easily verify it; note that the estimate (\ref{EKNT_gradient}) below is helpful.

\begin{Lemma}
We have
\begin{eqnarray*}
\frac{d}{dt}\bW(g(t),u(\cdot,t),T-t)=2(T-t)\int_M \left|\Ric+\nabla^2f-\frac{1}{2(T-t)}g\right|^2udg_t\geq 0
\end{eqnarray*}
for all $t\in(-\infty, t_0)$.
\end{Lemma}

The result of this section consists of the following two theorems, and we shall prove them one-by-one.

\begin{Theorem}\label{backwardlimit}
We have
$$\lim_{t\rightarrow-\infty}\bW(g(t),u(\cdot,t),T-t)\geq \mu_\infty,$$ where $\mu_\infty$ is the entropy of the asymptotic shrinker $(M_\infty,g_\infty,\ell_\infty)$.
\end{Theorem}

\begin{Theorem}\label{forwardlimit}
We have $$\lim_{t\rightarrow t_0-}\bW(g(t),u(\cdot,t),T-t)=\bW(g(t_0),u_0,\tau_0).$$
\end{Theorem}

Recall that the ancient solution $(M,g(t))$ is locally uniformly Type I along the space-time sequence $\{(p_i,-\tau_i)\}_{i=1}^\infty$. We shall use this condition to obtain a locally uniformly lower bound for $u$ around $(p_i,-\tau_i)$. From this point on until the completion of the proof of Theorem \ref{backwardlimit}, we shall assume that
\begin{eqnarray*}
\tau_i\gg |t_0|,\quad \tau_i\gg \tau_0.
\end{eqnarray*}

\begin{Proposition}
For any $\varepsilon\in(0,\frac{1}{4})$, there is a decreasing positive function $c(\cdot,\varepsilon):(0,\infty)\rightarrow(0,1)$ with the following property: for any $r>0$, there is an $i_0$, such that
whenever $i\geq i_0$ and for all $y\in B_{t_0}(x_0,D_0)$, we have
\begin{eqnarray}\label{nonsense8004001}
K(y,t_0\,|\,x,t)\geq c(r,\varepsilon)\tau_i^{-\frac{n}{2}}\quad\text{ for all }\quad(x,t)\in B_{-\tau_i}(p_i,r\sqrt{\tau_i})\times[-2\tau_i,-(1+\varepsilon)\tau_i],
\end{eqnarray}
and consequently
\begin{eqnarray}\label{nonsense8004002}
u(x,t)\geq c(r,\varepsilon)\tau_i^{-\frac{n}{2}}\quad\text{ for all }\quad(x,t)\in B_{-\tau_i}(p_i,r\sqrt{\tau_i})\times[-2\tau_i,-(1+\varepsilon)\tau_i].
\end{eqnarray}
\end{Proposition}
\begin{proof}
Let us fix an arbitrary $y\in B_{t_0}(x_0,D_0)$, then we have
\begin{eqnarray*}
\dist^{g_{-\tau_i}}_{W_1}(\nu_{y,t_0\,|\,-\tau_i},\nu_{p_0,0\,|\,-\tau_i})&\leq&\dist^{g_{t_0}}_{W_1}(\delta_y,\nu_{p_0,0\,|\,t_0})
\\
&\leq& \dist_{t_0}(x_0,y)+\dist^{g_{t_0}}_{W_1}(\delta_{x_0},\nu_{p_0,0\,|\,t_0})\leq C,
\end{eqnarray*}
where $C$ is independent of both $i$ and $y$.

Next, we let $(z_i,-\tau_i)$ and $(z'_i,-\tau_i)$ be $H_n$-centers of $(p_0, 0)$ and $(y, t_0)$, respectively. Arguing in the exactly same way as in the proof of formula (\ref{nonsense6-2-1}), we have
\begin{eqnarray}
\dist_{-\tau_i}(z_i,z'_i)\leq C\sqrt{\tau_i} \quad\text{ for all } i\in\mathbb{N},
\end{eqnarray}
where $C$ is a constant independent of both $i$ and $y$. 

Let $(p'_i,-\tau_i)$ be an $\ell$-center of $(y, t_0)$. By Proposition \ref{H_n_l_n} (note that Proposition \ref{H_n_l_n} relies only on a Nash entropy lower bound, which is true by Corollary \ref{entropynoloss}), we have $\dist_{-\tau_i}(p_i,z_i)\leq C\sqrt{\tau_i}$ and $\dist_{-\tau_i}(z'_i,p'_i)\leq C\sqrt{\tau_i+t_0}\leq C\sqrt{\tau_i}$. Hence, we have
\begin{eqnarray}\label{closebycenters}
\dist_{-\tau_i}(p_i,p'_i)\leq C\sqrt{\tau_i}\quad\text{ for all } i\in\mathbb{N}.
\end{eqnarray}
Therefore, the locally uniformly Type I condition along $\{(p_i,-\tau_i)\}_{i=1}^\infty$ (c.f. Proposition \ref{LUTypeI}) also implies the existence of a positive function $C:(0,\infty)\rightarrow(0,\infty)$, such that, for any $r>0$, there is an $i_0\in\mathbb{N}$ with the property that
\begin{eqnarray}\label{anotherlutypei}
|\Rm|\leq\frac{C(r)}{\tau_i}\quad\text{ on }\quad B_{-\tau_i}(p'_i,r\sqrt{\tau_i})\times[-2\tau_i,-\tau_i]
\end{eqnarray}
for all $i\geq i_0$; in particular, $i_0$ depends on the constant in (\ref{closebycenters}) and the original locally uniformly Type I condition along $\{(p_i,-\tau_i)\}_{i=1}^\infty$. It is to be emphasized that this $i_0$ is independent of $y$, and this is due to the fact that the constant $C$ in (\ref{closebycenters}) is independent of $y$. Since $\ell_{y,t_0}(p'_i,-\tau_i)\leq\frac{n}{2}$, arguing in the same way as the proof of \cite[Proposition 5.1(1)]{CZ20} and using (\ref{anotherlutypei}), we have that for any $r>0$, if $i$ is large enough (independent of $y$), then
\begin{eqnarray*}
\ell_{y, t_0}(x,|t|)\leq C(r,\varepsilon)\quad \text{ for all }\quad (x,t)\in B_{g_{-\tau_i}}(p'_i,r\sqrt{\tau_i})\times[-2\tau_i,-(1+\varepsilon)\tau_i],
\end{eqnarray*}
where $C(\cdot,\varepsilon):(0,\infty)\rightarrow(0,\infty)$ is a function depending on $\varepsilon$. Therefore,  (\ref{subsolution}) implies that
\begin{eqnarray*}
K(y,t_0\,|\,x,t)\geq \frac{1}{(4\pi(t_0-t))}e^{-\ell_{y, t_0}(x,|t|)}\geq c(r,\varepsilon)\tau_i^{-\frac{n}{2}},
\end{eqnarray*}
for all $t\in[-2\tau_i,-(1+\varepsilon)\tau_i]$ and for all $x\in B_{g_{-\tau_i}}(p'_i,r\sqrt{\tau_i})\subset B_{g_{-\tau_i}}\big(p_i,(r-C)\sqrt{\tau_i}\big)$; here we have used (\ref{closebycenters}) again. This finishes the proof of (\ref{nonsense8004001}), and (\ref{nonsense8004002}) follows from (\ref{thedefinitionofu}), (\ref{nonsense8004001}), and the fact that $\int_Mu_0dg_{t_0}=1$.
\end{proof}

Next, we prove a Gaussian upper bound for $u$.

\begin{Proposition}\label{upperbound}
There is a constant $C_0$, depending on the geometry bounds on $M\times[t_0,0]$, the value of $D_0$, $\dist_0(p_0,x_0)$, and the upper bound of $u_0$, such that
\begin{eqnarray*}
u(x,t)\leq C_0K(p_0,0\,|\,x,t) \quad\text{ for all }\quad (x,t)\in M\times(-\infty,t_0).
\end{eqnarray*}
\end{Proposition}
\begin{proof}
Since $K(p_0,0\,|\,\cdot,t_0)$ is a positive function on $M$, we may the take
\begin{eqnarray*}
C_0:=\frac{\sup u_0}{\inf_{\spt}K(p_0,0\,|\,\cdot,t_0)}\in(0,\infty).
\end{eqnarray*}
Then we have $u_0\leq C_0 K(p_0,0\,|\,\cdot,t_0)$ everywhere on $M$, and the conclusion follows from the maximum principle.
\end{proof}

According to the scaling property of the $\bW$ functional, we have $\displaystyle\bW(g(\tau_it),u(\cdot,\tau_it),T-\tau_it)=\bW\left(\tau_i^{-1}g(\tau_it),\tau_i^{\frac{n}{2}}u(\cdot,\tau_it),\tfrac{T}{\tau_i}-t\right)$. Hence, for $t\in[-2,-1]$, we may define
\begin{eqnarray*}
g_i(t)&:=&\tau_i^{-1}g(\tau_it),
\\
f_i(\cdot,t)&:=&f(\cdot,\tau_it)+\frac{n}{2}\log\left(\frac{T+\tau_i|t|}{\tau_i|t|}\right),
\\
\W_i(t)&:=&\bW\left(\tau_i^{-1}g(\tau_it),\tau_i^{\frac{n}{2}}u(\cdot,\tau_it),\tfrac{T}{\tau_i}-t\right)+\frac{n}{2}\log\left(\frac{T+\tau_i|t|}{\tau_i|t|}\right)
\\
&=&\int_M\left(\big(\tfrac{T}{\tau_i}+|t|\big)\big(|\nabla f_i|^2+R_i\big)+f_i-n\right)(4\pi|t|)^{-\frac{n}{2}}e^{-f_i}dg_{i,t}.
\end{eqnarray*}
Since $T/\tau_i\rightarrow 0$, to prove Theorem \ref{backwardlimit}, it suffices to show that
\begin{eqnarray*}
\liminf_{i\rightarrow\infty} \mathcal{W}_i(t)\geq \mu_\infty\quad\text{ for some } \quad t\in[-2,-1].
\end{eqnarray*}

By Proposon \ref{upperbound} (arguing as the proof in \cite[Theorem 2.1]{Lu12} again), we have that $u_i=(4\pi|t|)^{-\frac{n}{2}}e^{-f_i}\rightarrow u_\infty$ locally smoothly on $M_\infty\times [-2,-1)$, where $u_\infty$ is a nonnegative solution to the conjugate heat equation.  (\ref{nonsense8004002}) then implies $u_\infty>0$ everwhere on $M_\infty\times[-2,-1)$ and we may define $f_\infty: M_\infty\times[-2,-1)\rightarrow\mathbb{R}$ by  $u_\infty:=(4\pi|t|)^{-\frac{n}{2}}e^{-f_\infty}$. Then we have $f_i\rightarrow f_\infty$  locally smoothly on $M_\infty\times[-2,-1)$. By Proposition \ref{upperbound} and arguing as in the proof of Lemma \ref{lem: dnu^oo_s}, we have that
\begin{eqnarray}\label{anotherunitmeasurenonsense}
\int_{M_\infty}u_\infty(\cdot,t)dg_{\infty,t}\equiv 1\quad\text{ for all }\quad t\in[-2,1).
\end{eqnarray}
Since we have $\int_M u(\cdot,t)dg_t\equiv 1$ for all $t\in(-\infty,t_0)$ by the property of the conjugate heat equation, the proof of (\ref{anotherunitmeasurenonsense}) is in fact somewhat like the dominated convergence theorem,.

Proposition \ref{upperbound} also implies that $f_i$ has at least quadratic growth, that is, there is a constant $C$, such that
\begin{eqnarray}\label{quadratic_1}
f_i(x,t)&\geq& \frac{1}{C}\dist_{g_{i,t}}^2(p_i,x)-C\quad\text{ for all }\quad (x,t)\in M\times[-2,-1],
\\
f_\infty(x,t)&\geq& \frac{1}{C}\dist_{g_{\infty,t}}^2(p_\infty,x)-C\quad\text{ for all }\quad (x,t)\in M_\infty \times[-2,-1).\label{quadratic_2}
\end{eqnarray}
If we denote $$\mathcal{W}_\infty(t)=\int_{M_\infty}\big(|t|(|\nabla f_\infty|^2+R_\infty)+f_\infty-n\big)u_\infty dg_{\infty,t},$$ then (\ref{quadratic_1}) and  (\ref{quadratic_2}) are sufficient to show 
\begin{eqnarray}\label{wbconvergence}
\liminf_{i\rightarrow\infty} \mathcal{W}_i(t)\geq \mathcal{W}_\infty(t)\quad\text{ for all }\quad t\in[-2,-1).
\end{eqnarray}
To see this, let us fix an arbitrary large number $A$, such that the integrand of both $\mathcal{W}_i(t)$ and $\mathcal{W}_\infty(t)$ are positive outside a disk with radius $A$; this is possible because of (\ref{quadratic_1}) and  (\ref{quadratic_2}). Then we may compute
\begin{align*}
&\quad\quad\int_{B_{g_{\infty,t}}(p_\infty,A)}\big(|t|(|\nabla f_\infty|^2+R_\infty)+f_\infty-n\big)u_\infty dg_{\infty,t}
\\
&=\lim_{i\rightarrow\infty}\left(\int_{B_{g_{i,t}}(p_i,A)}\left(\big(\frac{T}{\tau_i}+|t|\big)\big(|\nabla f_i|^2+R_i\big)+f_i-n\right)(4\pi|t|)^{-\frac{n}{2}}e^{-f_i}dg_{i,t}+\frac{n}{2}\log\left(\frac{T+\tau_i|t|}{\tau_i|t|}\right)\right)
\\
&\leq\liminf_{i\rightarrow\infty}\mathcal{W}_i(t),
\end{align*}
taking $A\rightarrow\infty$, (\ref{wbconvergence}) follows. 

\begin{Lemma}\label{anotherl2nonsese}
For all $t\in[-2,-1)$, we have
\begin{eqnarray*}
\int_M |\nabla f_\infty|^2 u_\infty dg_{\infty,t}<\infty.
\end{eqnarray*}
\end{Lemma}
\begin{proof}
(\ref{wbconvergence}) implies that $\W_\infty(t)<\infty$. Furthermore, by (\ref{quadratic_2}), we have $f_\infty(\cdot,t)>-C$ for some constant $C$. Hence
\begin{eqnarray*}
|t|\int_M(|\nabla f_\infty|^2 +R_\infty)u_\infty dg_{\infty,t}&=&\W_\infty(t)-\int_M(f_\infty-n)u_\infty dg_{\infty,t}
\\
&\leq& \W_\infty(t)+C+n<\infty.
\end{eqnarray*}
This finishes the proof.
\end{proof}

\begin{proof}[Proof of Theorem \ref{backwardlimit}]
Let us fix an arbitrary $t\in[-2,-1)$. By (\ref{quadratic_2}), we have that $u_\infty$ also has a Gaussian upper bound. Hence, by the Euclidean volume growth bound for shrinkers (\cite{Mun09, CZ10}), we have
\begin{eqnarray*}
\int_{M_\infty}\dist_{g_{\infty,t}}^2(p_\infty,\cdot)u_\infty(\cdot,t)dg_{\infty,t}<\infty.
\end{eqnarray*}
Consequently, by (\ref{anotherunitmeasurenonsense}) and Lemma \ref{anotherl2nonsese}, $u_\infty$ can be used as a test function for the $\mu$-functional (see formulas (91) and (92) in \cite{LW20}). Then, by Proposition \ref{shrinkernufunctional}, we have
\begin{eqnarray*}
\W_\infty(t)=\W\big(g_\infty(t),f_\infty(\cdot,t),|t|\big)\geq\mu(g_\infty(t),|t|)=\mu(g_\infty(-1),1)=\mu_\infty.
\end{eqnarray*}
Taking (\ref{wbconvergence}) into account, this finishes the proof.
\end{proof}

From this point on, we will proceed to prove Theorem \ref{forwardlimit}. It turns our that the following coarse Gaussian estimates are very helpful.

\begin{Proposition}[Theorem 26.25 and Theorem 26.31 in \cite{RFV3}]\label{coarsegaussboundthatisuseful}
There is a constant $C$ depending only on the geometry bounds on $M\times[t_0-1,t_0]$, such that the following holds
\begin{eqnarray}
\frac{1}{C(t-s)^{\frac{n}{2}}}\exp\left(-\frac{C\dist_t^2(x,y)}{(t-s)}\right)\leq K(x,t\,|\,y,s)\leq \frac{C}{(t-s)^{\frac{n}{2}}}\exp\left(-\frac{\dist_t^2(x,y)}{C(t-s)}\right),
\end{eqnarray}
for all $x, y\in M$ and for all $t_0-1\leq s<t\leq t_0$. Here $\dist_t$ can be replaced by $\dist_{t'}$ for any $t'\in[t_0-1,t_0]$.
\end{Proposition}

\textbf{Remark:} Because of the boundedness of geometry, we have that for any $x\in M$ and $t_0-1\leq s<t\leq t_0$, it holds that
\begin{eqnarray*}
\Vol_{g_t}\big(B_t(x,\sqrt{t-s})\big)\geq c(t-s)^{\frac{n}{2}},
\end{eqnarray*}
where $c$ depends only on the geometric bounds on $M\times[t_0-1,t_0]$.

\begin{Lemma}\label{coarsebound}
There exists positive constant $C<\infty$, such that
\begin{eqnarray}\label{nonsense_grad}
|\nabla ^k u|\leq C \quad\text{ on }\quad M\times[t_0-1,t_0],
\end{eqnarray}
where $k\in\{0,1,2,3\}$. 
\end{Lemma}
\begin{proof}
Since $u_0$ is smooth and compactly supported, we have $|\nabla^k u_0|\leq C_k$ everywhere on $M$. Then, one may use the same argument as in the proof of \cite[Theorem 10]{LT} to prove (\ref{nonsense_grad}). This is a standard Shi-type estimate and the proof is left to the readers.
\end{proof}

\begin{Lemma}
There exists a constant $C<\infty$, such that
\begin{align}\label{quadratic_nonsense}
u(x,t)\leq \frac{C}{(t_0-t)^{\frac{n}{2}}}\exp\left(-\frac{\dist^2_{t_0}(x,\spt)}{C(t_0-t)}\right),
\end{align}
for all $(x,t)\in M\times[t_0-1,t_0)$. Here $\spt=\spt u_0$ denotes the compact support set of $u_0$.
\end{Lemma}
\begin{proof}
To see this, we compute using Proposition \ref{coarsegaussboundthatisuseful}
\begin{eqnarray*}
u(x,t)&=&\int_{M}u_0(y)K(y,t_0\,|\,x,t)dg_{t_0}(y)
\\
&\leq&\int_{\spt}u_0(y)\cdot \frac{C}{(t_0-t)^{\frac{n}{2}}}\exp\left(-\frac{\dist^2_{t_0}(x,y)}{C(t_0-t)}\right)dg_{t_0}(y)
\\
&\leq&\frac{C}{(t_0-t)^{\frac{n}{2}}}\exp\left(-\frac{\dist^2_{t_0}(x,\spt)}{C(t_0-t)}\right)\int_Mu_0dg_{t_0},
\end{eqnarray*}
the lemma then follows.
\end{proof}

\begin{Proposition}\label{ulogunonsense}
\begin{eqnarray*}
\lim_{t\rightarrow t_0-}\int_M u(\cdot,t)\log u(\cdot,t)\, dg_{t}=\int_M u_0\log u_0\,dg_{t_0}.
\end{eqnarray*}
\end{Proposition}

\begin{proof}
First of all, we show that
\begin{eqnarray}\label{uniform_nonsense}
u(\cdot,t)\log u(\cdot,t)\rightarrow u_0\log u_0 \quad \text{ uniformly on }\quad M.
\end{eqnarray}
Let us fix an arbitrary small $\delta>0$. Since $u$ is bounded from above, and $u\rightarrow u_0$ uniformly (note that $|\nabla^2 u|\leq C$  by Lemma \ref{coarsebound}, we have that $|\partial_t u|\leq |\nabla^ 2u|+Ru \leq C$ near $t=t_0$), we have that $u\log u\rightarrow u_0\log u_0$ uniformly on $\{u_0\geq \delta\}$. Hence, we may find a small $\varepsilon\in(0,\delta)$, such that 
\begin{eqnarray*}
\sup_{\{u_0\geq\delta\}}\big|u(\cdot,t)\log u(\cdot,t)-u_0\log u_0\big|\leq \delta\quad\text{ for all }\quad t\in[t_0-\varepsilon,t_0].
\end{eqnarray*}
On the other hand, we may take $\delta$ and $\varepsilon$ small enough, such that $a\log a\leq a^{\frac{1}{2}}$ whenever $a\in(0,\delta+C\varepsilon)$, where $C$ is the constant in (\ref{nonsense_grad}). Therefore, if $x\in\{u_0<\delta\}$ and $t\in[t_0-\varepsilon,t_0]$, we have $0<u(x,t)\leq u_0(x)+C|t_0-t|< \delta+C\varepsilon$ and $u\log u\leq u^{\frac{1}{2}}\leq\sqrt{\delta+C\varepsilon}$. This implies
\begin{eqnarray*}
\sup_{\{u_0<\delta\}}\big|u(\cdot,t)\log u(\cdot,t)-u_0\log u_0\big|\leq \sqrt{\delta+C\varepsilon}+\delta\quad\text{ for all }\quad t\in[t_0-\varepsilon,t_0].
\end{eqnarray*}
This proves (\ref{uniform_nonsense}).

Next, we shall prove that for any $A>2D_0$ (recall that $D_0>0$ is such that $\spt\subset B_{t_0}(x_0,D_0)$), we have
\begin{eqnarray*}
\lim_{t\rightarrow t_0-}\int_{M\setminus B_{t_0}(\spt, A)}\left|u(\cdot,t)\log u(\cdot,t)\right|dg_t\rightarrow 0.
\end{eqnarray*}
Here $B_{t_0}(\spt,A):=\{x\ |\ \dist_{t_0}(\spt,x)< A\}$. Apparently, this is sufficient for the lemma. Fixing any $A>2D_0$, we may, by (\ref{quadratic_nonsense}), find a positive number $\varepsilon$, such that $u$ is small enough on $M\setminus B_{t_0}(\spt, A)\times[t_0-\varepsilon,t_0)$ and satisfies
\begin{eqnarray*}
0&<&\left|u(x,t)\log u(x,t)\right|<u^{\frac{1}{2}}(x,t)\leq (t_0-t)^{\frac{n}{4}}\frac{C}{(t_0-t)^{\frac{n}{2}}}\exp\left(-\frac{\dist^2_{t_0}(x,\spt)}{2C(t_0-t)}\right)
\\
&\leq&(t_0-t)^{\frac{n}{4}}\frac{C}{(t_0-t)^{\frac{n}{2}}}\exp\left(-\frac{(\dist_{t_0}(x,x_0)-D_0)^2}{2C(t_0-t)}\right)
\\
&\leq&(t_0-t)^{\frac{n}{4}}\frac{C}{(t_0-t)^{\frac{n}{2}}}\exp\left(-\frac{\dist_{t_0}^2(x,x_0)}{4C(t_0-t)}\right)
\end{eqnarray*}
for all $(x,t)\in M\setminus B_{t_0}(\spt, A)\times[t_0-\varepsilon,t_0]$. Therefore, we have
\begin{eqnarray*}
\int_{M\setminus B_{t_0}(\spt, A)}\left|u(\cdot,t)\log u(\cdot,t)\right|dg_t&\leq& (t_0-t)^{\frac{n}{4}}\int_{M\setminus B_{t_0}(x_0,A)}\frac{C}{(t_0-t)^{\frac{n}{2}}}\exp\left(-\frac{\dist_{t_0}^2(x,x_0)}{4C(t_0-t)}\right) dg_t
\\
&\leq& C(t_0-t)^{\frac{n}{4}}\rightarrow 0 \quad \text{ as }\quad t\rightarrow t_0-.
\end{eqnarray*}
The integral estimate in the last inequality above is a standard result using the Bishop-Gromov volume comparison theorem. Similar estimates will appear many times later.
\end{proof}

Next, we deal with the integral of the gradient term. First of all, we recall the following gradient estimate from \cite[Theorem 10]{EKNT08}. Although $u$ is not positive at $t=t_0$, yet one may apply this theorem to $u$ on $M\times [t,\frac{t+t_0}{2}]$, for any $t\in[t_0-1,t_0)$.

\begin{Lemma}[Theorem 10 in \cite{EKNT08}]
\label{lem: grad est on CHF}
There is a constant $C$ such that the following holds
\begin{eqnarray}\label{EKNT_gradient}
\frac{|\nabla u|^2}{u}\leq\frac{Cu}{t_0-t}\left(1+\log\frac{C}{u}\right)^2\leq\frac{C}{t_0-t}(u+u\log^2 u) \quad \text{ on }\quad M\times[t_0-1,t_0).
\end{eqnarray}
\end{Lemma}

\begin{Lemma}\label{someothergradientl2nonsense}
For all $A>2D_0+1$, we have
\begin{eqnarray}
\lim_{t\rightarrow t_0-}\int_{M\setminus B_{t_0}(\spt, A)}\frac{|\nabla u|^2}{u}(\cdot,t)dg_{t}= 0.
\end{eqnarray}
\end{Lemma}
\begin{proof}
Let us fix an arbitrary $A>2D_0+1$. By (\ref{quadratic_nonsense}), we may let $\varepsilon\in(0,1]$ be small enough, such that $u\log^2 u\leq u^\frac{1}{2}$ on $M\setminus B_{t_0}(\spt, A)\times [t_0-\varepsilon,t_0)$. Hence we have
\begin{eqnarray*}
\frac{|\nabla u|^2}{u}(x,t)&\leq& \frac{C}{t_0-t}\Big(u(x,t)+u(x,t)\log^2 u(x,t)\Big) \leq\frac{C}{t_0-t}\big(u(x,t)+u^{\frac{1}{2}}(x,t)\big)
\\
&\leq&\frac{C}{t_0-t}\left(\frac{1}{(t_0-t)^{\frac{n}{2}}}\exp\left(-\frac{\dist^2_{t_0}(x,\spt)}{C(t_0-t)}\right)+\frac{1}{(t_0-t)^{\frac{n}{4}}}\exp\left(-\frac{\dist^2_{t_0}(x,\spt)}{2C(t_0-t)}\right)\right)
\\
&\leq&\frac{C}{(t_0-t)^{\frac{n}{2}+1}}\exp\left(-\frac{\dist^2_{t_0}(x,\spt)}{2C(t_0-t)}\right) 
\end{eqnarray*}
for all $(x,t)\in M\setminus B_{t_0}(\spt, A)\times [t_0-\varepsilon,t_0)$. Integrating the above inequality, we have
\begin{eqnarray*}
\int_{M\setminus B_{t_0}(\spt, A)}\frac{|\nabla u|^2}{u}dg_t&\leq& \int_{M\setminus B_{t_0}(\spt, A)}\frac{C}{(t_0-t)^{\frac{n}{2}+1}}\exp\left(-\frac{\dist^2_{t_0}(x,\spt)}{2C(t_0-t)}\right) dg_t(x)
\\
&\leq&\sqrt{t_0-t}\int_{M\setminus B_{t_0}(\spt, A)}\frac{C\dist^3_{t_0}(x,\spt)}{(t_0-t)^{\frac{n}{2}+\frac{3}{2}}}\exp\left(-\frac{\dist^2_{t_0}(x,\spt)}{2C(t_0-t)}\right)dg_t(x)
\\
&\leq&\sqrt{t_0-t}\int_{M\setminus B_{t_0}(x_0, A)}\frac{\dist^3_{t_0}(x,x_0)}{(t_0-t)^{\frac{3}{2}}}\cdot\frac{C}{(t_0-t)^{\frac{n}{2}}}\exp\left(-\frac{\dist^2_{t_0}(x,x_0)}{4C(t_0-t)}\right)dg_t(x)
\\
&\leq& C\sqrt{t_0-t}\quad\text{ for all }\quad t\in[t_0-\varepsilon,t_0).
\end{eqnarray*}
This finishes the proof.
\end{proof}

\begin{Proposition}\label{grad u l2}
We have
\begin{eqnarray*}
\lim_{t\rightarrow t_0-}\int_M\frac{|\nabla u|^2}{u}(\cdot,t)dg_{t}=\int_M\frac{|\nabla u_0|^2}{u_0}dg_{t_0}.
\end{eqnarray*}
\end{Proposition}

\begin{proof}
Since, by our assumption, $\sqrt{u_0}\in C_0^\infty(M)$, we have
\begin{eqnarray*}
|\nabla u_0|^2\leq C u_0.
\end{eqnarray*}
Then, arguing in the same way as \cite[Lemma 4.3]{Wang18}, we have that, fixing any $\varepsilon\in(0,1)$, there is a constant $C$, such that
\begin{eqnarray*}
|\nabla u|^2\leq C u \quad \text{ on }\quad M\times [t_0-\varepsilon,t_0).
\end{eqnarray*}
In view of Lemma \ref{someothergradientl2nonsense}, if we can prove
\begin{eqnarray}\label{grad u pw convergence}
\frac{|\nabla u|^2}{u}(\cdot,t)\rightarrow \frac{|\nabla u_0|^2}{u_0}\quad \text{ pointwise on $M$ as } t\rightarrow t_0-,
\end{eqnarray}
then the current proposition follows from the bounded convergence theorem. We shall then prove (\ref{grad u pw convergence}) below.

Obviously, (\ref{grad u pw convergence}) is true on $\{u_0>0\}$, since $u\rightarrow u_0$ and $|\nabla u|^2\rightarrow |\nabla u_0|^2$ uniformly, because of their evolution equations, and because of Lemma \ref{coarsebound}. We shall then consider a point $x\in M$ such that $u_0(x)=0$. Since $\sqrt{u_0}\in C_0^\infty(M)$ is a nonnegative smooth function, we have that $|\nabla\sqrt{u_0}|(x)=0$ and $|\nabla^2\sqrt{u_0}|\leq C$ everwhere on $M$. It then follows that
\begin{eqnarray*}
|\nabla \sqrt{u_0}|(y)\leq C\dist_{t_0}(x,y),\quad \sqrt{u_0}(y)\leq C\dist^2_{t_0}(x,y),\quad \text{ for all }\quad y\in M,
\end{eqnarray*}
in other words,
\begin{eqnarray}\label{u_0 growth}
u_0(y)\leq C \dist^4_{t_0}(x,y)\quad \text{ for all }\quad y\in M.
\end{eqnarray}
Hence, for any $t\in [t_0-1,t_0)$, we may compute using (\ref{u_0 growth})
\begin{eqnarray*}
u(x,t)&=&\int_M K(y,t_0\,|\,x,t)u_0(y)dg_{t_0}(y)
\\
&\leq& \int_M C\dist^4_{t_0}(x,y)\cdot\frac{C}{(t_0-t)^{\frac{n}{2}}}\exp\left(-\frac{\dist^2_{t_0}(x,y)}{C(t_0-t)}\right)dg_{t_0}(y)
\\
&\leq&(t_0-t)^2\int_M\left(\frac{\dist_{t_0}(x,y)}{\sqrt{t_0-t}}\right)^4\cdot\frac{C}{(t_0-t)^{\frac{n}{2}}}\exp\left(-\frac{\dist^2_{t_0}(x,y)}{C(t_0-t)}\right)dg_{t_0}(y)
\\
&\leq& C(t_0-t)^2.
\end{eqnarray*}
Applying the above result to (\ref{EKNT_gradient}), we have that if $t_0-t$ is small is small enough, then
\begin{eqnarray*}
\frac{|\nabla u|^2}{u}(x,t)&\leq& \frac{C}{t_0-t}(u+u\log^2 u)\leq \frac{C}{t_0-t}(u+u^{\frac{3}{4}})
\\
&\leq&\frac{C}{t_0-t}\big((t_0-t)^2+(t_0-t)^{\frac{3}{2}}\big)\leq C(t_0-t)^{\frac{1}{2}}
\\
&\rightarrow &0=\frac{|\nabla u_0|^2}{u_0}(x).
\end{eqnarray*}
This finishes the proof.
\end{proof}

\begin{proof}[Proof of Theorem \ref{forwardlimit}]
This theorem is now but a consequence of Proposition \ref{ulogunonsense}, Proposition \ref{grad u l2}, and the definition of $\bW(g(t),u(\cdot,t),T-t)$.
\end{proof}

\begin{proof}[Proof of Theorem \ref{nu-functional}]
Combining Theorem \ref{backwardlimit} and Theorem \ref{forwardlimit}, we have that, for any $t_0\in(-\infty,0]$, for any $\tau_0>0$, and for any $u_0:M\rightarrow\mathbb{R}$ satisfying $u_0\geq 0$, $\sqrt{u_0}\in C_0^\infty(M)$, and $\int_M u_0dg_0=1$, it holds that $$\bW(g(t_0),u_0,\tau_0)\geq \mu_\infty.$$ Hence we have $$\inf_{t\leq 0}\nu(g(t))\geq \mu_\infty.$$

To see that the equality in (\ref{nu_nonsense_00}) holds, recall that $(M,g_i(-1))\rightarrow (M_\infty,g_\infty(-1))$ in the smooth Cheeger-Gromov sense, and that the $\mu$-functional is upper semi-continuous with respect to the Cheeger-Gromov convergence (see Lemma 6.28 in \cite{RFV1}; this fact can be easily observed by taking an arbitrary compactly supported function on the limit, and pull it back to the sequence using the defining diffeomorphisms of the Cheeger-Gromov convergence), and these two facts imply
\begin{eqnarray*}
\limsup_{i\rightarrow\infty}\mu(g(-\tau_i),\tau_i)=\limsup_{i\rightarrow\infty}\mu(g_i(-1),1)\leq \mu(g_\infty(-1),1)=\mu_\infty.
\end{eqnarray*}
This then finishes the proof.
\end{proof}

\section{Synthesis with the classical cases}

In this section, we consider some classical cases in which Perelman's asymptotic shrinker exists. We shall see that, in all these cases, Assumption B is satisfied. Hence, Perelman's asymptotic shrinker is identical to Bamler's tangent flow at infinity. We consider an ancient solution $(M,g(t))_{t\in(-\infty,0]}$ with bounded curvature within each compact time interval, satisfying \emph{either one} of the following conditions:
\begin{enumerate}[(1)]
    \item $g(t)$ satisfies a Type I curvature bound, that is, there is a constant $C$ such that $$|\Rm_{g(t)}|\leq\frac{C}{|t|}\quad\text{ for all }\quad t\in(-\infty,0).$$
    \item $g(t)$ satisfies Hamilton's trace Harnack, that is, $$\frac{\partial R}{\partial t}+2\langle X,\nabla R\rangle+2\Ric(X,X)\geq 0\quad\text{ for all vector field } X,$$
    and there is a constant $C$ such that
   $$|\Rm|\leq CR\quad\text{ everywhere on }\quad M\times(-\infty,0].$$
\end{enumerate}

Except for the assumption above, we would also like to impose a $\kappa$-noncollapsing assumption. This is actually equivalent to
\begin{eqnarray*}
\mathcal{N}_{p_0,t_0}(\tau)\geq-Y\quad\text{ for all }\quad \tau>0,
\end{eqnarray*}
where $(p_0,t_0)$ is an arbitrarily fixed point on $M\times(-\infty,0]$. 

\begin{Theorem}
Under the assumptions of this section, Perelman's asymptotic shrinker is identical to Bamler's tangent flow at infinity.
\end{Theorem}
\begin{proof}
Let us fixed a point $p_0\in M$ and a sequence $\tau_i\nearrow\infty$. Let $\ell$ be the reduced distance based at $(p_0,0)$ and let $\{(p_i,-\tau_i)\}_{i=1}^\infty$ be a sequence of $\ell$-centers of $(p_0,0)$. We will still use $\ell_i(\cdot,\tau):=\ell_i(\cdot,\tau_i\tau)$ and $g_i(t):=\tau_i^{-1}g(\tau_it)$ to denote the scaled reduced distances and the scaled Ricci flows. Note that $\ell_i$ is the reduced distance based at $(p_0,0)$ with respect to the flow $g_i$. Recall that by \cite{Per02} and \cite{N10}, the following formulas hold everywhere on $M\times(-\infty,0)$ in the barrier sense or in the sense of distribution
\begin{align}\label{nonsense7_1_0}
\big|\nabla\ell_i(\cdot,|t|)\big|^2+R_i(\cdot,t)\leq\frac{C\ell_i(\cdot,|t|)}{|t|},&
\\\label{nonsense7_1_1}
-2\frac{\partial\ell_i(\cdot,|t|)}{\partial t}+\big|\nabla\ell_i(\cdot,|t|)\big|^2-R_i(\cdot,t)+\frac{\ell_i(\cdot,|t|)}{|t|}=0.&
\end{align}
Combining (\ref{nonsense7_1_0}) and (\ref{nonsense7_1_1}), we have
\begin{eqnarray}\label{nonsense7_1_2}
\left|\frac{\partial\ell_i(\cdot,|t|)}{\partial t}\right|\leq\frac{C\ell_i(\cdot,|t|)}{|t|}.
\end{eqnarray}
Since $\ell_i(p_i,1)\leq\frac{n}{2}$, we may integrate (\ref{nonsense7_1_2}) and obtain that, for any $\varepsilon\in(0,1)$, there is a constant $C=C(\varepsilon)$, such that
\begin{eqnarray*}
\ell_i(p_i,|t|)\leq C\quad\text{ for all }\text t\in[-\varepsilon^{-1},-\varepsilon].
\end{eqnarray*}
It then follows from (\ref{nonsense7_1_0}) that for $t\in[-\varepsilon^{-1},-\varepsilon]$, all the functions $\ell_i(\cdot,|t|)$ have uniformly quadratic growth bounds around $p_i$, and all curvatures $|\Rm_{g_{i,t}}|$ have uniform growth bounds around $p_i$ (indeed, uniform upper bound in the Type I case). Hence, $(M,g(t))_{t\in(-\infty,0]}$ satisfies a stronger version of Assumption B, with the $[-2,-1]$ time interval in formula (\ref{smoothconvergence}) replaced by $[-\varepsilon^{-1},-\varepsilon]$ for any $\varepsilon\in(0,1)$, and we will let $(M_\infty,g_\infty(t),\ell_\infty)_{t\in(-\infty,0)}$ denote the asymptotic shrinker as given by Proposition \ref{shrinkerstructure}.

Let $(\nu^i_s)_{s\in(-\infty,0]}$ be the conjugate heat kernel based at $(p_0,0)$ on the Ricci flow $g_i$, then we have
\begin{eqnarray*}
\big((M,g_i(t))_{t\in(-\infty,0]},(\nu^i_s)_{s\in(-\infty,0]}\big)\xrightarrow{\makebox[1cm]{$\mathbb{F}$}}\mathcal{X},
\end{eqnarray*}
where $\mathcal{X}$ is a conjugate flow pair. By the arguments in sections 5---7, we have that, for any $\varepsilon\in(0,1)$, it holds that
\begin{eqnarray*}
\mathcal{X}_{[-\varepsilon^{-1},-\varepsilon)}=\big((M_\infty,g_\infty(t))_{t\in[-\varepsilon^{-1},-\varepsilon)},(\nu^\infty_s)_{s\in[-\varepsilon^{-1},-\varepsilon)}\big),
\end{eqnarray*}
where $(\nu^\infty_s)_{s\in(-\infty,0)}$ is a conjugate heat flow on $(M_\infty,g_\infty)$ made of a shrinker potential. This then finishes the proof.
\end{proof}

\section{Ancient Ricci flows with bounded entropy and smooth tangent flows}

In this section we prove Theorem \ref{Thm_main_reciprocal}. We would like to emphasize again that the arguments in this section are true only if the results in \cite{Bam20c} (especially Theorem 1.6) are true for noncomapct Ricci flows with bounded geometry within each compact time interval (or if the ancient solution in question is on a closed manifold).

Let $(M,g(t))_{t\in(-\infty,0]}$ be an ancient solution with bounded geometry within each compact time interval. Let $p_0\in M$ be a fixed point and let $Y>0$ be a constant such that
\begin{eqnarray}\label{bddentropychpt10}
\mathcal{N}_{p_0,0}(\tau)>-Y\quad\text{ for all }\quad \tau>0.
\end{eqnarray}
Let $\{\tau_i\}_{i=1}^\infty$ be a sequence of positive numbers such that $\tau_i\nearrow\infty$, and we shall denote
\begin{eqnarray*}
g_i(t):= \tau_i^{-1}g(\tau_it),\quad \nu^i_t:=\nu_{p_0,0\,|\,\tau_it},\quad\text{ for all }\quad t<0.
\end{eqnarray*}
Let $\big((M_\infty,g_\infty(t))_{t\in(-\infty,0)},(\nu^\infty_s)_{s\in(-\infty,0)}\big)$ be the smooth tangent flow mentioned in the statement of the theorem, that is,
\begin{eqnarray*}
\big((M,g_i(t))_{t\in(-\infty,0]},(\nu^i_s)_{s\in(-\infty,0]}\big)\xrightarrow{\makebox[1cm]{$\mathbb{F}$}}\big((M_\infty,g_\infty(t))_{t\in(-\infty,0)},(\nu^\infty_s)_{s\in(-\infty,0)}\big).
\end{eqnarray*}

Since $(M_\infty,g_\infty(t))$ is smooth, we have, by \cite[Theorem 1.6]{Bam20c}, that the above convergence is smooth. Precisely, this means that one can find an increasing open sets in space-time $U_1\subset U_2\subset\cdots\subset M_\infty\times(-\infty,0)$ with $\cup_{i=1}^\infty U_i=M_\infty\times(-\infty,0)$, open subsets $V_i\subset M\times(-\infty,0)$, time-preserving diffeomorphisms $\psi_i:U_i\rightarrow V_i$, and a sequence $\varepsilon_i\searrow 0$, such that \cite[Theorem 9.31]{Bam20b} holds.

Let us then fix a point $x_\infty\in M_\infty$ and a positive radius $D$ (which could be very large), such that
$$\nu^\infty_{-1}\big(B_{g_{\infty,-1}}(x_\infty,D)\big)\geq \frac{1}{2}.$$
Note that this is possible since $\nu^\infty_{-1}$ is a probability measure. Since $\cup_{i=1}^\infty U_i=M_\infty\times(-\infty,0)$, we have that $$B_{g_{\infty,-1}}(x_\infty, D)\times\{-1\}\subset U_i\quad\text{ and } \quad \nu^i_{-1}\Big(\psi_{i,-1}\big(B_{g_{\infty,-1}}(x_\infty,D)\big)\Big)\geq\frac{1}{3}$$
for all $i$ large enough. Here (and below) $\psi_{i,-1}$ denotes the $t=-1$ time slice of $\psi_i$. Since $\psi_i$ is almost isometry, we can find $D_0\geq D$, such that
\begin{eqnarray*}
\nu^i_{-1}\big(B_{g_{i,-1}}(\psi_{i,-1}(x_\infty),D_0)\big)\geq \frac{1}{3}\quad\text{ for all $i$ large enough.} 
\end{eqnarray*}

Hence, by Proposition \ref{measureaccumulationofHcenter}, letting $(z_i,-1)$ be an $H_n$-center of $(p_0,0)$ with respect to the Ricci flow $g_i(t)$, we have
\begin{eqnarray}\label{10nonsense00001}
\dist_{g_{i,-1}}\big(z_i,\psi_{i,-1}(x_\infty)\big)\leq D_0+\sqrt{3AH_n}\quad \text{ for all $i$ large enough}.
\end{eqnarray}
By Proposition \ref{H_n_l_n} and (\ref{bddentropychpt10}), we can find a constant $C$ depending only on $Y$, such that
\begin{eqnarray}\label{10nonsense00002}
\dist_{g_{i,-1}}(z_i,p_i)\leq C\quad \text{ for all $i$ large enough},
\end{eqnarray}
where $(p_i,-1)$ is an $\ell$-center of $(p_0,0)$ with respect to the Ricci flow $g_i(t)$.

Combining (\ref{10nonsense00001}) and (\ref{10nonsense00002}), we have 
\begin{eqnarray*}
\dist_{g_{i,-1}}\big(\psi_{i,-1}(x_\infty),p_i\big)\leq C\quad \text{ for all $i$ large enough}.
\end{eqnarray*}
Arguing in the same way as in the proof of Proposition \ref{LUTypeI}, we have that $(M,g(t))$ is locally uniformly Type I along $(p_i,-\tau_i)$, and the smooth Cheeger-Gromov-Hamilton limit of $\{(M,g_i(t),p_i)_{t\in[-2,-1]}\}_{i=1}^\infty$ is an asymptotic shrinker in the sense of Perelman. Obviously, this limit must be $(M_\infty,g_\infty(t))_{t\in[-2,-1]}$.

Next, we shall show that $\nu^\infty_t$ is a conjugate heat flow made of a shrinker potential. Because of (\ref{bddentropychpt10}), \cite[Proposition 6.1]{Bam20c} (see Proposition \ref{almostselfsimilar} above) implies that, there is a sequence $\delta_i\searrow 0$, such that
\begin{eqnarray}\label{selfsimilar}
\int_{-\delta_i^{-1}}^{-\delta_i}\int_M|t|\left|\,\Ric_i+\nabla^2f_i-\frac{1}{2|t|}g_i\,\right|^2d\nu^i_tdt<\delta_i,
\end{eqnarray}
where $$d\nu^i_t:=(4\pi|t|)^{-\frac{n}{2}}e^{-f_i}dg_{i,t}.$$
Since, by \cite[Theorem 9.31]{Bam20b}, the conjugate heat kernel and the Ricci flow converge smoothly along the sequence, we may take a limit for (\ref{selfsimilar}) and apply Fatou's lemma to obtain
\begin{eqnarray*}
\int_{-\delta^{-1}}^{-\delta}\int_M|t|\left|\Ric_\infty+\nabla^2f_\infty-\frac{1}{2|t|}g_\infty\right|^2d\nu^\infty_tdt=0, \quad\text{ for any }\quad \delta\in(0,1),
\end{eqnarray*}
where $$d\nu^\infty_t:=(4\pi|t|)^{-\frac{n}{2}}e^{-f_\infty}dg_{\infty,t}.$$
This finishes the proof Theorem \ref{Thm_main_reciprocal}.

\def \MM {\mathcal{M}}

\appendix 
\appendixpage

\section{Logarithmic Sobolev inequalities for noncompact Ricci flows}
We prove Proposition \ref{Hein-Naber-log-Sobolev}, a slightly stronger version of Hein-Naber's \cite{HN14} logarithmic Sobolev inequalities and Poincar\'e inequalities. Let $(M^n,g(t))_{t\in [-T,0]}$ be a complete Ricci flow with bounded curvature. The boundedness of curvature will be the only assumption which we make in the proof. We shall prove the logarithmic Sobolev inequality (\ref{ineq: log-Sobolev}); the Poincar\'e inequality (\ref{ineq: poincare}) is similar (indeed, much easier), and the details of its proof is left to the readers.

Let $K(\cdot,\cdot\,|\,\cdot,\cdot)$ be the minimal heat kernel coupled with $(M,g(t))$ as introduced in section 2.1. Let $o\in M$ be a fixed point and $\nu_t:=\nu_{o,0\,|\,t}$ be the conjugate heat kernel based at $(o,0)$, where $d\nu_{x,t\,|\,s}:=K(x,t\,|\,\cdot,s)dg_s$ for $-T\leq s\leq t\leq 0$ and $x\in M$. By a shifting of time and a parabolic scaling, it suffices to show that
\begin{eqnarray}
    \int_M v \log v\, d\nu_{-1}
    \le \int_M \frac{|\nabla v|^2}{v}\, d\nu_{-1},
\end{eqnarray}
for any function $v$ satisfying $v\ge 0$, $\sqrt{v}\in C_0^{\infty}(M)$, and $\displaystyle \int_M v\,d\nu_{-1}=1$.
Now fix such a function $v$ and let $u: M\times[-1,0]\rightarrow\mathbb{R}$ be a slolution to
\[
    \Box u = 0\quad \text{ on }\quad M\times [-1,0],
    \quad u_{-1}=v,
\]
in other words, $\displaystyle u(x,t)=\int_MK(x,t\,|\,\cdot,-1)vdg_{-1}$ for all $(x,t)\in M\times[-1,0]$, where we have denoted $u_t:=u(\cdot,t)$. 

We shall regard $\bar g = g(0)$ as the fixed background metric on $M$. Then, obviously, we have the following from our assumptions:
\begin{enumerate}[(\text{a}1)]
    \item There is a positive constant $\Lambda$, such that $$ \sup_{M\times [-1,0]}|{\Rm}_{g_t}|\le \Lambda. $$
    \item There is a positive constant $C_0$ depending on $\Lambda$, such that $$
    C_0^{-1} \bar g \le g(t) \le C_0 \bar g\quad\text{ for all }\quad t\in [-1,0].$$
    \item There is a positive constant $D_0>1$ depending on the function $v$, such that $${\rm spt}\, v\subset B_{\bar g}(o,D_0),\quad
\sup_M v \le D_0.$$
    \item There is a positive constant $c_0$ depending on $D_0$, such that $$\inf_{x\in B_{\bar g}(o,2D_0)}{\rm Vol}_{\bar g}(B_{\bar g}(x,1)) \ge c_0.$$
    \item There are positive constants $r_0$ and $c_1$, and a point $x_0\in\spt v$, such that
    $$B_{\bar g}(x_0,r_0)\subset \spt v\quad\text{ and }\quad \min_{B_{\bar g}(x_0,r_0)} v \ge c_1.$$
\end{enumerate}
Here (a5) is true because $\displaystyle\int_M v\, d\nu_{-1}=1>0$.

\begin{Lemma}
\label{lem: decay on u}
For all $(x,t)\in (M\setminus B_{\bar g}(o,2D_0))\times (-1,0],$ we have
\[
    \frac{1}{C_1(1+t)^{n/2}}
    \exp\left\{
        - C_1\frac{\dist_{\bar g}^2(x,o)}{(1+t)}
    \right\}
    \le
    u(x,t) 
    \le \frac{C_2}{(1+t)^{n/2}}
    \exp\left\{
        - \frac{\dist_{\bar g}^2(x,o)}{C_2(1+t)}
    \right\},
\]
where
$C_1$ depends on $C_0$, $\Lambda$, $c_0$, $c_1$, $r_0$, and $D_0$;
$C_2$ depends on $C_0$, $\Lambda$, $D_0,$ and $c_0.$
\end{Lemma}
\begin{proof}

First, we prove the lower bound. 
By \cite[Theorem 26.31]{RFV3} and the Bishop-Gromov comparison theorem, we have that, for all $(x,t)\in M\times(-1,0]$ and for all $y\in B_{\bar g}(x_0,r_0)$, 
\begin{eqnarray*}
K(x,t\,|\,y,-1)&\geq&c\min\left\{\frac{1}{\Vol_{\bar g} B_{\bar g}\left(x,\sqrt{1+t}\right)},\frac{1}{\Vol_{\bar g} B_{\bar g}\left(y,\sqrt{1+t}\right)}\right\}\exp\left(-\frac{\dist_{\bar g}^2(x,y)}{c(1+t)}\right)
\\
&\geq&\frac{c}{(1+t)^{\frac{n}{2}}}\exp\left(-\frac{\dist_{\bar g}^2(x,y)}{c(1+t)}\right).
\end{eqnarray*}
Note that if $x\notin B_{\bar g}(o,2D_0)$ and $y\in B_{\bar g}(x_0,r_0)\subset B_{\bar g}(o,D_0)$, then
\[
    \dist_{\bar g}(x,y)
    \le \dist_{\bar g}(x,o) + D_0 
    \le \tfrac{3}{2}\dist_{\bar g}(x,o).
\]
Hence, for all $x\notin B_{\bar g}(o,2D_0)$, we have
\begin{align*}
    u(x,t)
    &=\int_M K(x,t\,|\,y,-1)v(y)\, dg_{-1}(y)
    \\
    &\ge C_0^{-1}\int_{B_{\bar g}(x_0,r_0)} K(x,t\,|\,y,-1)v(y) \, d\bar g(y)\\
    &\ge 
    \frac{c_1c}{C_0(1+t)^{n/2}}
    {\rm Vol}_{\bar g}B_{\bar g}(x_0,r_0)
    \exp\left\{- \frac{9\dist_{\bar g }^2(x,o)}{4c(1+t)} \right\}\\
    &\ge
     \frac{1}{C_1(1+t)^{n/2}}
    \exp\left\{
        - C_1\frac{\dist_{\bar g}^2(x,o)}{(1+t)}
    \right\}.
\end{align*}

Next, we prove the upper bound.
By \cite[Corollary 26.26]{RFV3}, for all $x,y\in M$ and $t\in (-1,0]$, we have
\[
    K(x,t\,|\,y,-1)
    \le C\min\left\{\frac{1}{\Vol_{\bar g} B_{\bar g}\left(x,\sqrt{\tfrac{1+t}{2}}\right)},\frac{1}{\Vol_{\bar g} B_{\bar g}\left(y,\sqrt{\tfrac{1+t}{2}}\right)}\right\}\exp\left(-\frac{\dist_{\bar g}^2(x,y)}{C(1+t)}\right),
\]
for some $C$ depending on the curvature bound.
By the Bishop-Gromov comparison theorem, we have, for all $y\in B_{\bar g}(o,D_0)$ and $t\in(-1,0]$,
\begin{align*}
    \Vol_{\bar g}B_{\bar g}\left(y,\sqrt{\tfrac{1+t}{2}}\right)\geq c(1+t)^{\frac{n}{2}}\Vol_{\bar g}B_{\bar g}(y,1)\geq cc_0(1+t)^{\frac{n}{2}},
\end{align*}
where $c$ depends on the Ricci curvature bound of $\bar g$. Hence, for all $x\notin B_{\bar g}(o, 2D_0)$ and $y\in B_{\bar g}(o,D_0)$, we have
\begin{eqnarray}\label{anothernonsensegaussianupperbound}
    K(x,t\,|\, y,-1)\le \frac{C}{(1+t)^{n/2}}
    \exp\left\{-\frac{\dist_{\bar g}^2(x,y)}{C(1+t)}\right\}
    \le 
    \frac{C}{(1+t)^{n/2}}
    \exp\left\{-\frac{\dist_{\bar g}^2(x,o)}{2C(1+t)}\right\}.
\end{eqnarray}
It then follows that
\begin{align*}
    u(x,t)
    &= \int_M K(x,t\,|\,y,-1) v(y)\, dg_{-1}(y)\\
    &\leq C_0\int_{B_{\bar g}(o,D_0)} K(x,t\,|\,y,-1)v(y)\,d\bar g(y)
    \\
    & \le C_0D_0 {\rm Vol}B_{\bar g}(o,D_0) 
     \frac{C}{(1+t)^{n/2}}
    \exp\left\{
        - \frac{\dist_{\bar g}^2(x,o)}{C(1+t)}
    \right\}\\
    &\le \frac{C_2}{(1+t)^{n/2}}
    \exp\left\{
        - \frac{\dist_{\bar g}^2(x,o)}{C_2(1+t)}
    \right\}.
\end{align*}

\end{proof}

Let $\eta:[0,\infty)\rightarrow[0,1]$ be the standard decreasing smooth cutoff function such that $\eta|_{[0,1]}=1$, $\eta|_{[2,\infty)}=0$, and $|\eta'|+|\eta''|\le 10.$
Let us define a time-dependent cutoff function
\[
\phi^A(x,t) := \eta(\dist_t(x,o)/A),
\]
for all $(x,t)\in M\times [-1,0].$
 By the smoothness and completeness of the Ricci flow, we can find a positive number $\rho_0>0$, such that
\[
    \Ric \le (n-1)\rho_0^{-2},\quad
    \text{ on }\quad  \bigcup_{t\in[-1,0]}B_{g_t}(o,\rho_0)\times\{t\}.
\]
According to \cite[8.3]{Per02}, we have
\[
    \Box \dist_t(o,\cdot) \ge 
    - \tfrac{5}{3}(n-1)\rho_0^{-1} \quad\text{ on }\quad M\setminus B_t(o,\rho_0)
\]
for all $t\in[-1,0]$; this differential inequality should be understood in the barrier sense. Since the parameter $A$ in the cut-off function $\phi^A$ will eventually be taken to infinity, we shall, throughout the whole proof, assume that $A\gg \rho_0+2C_0D_0.$
It then follows that
\[
    \Box \phi^A
    = \frac{1}{A}\eta'(\dist_t/A) \Box \dist_t
    - \frac{1}{A^2}\eta''(\dist_t/A)
    \le C_n\left(\frac{1}{\rho_0A} + \frac{1}{A^2}\right),
\]
in the barrier sense.

Let $U_t := u_t\log u_t - D_0\log D_0$ for $t\in[-1,0]$, then we have 
$-D_0\log D_0-1/e \le U_t \le 0$ everywhere on $M\times [-1,0]$, where the latter inequality is because $\max\{u,1\}\leq D_0$ everywhere.
Since $\phi^A(\cdot,t)$ is compactly supported
for all $t\in [-1,0],$ we have
\begin{align*}
    &\partial_t \int U_t \phi^A\, d\nu_t
    = \int \Box (U_t \phi^A)\, d\nu_t \nonumber\\
    =&\ \int \left(- \frac{|\nabla u|^2}{u} \phi^A 
    - 2 \langle \nabla U_t,\nabla \phi^A \rangle
    + U_t \Box \phi^A\right) d\nu_t,
\end{align*}
where we have used the fact $\Box (u\log u)=-\tfrac{|\nabla u|^2}{u}$. Integrating from $-1$ to $0$, we have
\begin{align}
    \int U_{-1}\phi^A\, d\nu_{-1}
    - U_0(o)
    &=  \int_{-1}^0\int \frac{|\nabla u|^2}{u} \phi^A \, d\nu_t dt
    + 2 \int_{-1}^0 \int \langle \nabla U_t,\nabla \phi^A \rangle\, d\nu_tdt
    - \int_{-1}^0 \int U_t \Box\phi^A\, d\nu_tdt \nonumber\\
    &=:\int_{-1}^0\int \frac{|\nabla u|^2}{u} \phi^A \, d\nu_t dt
    + \rom{1} + \rom{2}.
    \label{eq: int of derivatives}
\end{align}

We shall estimate ${\rm I}$ and ${\rm II}$ next. Since $-D_0\log D_0-1/e\le U_t\le 0,$ we have
\[
{\rom{2}}
\le C_n\left(\frac{1}{\rho_0A} + \frac{1}{A^2}\right)\left(D_0\log D_0+\frac{1}{e}\right)\leq\frac{C}{A}.
\]
Note that here we have only one-side bound for ${\rm II}$, yet it is sufficient for our purpose. Now we consider $\rom{1}.$
By the gradient estimate of heat equation \cite{Zhq07, BCP10} (c.f. \cite[Lemma 2.4]{Zhang21}) and the Shi-type estimates, we have 
\begin{gather}\label{QSZhangtypeestimateforHE}
        \frac{|\nabla u|^2}{u^2}
    \le \frac{1}{1+t} \log \frac{D_0}{u} \quad \text{ on }\quad M\times(-1,0],
    \\
    |\nabla u|\leq \Lambda_1, \quad  |\nabla^2 u|\leq \Lambda_2\quad \text{ on }\quad M\times[-1,0],\label{ShiTypeestimateforHE}
\end{gather}
where the derivative estimates follow by the maximum principle and the fact that the derivatives of the initial condition $v 
$ are bounded.
Thus, by (\ref{QSZhangtypeestimateforHE}) we have, for all $t\in(-1,0]$,
\begin{align*}
    \int \langle \nabla U_t,\nabla \phi^A \rangle\, d\nu_t
    &\le \int |\nabla \phi^A|
    |\nabla u| (1+|\log u|)\, d\nu_{t}\\
    & \le \frac{C}{A(1+t)^{1/2}} 
    \int u \left(\log \tfrac{D_0}{u}\right)^{1/2}(1+|\log u|)\, d\nu_t\\
    &\le  \frac{C}{A(1+t)^{1/2}},
\end{align*}
where we have applied the fact $u \left(\log \tfrac{D_0}{u}\right)^{1/2}(1+|\log u|)\leq C(D_0)$, and this is because $u$ is bounded by $D_0$. Hence, we have
\[
    {\rom{1}}
    \le \int_{-1}^0\frac{C}{A(1+t)^{1/2}}dt \le C/A.
\]
Letting $A\to \infty$ in \eqref{eq: int of derivatives}, we have
\[
    \int U_{-1} d\nu_{-1}
    \le U_0(o) + \int_{-1}^0\int \frac{|\nabla u|^2}{u}(x,t) \, d\nu_t(x) dt,
\]
which, by the fact that 
    $\displaystyle u(o,0) = \int u_{-1}\, d\nu_{-1}=1$, $
    U_0(o) = -D_0\log D_0$, 
and by the definition of $U_t$, is equivalent to
\begin{equation}
\label{eq: log sobolev first int by parts}
    \int v\log v \, d\nu_{-1}
    \leq \int_{-1}^0 dt\int_M \frac{|\nabla u|^2}{u}(x,t)\, d\nu_t(x).
\end{equation}

To handle the integrand above,let us fix $t\in (-1,0] $ and $ x\in B_t(o,A/2).$
Integrating by parts, we have
\begin{align*}
   & \partial_s 
    \int \frac{|\nabla u|^2}{u}(y,s)\phi^A(y,s)\, d\nu_{x,t|s}(y)
    = \int \Box\left(\frac{|\nabla u|^2}{u}\phi^A\right)\, d\nu_{x,t|s}\\
=&\ \int \Box\frac{|\nabla u|^2}{u} \phi^A\, d\nu_{x,t|s}
- 2\int \nabla\frac{|\nabla u|^2}{u} \cdot\nabla \phi^A \, d\nu_{x,t|s}
 + \int \frac{|\nabla u|^2}{u} \Box \phi^A \, d\nu_{x,t|s}.
\end{align*}
Fixing some small $\theta\in (0,t)$ and using the fact that $\Box  \frac{|\nabla u|^2}{u}=-\frac{2}{u}\left|\nabla^2u-\frac{\nabla u\otimes\nabla u}{u}\right|^2\le 0$, we have, by integrating the above inequality from $-1+\theta$ to $t$
\begin{align*}
    &\frac{|\nabla u|^2}{u}(x,t)
    - \int \frac{|\nabla u|^2}{u} \phi^A
    \, d\nu_{x,t|-1+\theta}\\
    \le&\  -2 \int_{-1+\theta}^t ds
    \int \nabla\frac{|\nabla u|^2}{u} \cdot\nabla \phi^A \, d\nu_{x,t|s}
 + \int_{-1+\theta}^t ds\int \frac{|\nabla u|^2}{u} \Box \phi^A \, d\nu_{x,t|s}\\
 =:&\ {\rom{3}}(x,t)+{\rom{4}}(x,t).
\end{align*}

To estimate $\rom{3}$, we may apply (\ref{QSZhangtypeestimateforHE}) and (\ref{ShiTypeestimateforHE}) to obtain
\begin{align*}
    \left|
    \nabla \frac{|\nabla u|^2}{u}
    \right|(y,s)
    &\le 2 |\nabla^2u| |\nabla \ln u|
    + |\nabla u||\nabla \ln u|^2\\
    &\le 
    \frac{2\Lambda_2+\Lambda_1}{1+s}
    \left(\left(\log \tfrac{D_0}{u}\right)^{1/2}+\log \tfrac{D_0}{u}
    \right).
\end{align*}
Since $A\gg 2C_0D_0,$ by the lower bound on $u$ in Lemma \ref{lem: decay on u}, we have, for all $s\in [-1+\theta,t)$ and $y\in \spt \nabla_{g_s} \phi^A(\cdot,s)\subset M\setminus B_{g_s}(o, A)\subset M\setminus B_{\bar g}(o, C_0^{-1}A)$, 
\[
     \left|
    \nabla \frac{|\nabla u|^2}{u}
    \right|(y,s)
    \le 
    \frac{C}{\theta} 
    \left(\log\frac{D_0}{C_1} + \frac{n}{2}\log(1+s) 
    +C_1 \frac{d_{\bar g}^2(x,o)}{1+s}
    \right)
    \le C_\theta A^2,
\]
where $C_\theta$ depends on $\theta,C_1,\Lambda$, and $D_0.$ 
Hence, we have
\begin{align*}
    \int_{-1+\theta}^0dt\int_M \left|{\rom{3}}(x,t)\right|\, d\nu_t(x)
    & \le \int_{-1+\theta}^0dt\int_M \int_{-1+\theta}^t
    C_\theta A\nu_{x,t\,|\,s}(M\setminus B_{\bar g}(o,C_0^{-1}A))
    \, dsd\nu_t(x)\\
    &=  \int_{-1+\theta}^0dt\int_{-1+\theta}^t C_\theta A\nu_s(M\setminus B_{\bar g}(o,C_0^{-1}A))\, ds\\
    &\le C_\theta A\exp(- A^2/C_\theta).
\end{align*}
The last inequality above follows from a similar Gaussian upper bound as (\ref{anothernonsensegaussianupperbound}) and the Bishop-Gromov comparison theorem.

For $\rom{4},$ (\ref{QSZhangtypeestimateforHE}) implies that
\[
    \frac{|\nabla u|^2}{u}(x,t)
    \le \frac{1}{1+t} u \log \frac{D_0}{u}
    \le \frac{C}{\theta}.
\]
It follows that
\begin{align*}
    {\rm IV}(x,t)
    &\le \int_{-1+\theta}^t ds
    \int \frac{C}{\theta}C_n\left(\frac{1}{\rho_0A} + \frac{1}{A^2}\right)\, d\nu_{x,t|s}
    \le (t+1-\theta)\frac{C}{\theta\rho_0A}.
\end{align*}
Combining with the estimates above, we have
\begin{align*}
 \int_{-1+\theta}^0dt \int 
\frac{|\nabla u|^2}{u}(x,t)\, d\nu_t(x)
\le & \int_{-1+\theta}^0dt \int d\nu_t(x) \int \frac{|\nabla u|^2}{u} \phi^A
    \, d\nu_{x,t\,|\,-1+\theta}\\
& + \ C_\theta A\exp(- A^2/C_\theta)
+\frac{C}{\theta\rho_0A}.
\end{align*}
Taking $A\to \infty$ and applying the reproduction formula, we have
\[
    \int_{-1+\theta}^0dt \int 
\frac{|\nabla u|^2}{u}(x,t)\, d\nu_t(x)
\le (1-\theta)\int \frac{|\nabla u|^2}{u}  d\nu_{-1+\theta}.
\]
Letting $\theta\to 0$ and combining with \eqref{eq: log sobolev first int by parts},
\[
    \int u\log u \, d\nu_{-1}
    \le \int_{-1}^0 dt\int \frac{|\nabla u|^2}{u}(x,t)\, d\nu_t(x)
    \le \int \frac{|\nabla v|^2}{v}  d\nu_{-1}.
\]
Here we used the fact that
\[
    \lim_{t\rightarrow -1+}\int \frac{|\nabla u|^2}{u}  d\nu_{t}
    =\int \frac{|\nabla v|^2}{v}  d\nu_{-1}.
\]
The proof is parallel to Proposition \ref{grad u l2} and we omit the details.

\section{Convergence of heat kernels.}

Let $\{(M_i,g_i(t),o_i)_{t\in (-T_i,0]}\}_{i=1}^\infty$ be a sequence of complete Ricci flows with base point, and assume that each Ricci flow therein has bounded curvature within every compact time interval. Assume moreover that this sequence converges to $(M_\infty,g_\infty(t),o)_{t\in (-T_\infty,0]}$ in the pointed smooth Cheeger-Gromov-Hamilton sense, where $T_\infty=\limsup_{i\rightarrow\infty}T_i\in(0,\infty]$. Note that we do not make any assumption on the limit $(M_\infty,g_\infty(t),o)_{t\in (-T_\infty,0]}$.

By the definition of smooth convergence, we may find an increasing sequence of pre-compact open sets $U_i\subset M_\infty$ with $\cup_{i=1}^\infty U_i=M_{\infty}$, a sequence of diffeomorphisms $\Psi_i:U_i\to V_i\subset M_i$, and a sequence of positive numbers $\varepsilon_i\searrow 0$, such that
\begin{gather*}
\Psi_i(o)=o_i,
\\
    \|\Psi_i^*g_i - g_\infty\|_{C^{[\varepsilon_i^{-1}]}(U_i\times [-T_\infty+\varepsilon_i,0])}
    < \varepsilon_i,
\end{gather*}
where we let $-T_\infty+\varepsilon_i=-\varepsilon_i^{-1}$ in the case $T_\infty=\infty$.

Let $K_i(x,t\,|\,y,s)$ be the heat kernel coupled with $(M_i,g_i(t))$ as introduces in section 2.1. For any $x,y\in U_i$ and for any $ -T_\infty+\varepsilon_i\leq s<t\leq0$, we shall define
\[
    \bar{K}_i(x,t\,|\,y,s):=(\Psi_i^* K_i)(x,t\,|\,y,s)
    = K_i(\Psi_i(x),t\,|\, \Psi_i(y),s).
\]As a slight generalization of \cite{Lu12}, we shall prove the following.
\begin{Theorem}
\label{thm: HK conv under CGH}
There is a positive heat kernel $K_\infty$ coupled with $(M_\infty,g_\infty(t))$, such that, after passing to a subsequence, we have
\[
    \bar K_i \to K_\infty,
\]
in the $C_c^\infty$-topology, and the convergence is uniform on any compact subset of 
\[
    \MM:=
    \big\{\,
    (x,t,y,s)\,\big\vert\,
    x,y\in M_\infty, s,t\in (-T_\infty,0], s<t\,
    \big\}.
\]
\end{Theorem}

\begin{proof}
For any $A \gg 1,$ we define
\[
    \MM_{A} 
    :=
    \left\{
    (x,t,y,s)\in \MM\,\Big|\,
    x,y\in B_{g_{\infty,0}}(o,A),
    -T_\infty+\tfrac{1}{A^2}\le s \le t - \tfrac{1}{A^2} < t
    \le 0
    \right\},
\]
where once more we let $-T_\infty+\frac{1}{A^2}=-A^2$ if $T_\infty=\infty$.

Since $(M_\infty,g_\infty(t))_{t\in(-\infty,0]}$ is smooth and complete, for each $k\geq 0$, we have
\[
    |\nabla^k \Rm_{g_\infty}|
    \le  \Lambda_k,\quad
    \text{ on }\quad 
    B_{g_{\infty,0}}(o,10A)\times [-T_\infty+\tfrac{1}{A^2},0],
\]
for some $\Lambda_k=\Lambda_k(A)$.
Let us define
$$\bar g:= g_\infty(0),\quad
\bar g_i(t) := \Psi_i^* g_i(t),$$ 
where we will use $\bar{g}$ as the background metric on $M_\infty$. By the smooth convergence of the Ricci flows, we may find a large integer $\bar{i}$ depending on $A$, such that whenver $i\geq \bar{i}$, we have
\begin{align}
    |\nabla^k \Rm_{\bar g_i}|
    \le 2\Lambda_k, \quad&\text{ on }\quad
    B_{\bar g}(o,10A)\times [-T_\infty+\tfrac{1}{A^2},0],
\\
    C_0^{-1} \bar g
    \le \bar g_i
    \le C_0 \bar g,\quad
    &\text{ on }\quad B_{\bar g}(o,10A)\times [-T_\infty+\tfrac{1}{A^2},0],\label{equivalenceofthemetric}
\end{align}
where $C_0$ depends only on $\Lambda_0$, $A$, and $T_\infty$.
Let $r_0:=\tfrac{1}{10A}$, then for any $x\in B_{\bar g}(o,5A)$, we have, by the Bishop-Gromov comparison theorem,
\begin{eqnarray}\label{anotherlowervolumebounde}
    {\rm Vol}_{\bar g}
    \left[
        B_{\bar g}(x,r_0)
    \right]
    \ge c_1
    {\rm Vol}_{\bar g}
    \left[
        B_{\bar g}(o,r_0)
    \right]
    \ge c_0>0,
\end{eqnarray}
where $c_1$ depends on $A$ and $\Lambda_0$,
and $c_0$ depends on $c_1$  and the geometry of $\bar g$ near $o$.

Let us use
\[
    P^-(x,t,r) := B_{\bar g}(x,r)\times [t-r^2,t],\quad
    P^+(x,s,r) := B_{\bar g}(x,r)\times [s,s+r^2],
\]
to denote the backward and forward parabolic disk. Then, fixing any $(x_*,t_*,y_*,s_*)\in \MM_{A}$,  by the mean value inequality \cite[Theorem 25.2]{RFV3}, we have
\begin{align}\label{aC0estimateofCHK}
   &\sup_{P^-(x_*,t_*,r_0)\times P^+(y_*,s_*,r_0)} \bar K_i \\\nonumber
   \le & \ \sup_{(x,t)\in P^-(x_*,t_*,r_0)}
   C \int_{P^+(y_*,s_*,2r_0)} 
   \bar K_i(x,t\,|\,\cdot,\cdot)
   \, d\bar g ds\\\nonumber
  \le& \  
   \sup_{(x,t)\in P^-(x_*,t_*,r_0)}
   C \int_{s_*}^{s_*+4r_0^2}
   \int_{\Psi_i(B_{\bar g}(y_*,2r_0))}
   K_i(\Psi_i(x),t\,|\,y,s)
   \, d g_{i,s}(y) ds
   \\\nonumber
   \leq& \sup_{(x,t)\in P^-(x_*,t_*,r_0)}
   C \int_{s_*}^{s_*+4r_0^2}
   \int_{M_i}
   K_i(\Psi_i(x),t\,|\,y,s)
   \, d g_{i,s}(y) ds
   \\\nonumber
   \le&\  C,
\end{align}      
where $C$ depends on $c_0$ in (\ref{anotherlowervolumebounde}), $\Lambda_0$, $A$, and $C_0$ in (\ref{equivalenceofthemetric}). In particular, $C$ is independent of $i\geq \bar i$ and the point $(x_*,t_*,y_*,s_*)$. Here we have used the fact that $\int_{M_i}
   K_i(\Psi_i(x),t\,|\,y,s)
   \, d g_{i,s}(y)\equiv 1$.

Next, fixing an arbitrary $(y,s)\in P^+(y_*,s_*,r_0)$, we have, by the standard gradient estimate of the heat equation,
\begin{eqnarray*}
\sup_{(x,t)\in P^-(x_*,t_*,\tfrac{1}{2}r_0)}\big\|\,\partial^k_t\nabla^l_x \bar K_i(x,t\,|\, y,s)\,\big\|_{\bar g}\leq C_{k,l},
\end{eqnarray*}
for all $k, l\geq 0$. Here the constant $C_{k,l}$ depends on $C_0$ in (\ref{equivalenceofthemetric}), $C$ in (\ref{aC0estimateofCHK}),  $\Lambda_{k+l}$, and $A$. In particular, $C_{k,l}$ is independent of $i\geq \bar i$ and the choice of the point $(y,s)\in P^+(y_*,s_*,r_0)$. By using this estimate and the standard gradient estimate of the conjugate heat equation
\begin{eqnarray*}
\Box^*_{\bar g_{i,s},y,s}\Big(\partial^k_t\nabla^l_x \bar K_i(x,t\,|\, y,s)\Big)=0,
\end{eqnarray*}
we have
\begin{eqnarray*}
\sup_{(x,t,y,s)\in P^-(x_*,t_*,\tfrac{1}{2}r_0)\times P^+(y_*,s_*,\tfrac{1}{2}r_0)}\big\|\,\partial^m_s\nabla^n_y\partial^k_t\nabla^l_x \bar K_i(x,t\,|\, y,s)\,\big\|_{\bar g}\leq C_{k,l,m,n},
\end{eqnarray*}
where $C_{k,l,m,n}$ depends on $C_0$ in (\ref{equivalenceofthemetric}), $C_{k,l}$, $A$, and $\Lambda_{m+n}$. In particular, $C_{k,l,m,n}$ is independent of $i\geq \bar i$ and the choice of the point $(x_*,t_*,y_*,s_*)\in \MM_{A}$. Therefore, we have
\[
    \sup_{(x,t,y,s)\in\MM_{A}} \big\|\,\partial^m_s\nabla^n_y\partial^k_t\nabla^l_x \bar K_i(x,t\,|\, y,s)\,\big\|_{\bar g}\leq C_{k,l,m,n},
\]
for all $i\geq \bar{i}$ and for all $k,l,m,n\geq 0$. 

Finally, according to the Arzela-Ascoli theorem, by passing to a subsequence, we can find a smooth function $K_\infty:\MM\rightarrow [0,\infty)$, such that
\[
    \bar K_i\to K_\infty
\]
smoothly on compact subsets of $\MM$. As a consequence, we have
\begin{eqnarray*}
\Box_{g_{\infty,t},x,t}K_\infty(x,t\,|\,y,s)\equiv 0,\quad \Box_{g_{\infty,s},y,s}K_\infty(x,t\,|\,y,s)\equiv 0 \quad \text{ on }\quad \MM.
\end{eqnarray*}
The convergence of $\bar K_i$ also implies that, fixing any $(x,t)$ or $(y,s)\in M_\infty\times(-T_\infty,0]$, we have
\begin{align*}
\bar K_i(x,t\,|\,\cdot,\cdot)\rightarrow K_\infty(x,t\,|\,\cdot,\cdot)&\quad \text{ locally smoothly on }\quad M_\infty\times (-T_\infty,t),
\\
\bar K_i(\cdot,\cdot\,|\,y,s)\rightarrow K_\infty(\cdot,\cdot\,|\,y,s)&\quad \text{ locally smoothly on }\quad M_\infty\times (s,0].
\end{align*}
One may then easily apply the same arguments in \cite[Section 2.5]{Lu12} to show that
\[
    \lim_{s\to t^-} K_\infty(x,t\,|\,\cdot,s)
    = \delta_x,\quad
    \lim_{t\to s^+} K_\infty(\cdot,t\,|\,y,s)
    = \delta_y,
\]
in the sense of distributions. 
The strong maximum principle then implies that $K_\infty>0$ everywhere on $\MM.$
\end{proof}

\begin{Theorem}\label{thepropertiesoftheCHK}
For any $(x,t)\in M_\infty\times(-T_\infty,0]$, we have $K_\infty(x,t\,|\,\cdot,s)dg_{\infty,s}\in \PP( M_\infty)$ for all $s\in (-\infty,t]$. Furthermore, $K_\infty$ satisfies the reproduction formula, that is,
\begin{eqnarray*}
\int_SK_\infty(x,t\,|\, \cdot,s')dg_{\infty,s'}=\int_{M_\infty}K_\infty(x,t\,|\,z,s)\left(\int_SK_\infty(z,s\,|\,y,s')dg_{\infty,s'}(y)\right)dg_{\infty,s}(z).
\end{eqnarray*}
for all $x\in M$, $-T_\infty<s'<s<t\leq 0$, and for all measurable set $S\subset M_\infty$
\end{Theorem}
\begin{proof}
We will still use the same notation as in the proof of the former Theorem. Let us fix $(x,t)\in M_\infty\times(-T_\infty,0]$ and $s\in(-T_\infty,0)$. We will prove
\begin{eqnarray}\label{theintegralisequaltoone}
\int_{M_\infty}K_\infty(x,t\,|\,\cdot,s)\,dg_{\infty,s}=1.
\end{eqnarray}
By Fatou's lemma, we have
\begin{eqnarray}\label{theintegralislessthan1}
\int_{M_\infty}K_\infty(x,t\,|\,\cdot,s)\,dg_{\infty,s}&\leq&\liminf_{i\rightarrow\infty}\int_{U_i}\bar K_i(x,t\,|\,\cdot,s)\,d\bar g_{i,s}
\\\nonumber
&\leq& \liminf_{i\rightarrow\infty}\int_{\Psi_i(U_i)}K_i(\Psi_i(x),t\,|\,\cdot,s)\,dg_{i,s}
\\\nonumber
&\leq& \liminf_{i\rightarrow\infty}\int_{M_i}K_i(\Psi_i(x),t\,|\,\cdot,s)\,dg_{i,s}
\\\nonumber
&=&1.
\end{eqnarray}

To show the other direction of the inequality, we shall apply Proposition \ref{measureaccumulationofHcenter}. By the locally smooth convergence of $\bar K_i(x,t\,|\,\cdot,\cdot)$, we have
\begin{eqnarray*}
\liminf_{i\rightarrow\infty}\int_{B_{g_{i,s}}(\Psi_i(x),1)}K_i(\Psi_i(x),t\,|\,\cdot,s)\,dg_{i,s}&=&\liminf_{i\rightarrow\infty}\int_{B_{\bar g_{i,s}}(x,1)}\bar K_i(x,t\,|\,\cdot, s)\,d\bar g_{i,s}
\\
&=&\int_{B_{g_{\infty,s}}(x,1)}K_\infty(x,t\,|\,\cdot,s)\,dg_{\infty,s}
\\
&>&0.
\end{eqnarray*}
Hence, there is a positive constant $c_0$, such that
\begin{eqnarray*}
\int_{B_{g_{i,s}}(\Psi_i(x),1)}K_i(\Psi_i(x),t\,|\,\cdot,s)\,dg_{i,s}\geq c_0\quad\text{ for all }\quad i\in\mathbb{N}.
\end{eqnarray*}
Therefore, by Proposition \ref{measureaccumulationofHcenter}, we have
\begin{eqnarray}\label{theHncenterdoesnotgoaway}
\dist_{g_{i,s}}(z_i,\Psi_i(x))\leq\sqrt{\frac{2}{c_0}H(t-s)}+1\leq C \quad\text{ for all }\quad i\in\mathbb{N},
\end{eqnarray}
where $(z_i,s)$ is an $H_n$-center of $(\Psi_i(x),t)$ with respect to the Ricci flow $(M_i,g_i)$.

In view of $(\ref{theHncenterdoesnotgoaway})$, by passing to a subsequence, we may assume $z'_i:=\Psi_i^{-1}(z_i)\rightarrow z'_\infty\in M_\infty$. Hence, fixing any $\delta\in(0,1)$, we have
\begin{align}\label{theintegralisalmost1}
&\int_{B_{g_{\infty,s}}\left(z'_\infty, \sqrt{\delta^{-1}H_n(t-s)}\right)}K_\infty(x,t\,|\,\cdot,s)\,dg_{\infty,s}
\\\nonumber
=&\lim_{i\rightarrow\infty}\int_{B_{\bar g_{i,s}}\left(z'_i, \sqrt{\delta^{-1}H_n(t-s)}\right)}\bar K_i(x,t\,|\,\cdot,s)\,d\bar g_{i,s}
\\\nonumber
=&\lim_{i\rightarrow\infty}\int_{B_{ g_{i,s}}\left(z'_i, \sqrt{\delta^{-1}H_n(t-s)}\right)} K_i(\Psi_i(x),t\,|\,\cdot,s)\,d g_{i,s}
\\\nonumber
\geq&\  1-\delta.
\end{align}
Combining (\ref{theintegralislessthan1}) and (\ref{theintegralisalmost1}), we have proved (\ref{theintegralisequaltoone}).

Next, we prove the reproduction formula. Let us fix $x\in M_\infty$, $-T_\infty<s'<s<t\leq 0$, and measurable set $S\subset M_\infty$. Let us first of all consider the case when $S$ is bounded. Then we have the functions
\begin{eqnarray*}
u_i:z\rightarrow \int_S\bar K_i(z,s\,|\,y,s')\,d\bar g_{i,s'}(y)
\end{eqnarray*}
converge to the function
\begin{eqnarray*}
u_\infty:z\rightarrow \int_S K_\infty(z,s\,|\,y,s')\, dg_{\infty,s'}(y)
\end{eqnarray*}
pointwise and with the uniform upper bound $1$.

Finally, in view (\ref{theintegralisalmost1}), fixing any $\delta\in (0,1)$, we have, by the bounded convergence theorem,
\begin{align}\label{someimportantconvergenceargument}
&\int_{M_\infty}u_\infty K_\infty(x,t\,|\,\cdot,s)\,dg_{\infty,s}
\\\nonumber
\leq&\int_{B_{g_{\infty,s}}\left(z'_\infty, \sqrt{\delta^{-1}H_n(t-s)}\right)}u_\infty K_\infty(x,t\,|\,\cdot,s)\,dg_{\infty,s}+\delta
\\\nonumber
=&\lim_{i\rightarrow\infty}\int_{B_{\bar g_{i,s}}\left(z'_i, \sqrt{\delta^{-1}H_n(t-s)}\right)}u_i\bar K_i(x,t\,|\,\cdot,s)\,d\bar g_{i,s}+\delta
\\\nonumber
\leq &\lim_{i\rightarrow\infty}\int_{M_i}K_i(\Psi_i(x),t\,|\,\cdot,s)\left(\int_{\Psi_i(S)}K_i(\cdot,s\,|\,y,s')\,dg_{i,s'}(y)\right)\,d g_{i,s}+\delta
\\\nonumber
=&\lim_{i\rightarrow\infty}\int_{\Psi_i(S)}K_i(\Psi_i(x),t|\cdot,s')\,dg_{i,s'}+\delta
\\\nonumber
=&\lim_{i\rightarrow\infty}\int_S\bar K_i(x,t\,|\,\cdot,s')\,d\bar g_{i,s'}+\delta
\\\nonumber
=&\int_SK_\infty(x,t\,|\,\cdot,s')\,dg_{\infty,s'}+\delta,
\end{align}
where we have applied the reproduction property of $K_i$, which is but standard for Ricci flows with bounded curvature \cite[Lemmas 26.8 and 26.16]{RFV3}. Taking $\delta\rightarrow 0$, the case when $S$ is bounded is done. If $S$ is unbounded, then one can consider $S_i:=S\cap B_{g_{\infty,s'}}(x,i)$ and take $i\rightarrow\infty$, then the conclusion follows from the bounded convergence theorem.

\end{proof}

\bibliography{bibliography}{}
\bibliographystyle{amsalpha}

\newcommand{\etalchar}[1]{$^{#1}$}
\providecommand{\bysame}{\leavevmode\hbox to3em{\hrulefill}\thinspace}
\providecommand{\MR}{\relax\ifhmode\unskip\space\fi MR }
\providecommand{\MRhref}[2]{%
  \href{http://www.ams.org/mathscinet-getitem?mr=#1}{#2}
}
\providecommand{\href}[2]{#2}

\bigskip
\bigskip

\noindent Department of Mathematics, University of California, San Diego, CA, 92093
\\ E-mail address: \verb"pachan@ucsd.edu "
\\

\noindent Department of Mathematics, University of California, San Diego, CA, 92093
\\ E-mail address: \verb"zim022@ucsd.edu"
\\

\noindent Department of Mathematics, University of Minnesota, Twin Cities, MN, 55414
\\ E-mail address: \verb"zhan7298@umn.edu"

\end{document}